\documentclass[nosumlimits,twoside]{amsart}
\usepackage{amsfonts, amsmath, amssymb}
\usepackage[bookmarksnumbered,plainpages,driverfallback=dvipdfm,backref=page]{hyperref}
\usepackage{graphicx}
\usepackage{float}
\usepackage{srcltx}
\usepackage[all]{xy}
\usepackage{version}
\usepackage[final]{showlabels} 
\usepackage[T1]{fontenc}
\usepackage[shortlabels]{enumitem}
\usepackage[normalem]{ulem}
\usepackage{bm}
\usepackage{latexsym}
\usepackage[2emode]{psfrag}
\usepackage{yhmath}
\usepackage{array}
\usepackage{dsfont}
\usepackage{mathrsfs}
\usepackage{tikz-cd}
\usepackage{comment}
\usepackage{xcolor}
\usepackage{tabu}
\usepackage{makecell}
\usepackage{aligned-overset}
\usepackage{stmaryrd}
\usepackage{cite}

\newtheorem{theorem}{\sc Theorem}[section]
\newtheorem{proposition}[theorem]{\sc Proposition}
\newtheorem{notation}[theorem]{\sc Notation}
\newtheorem{lemma}[theorem]{\sc Lemma}
\newtheorem{corollary}[theorem]{\sc Corollary}
\theoremstyle{definition}
\newtheorem{definition}[theorem]{\sc Definition}
\newtheorem{example}[theorem]{\sc Example}
\newtheorem{remark}[theorem]{\sc Remark}

\allowdisplaybreaks
\excludeversion{invisible}

\newcommand{\A}{\mathcal{A}}

\renewcommand{\t}{\mathfrak{t}}
\newcommand{\T}{\mathfrak{T}}
\newcommand{\op}{\mathrm{op}}

\newcommand{\sgn}{\textup{sgn}}

\begin{document}

\title{On Separable and Frobenius Cowreaths of Type $(A \otimes H^{\op},H,\psi)$}

\author{Fabio Renda} 

\email{fabio.renda@outlook.it}

\date{\today}



\begin{abstract}
Cowreaths of type $(A \otimes H^{\op},H,\psi)$ have been investigated in \cite{BT,mt,mt2,mt3,FR3} as examples of (h-)separable and Frobenius coalgebras in monoidal categories. They are entwining structures built from a Hopf algebra $H$ and an $H$-comodule algebra $(A,\rho_A)$. In this article we develop a general theory that allows us to recover results from the aforementioned papers in an easier way and also to extend them to cowreaths in higher dimension. Focusing on separability, the crucial observation is that, when $H$ has bijective antipode, one should work with integrals on the coalgebra $(H,\psi)$ in place of Casimir morphisms, for they are easier to classify. On the other hand, in dealing with Frobenius properties, we obtain that a cowreath $(A \otimes H^{\op},H,\psi)$ is Frobenius if and only if the morphism $(\mathrm{Id}_A \otimes \mu)\rho_A$ is inner ($\mu$ is the distinguished grouplike element in $H^*$). While Frobenius cowreaths are always (h-)separable, examples of (h-)separable cowreaths that are not Frobenius will be presented at the end of this article.
\end{abstract}

\maketitle

\tableofcontents

\section{Introduction}

Cowreaths are generalized entwining structures that can be defined in the context of 2-categories \cite{BT-app} and that were first introduced by Brzezi\'nski in \cite{br} in order to develop Galois theory for coalgebras. When working in monoidal categories, cowreaths can be characterized in the following way. Consider an algebra $A$ in a monoidal category $\mathcal{M}$. One can build a new category $\mathcal{T}_A^{\#}$ whose objects are pairs $(X, \psi)$, where $X$ is an object in $\mathcal{M}$ and $\psi: X \otimes A \rightarrow A \otimes X$ is a morphism in $\mathcal{M}$ compatible with the algebra structure of $A$ in a suitable sense. The triple $(A, X, \psi)$ is called a cowreath in $\mathcal{M}$ if $(X, \psi)$ is a coalgebra in $\mathcal{T}_A^{\#}$. The widespread presence of cowreaths and their corepresentations (called entwined modules) in (quasi-)Hopf algebra theory has been illustrated in \cite{bc,bct2}: coring structures of the form $A \otimes X$ are in bijection with cowreaths $(A,X,\psi)$ and relative Hopf modules, Doi-Hopf modules and Yetter-Drinfeld modules are all examples of entwined modules over a cowreath. The study of Frobenius and separable properties for forgetful functors defined on categories of entwined modules over cowreaths that was conducted in \cite{bct,bct2}, and that has its roots in \cite{BT-app}, further led to the definitions of Frobenius and separable cowreath, which are the main objects of interest in this article.

When $H$ is a Hopf algebra over a field $\Bbbk$ and $(A,\rho_A)$ is a right $H$-comodule algebra, we find that $A \otimes H^{\op}$ is naturally a right $H \otimes H^{\op}$-comodule algebra and $H$ is a $H \otimes H^{\op}$-module coalgebra and we can build a cowreath $(A \otimes H^{\op}, H, \psi)$ in the category $\textrm{Vec}_{\Bbbk}$ of $\Bbbk$-vector spaces following the instructions of Proposition \ref{propXA}. This type of cowreaths were presented for the first time in \cite{BT} to give an example of a separable Frobenius Galois cowreath for which the associated subalgebra of coinvariants $A^{\textrm{co}X} \hookrightarrow A$ was not a separable extension. Examples of this form later appeared in \cite{mt,mt2}, where the study of h-coseparable coalgebras started in \cite{am} was extended to cowreaths, and in \cite{FR3}, whose aim was to generalize the instances presented in \cite{mt}. In each of these cases, the (h-)separability of the cowreath $(A \otimes H^{\op}, H, \psi)$ was to be understood as a list of equalities satisfied by a certain $\Bbbk$-linear map $B: H  \otimes H \rightarrow A \otimes H^{\op}$ called the Casimir morphism, but in dealing with concrete examples it became soon clear that this was not always the best approach, for the number of equalities to verify could grow pretty quickly. Suffice it to say that a great deal of effort was required in \cite{mt3} to obtain a partial classification of Casimir morphisms when $H$ and $A$ are just eight-dimensional algebras. 

A first simplification we propose to overcome these complications is to work with integrals on $(A \otimes H^{\op}, H, \psi)$ in place of Casimir morphisms. Integrals on cowreaths (also called integrals on the underlying coalgebra) were defined in \cite{BT2} and are a generalization of classic integrals for Hopf algebras. A remarkable connection between Casimir morphisms and integrals was observed for the first time in \cite[Prop. 3.5]{BT2}; in our case it boils down to the fact that given a Casimir morphism $B: H  \otimes H \rightarrow A \otimes H^{\op}$, the map $B(\-- \otimes 1_H): H \rightarrow A \otimes H^{\op}$ is an integral on $(A \otimes H^{\op}, H, \psi)$. The correspondence between Casimir morphisms and integrals can be made bijective, provided $H$ has an invertible antipode. Moreover, this bijection naturally restricts to one between normalized Casimir morphisms (i.e. those realizing separability) and total integrals, thus producing a Maschke type result for cowreaths of this form: a cowreath $(A \otimes H^{\op}, H, \psi)$ is separable if and only if it admits a total integral. The same happens when we move our attention to h-separability (a stronger notion introduced in \cite{am}), that is, normalized Casimir morphisms satisfying the h-separability condition are in bijective correspondence with total integrals verifying what we call the \emph{h-property}, therefore a cowreath $(A \otimes H^{\op}, H, \psi)$ is h-separable if and only if it admits a total integral with the h-property. Surprisingly enough, the hypothesis that $H$ has a bijective antipode is just what we require in order to produce a canonical total integral on $(A \otimes H^{\op}, H, \psi)$. Since the map $\underline{\T}:H \rightarrow A \otimes H^{\op}$ given by $\underline{\T}(h)=1_A \otimes S^{-1}(h)$ for every $h \in H$ is a total integral with the h-property, we find that every cowreath $(A \otimes H^{\op}, H, \psi)$ where $H$ has bijective antipode is automatically (h-)separable. 

Another important consequence of our Maschke type results is that they highlight the presence of a class of non-Frobenius Galois cowreaths for which the existence of a total integral is equivalent to separability, thus suggesting that the hypothesis of \cite[Thm. 5.4]{BT} could be relaxed. This becomes clear once we have proved that the Frobenius property for cowreaths $(A \otimes H^{\op}, H, \psi)$ is equivalent to the existence of an invertible element $\mathcal{A} \in A$ satisfying $a\mathcal{A}=\mu(a_1)\mathcal{A}a_0$ for every $a \in A$, where $\mu$ is the distinguished grouplike element of $H^*$ (this is an extension of \cite[Prop. 6.10]{BT}).  
In other words, $(A \otimes H^{\op}, H, \psi)$ is Frobenius if and only if the morphism $(\mathrm{Id}_A \otimes \mu)\rho_A$ is inner. In this case, such an element $\A$ must be unique up to a central invertible element \--- a consequence of the essential uniqueness of Frobenius pairs for (co)algebras in monoidal categories, which also agrees with the fact that $\mathrm{Inn}(A)\cong \frac{A^{\times}}{\mathcal{Z}(A^{\times})}$. If we consider an Hopf algebra $H$ with bijective antipode and an $H$-comodule algebra $A$ where no such element $\A$ can be found, then the cowreath $(A \otimes H^{\op}, H, \psi)$ is clearly separable, yet not Frobenius. Concrete examples will be presented in the last section of the paper.

On the other hand, Frobenius cowreaths of type $(A \otimes H^{\op}, H, \psi)$ are always (h-)separable, since the Frobenius property forces $H$ to be finite-dimensional (see \cite[Prop. 1.3]{bct2}) and thus its antipode must be invertible. Then we can build the canonical total integral $\underline{\T}$ and $(A \otimes H^{\op}, H, \psi)$ is evidently (h-)separable. We should however point out that, although $\underline{\T}$ arises quite naturally when dealing with separability, it is almost of no use if one tries to build a Frobenius pair with it. In fact the corresponding Casimir morphism $\underline{B}$ is part of a Frobenius pair only if $H \cong \Bbbk$.

As an application of the obtained theoretical results we provide a complete classification of (total) integrals on the cowreath $(A \otimes H^{\op},H,\psi)$, when $H=E(n)$, a $2^{n+1}$-dimensional pointed Hopf algebra that generalizes Sweedler's Hopf algebra, and $A=Cl(\alpha, \beta_i, \gamma_i, \lambda_{ij})$ a Clifford algebra which is naturally endowed with the structure of $E(n)$-comodule algebra \cite{PVO2}. Theorem \ref{thm:mainint} integrates the partial results obtained in \cite{mt2, mt3} and generalizes them to algebras in higher dimension. We will also show how to recover \cite[Thm. 5.7]{FR3} from it. 

\medskip

The article is organized as follows. In Section \ref{preliminaries} we recall some basics on (h-)separable and Frobenius coalgebras and cowreaths in monoidal categories, with particular attention to cowreaths of type $(A \otimes H^{\op},H,\psi)$. We present the definition of integral on $(A \otimes H^{\op},H,\psi)$ and how it can be conveniently rephrased if $H$ has bijective antipode.
Section \ref{Casvsint} explains how integrals on $(A \otimes H^{\op},H,\psi)$ can be used to describe (h-)separability when $H$ has bijective antipode and, in that case, we provide a canonical choice to show that each cowreath $(A \otimes H^{\op},H,\psi)$ is automatically (h-)separable. Theoretical results on Frobenius cowreaths are also included in this section, among which is a complete description of Frobenius pairs for a Frobenius cowreath $(A \otimes H^{\op},H,\psi)$. It shows that they are in bijective correspondence with invertible elements of $C_{A \otimes H^{\op}}(coH)$, the centralizer in $A \otimes H^{\op}$ of the subalgebra of coinvariants in the sense of \cite{BT2}. In Section \ref{sec:integrals} we recall the definitions of the Hopf algebra $H=E(n)$ and of the Clifford algebra $A=Cl(\alpha,\beta_i,\gamma_i, \lambda_{ij})$ and we obtain a complete classification of (total) integrals on $(A \otimes H^{\op},H,\psi)$. Although we could not provide a full list of integrals with the h-property when $n$ is unspecified, we were able to reduce the number of conditions that an integral has to check in order to satisfy the h-property. We have gathered in Section \ref{sec:classification} an explicit description of integrals realizing (h-)separability of the cowreath $(A \otimes H^{\op},H,\psi)$ when $n=1,2$ (to be compared with results in \cite{mt2, mt3}), in addition with a proof of \cite[Thm. 5.7]{FR3} that now follows from Theorem \ref{thm:mainint}. In the last section we present some interesting examples of Frobenius cowreaths and h-separable non-Frobenius cowreaths built on $H=E(n)$.

\medskip

\noindent\textit{Notations and conventions}. All vector spaces are understood to be over a fixed field $\Bbbk$ and by linear maps we mean $\Bbbk$-linear maps. $\Bbbk^{\times}$ is used to indicate the multiplicative group of the field $\Bbbk$, $A^{\times}$ the invertible elements of the algebra $A$. The unadorned tensor product $\otimes$ stands for $\otimes_{\Bbbk}$. We will write a tensor symbol before the exponent when denoting iterated tensor products: $X^{\otimes 2}=X \otimes X$, $X^{\otimes 3}=X \otimes X \otimes X$ and so on. All linear maps whose domain is a tensor product will usually be defined on generators and understood to be extended by linearity. 
Algebras over $\Bbbk$ will be associative and unital and coalgebras over $\Bbbk$ will be coassociative and counital. For an algebra, the multiplication and the unit are denoted by $m$ and $u$, respectively, while for a coalgebra, the comultiplication and the counit are denoted by $\Delta$ and $\varepsilon$, respectively. $S:H \rightarrow H$ denotes the antipode map of the Hopf algebra $H$. We use the classical Sweedler’s notation and we shall write $\Delta(h)= h_{1}\otimes h_{2}$ for any $h\in H$, where we omit the summation symbol. $H^{\op}$ denotes the opposite Hopf algebra, i.e. the Hopf algebra $H$ where the multiplication is replaced by $m^{\op}(h \otimes h'):=h'h$ for every $h,h' \in H$. For a (right) $H$-comodule $M$ we write its structure map as $\rho: M\to M\otimes H$, $\rho(m)=m_0\otimes m_1$ for every $m \in M$. Categories of (bi)modules will be indicated as usual with subscripts $\mathcal{M}_A$, ${_A}\mathcal{M}$, ${_A}\mathcal{M}_A$ and categories of (bi)comodules with superscripts $\mathcal{M}^C$, ${^C}\mathcal{M}$, ${^C}\mathcal{M}^C$. We use $\sqcup$ to denote the disjoint union of two sets, we write $|Q|$ for the cardinality of the set $Q$ and $Q^c$ to indicate its complement in $\textbf{n}:=\lbrace 1, 2, \ldots, n \rbrace$. The writing $n \equiv_2 m$ means that the natural numbers $n$ and $m$ share the same parity.

\medskip

\noindent\textbf{Important}. We find it very confusing to omit the multiplication symbol when working in $A \otimes H^{\op}$, because years of habit trick us into associating $(a \otimes b)(c \otimes d)$ to $(ac \otimes bd)$ and not to $(ac \otimes db)$. To avoid disastrous results we preferred to indicate $m_{A \otimes H^{\op}}$ with $\bullet$ when really needed, and write everything else as if we were working in $A \otimes H$, performing preventive suitable swaps. For example, equality \eqref{eq:Casimir1} could have been written
\[
(a_{0} \otimes g_1) \bullet B(g_{2}ha_{1}\otimes g_{3}h^{\prime }a_{2})=B(h\otimes h^{\prime }) \bullet (a \otimes g),
\]
but we decided to write it so that omitted multiplications on each side of the tensor can be understood as the usual ones. 

\section{Preliminaries}\label{preliminaries}

We start by recalling some preliminary notions on separable and Frobenius (co)algebras in monoidal categories. We will then introduce cowreaths and present in detail what it means for a cowreath to be (h-)separable or Frobenius, further specializing this definitions to the case of cowreaths of type $(A \otimes H^{\op},H,\psi)$. Next we will present the definition of (total) integral on a cowreath $(A \otimes H^{\op},H,\psi)$ and show that when $H$ has invertible antipode, there is an equivalent characterization which will be helpful with later calculations. A list of basics in classic integral theory for Hopf algebras concludes this section. All monoidal categories involved are assumed to be strict, thanks to the well-known Mac Lane's coherence theorem.

\subsection{(Co)separable (co)algebras}\label{sepcosep}

The concept of separable algebra is a generalization of that of separable field extension and as such was first defined in the context of vector spaces \cite{AG}. It is usually said that a $k$-algebra $A$ is separable if the multiplication map $m_A: A \otimes A \rightarrow A$ admits a section $\gamma: A \rightarrow A \otimes A$ which is also a morphism of $A$-$A$-bimodules, and this is easily generalized to the case of algebras in a monoidal category \cite{nbo}. Given some additional property on the unit object $\underline{1}$ of the category, it is possible to derive a series of equivalent conditions.

\begin{definition}{\cite[Def. 3.1]{BT}}\label{lgen}
An object $P$ of a monoidal category $(\mathcal{M}, \otimes, \underline{1})$ is a left $\otimes$-generator of $\mathcal{M}$ if, given morphisms $f, g : Y \otimes Z \rightarrow W$ in $\mathcal{M}$ such that $f (\lambda \otimes \mathrm{Id}_Z) = g(\lambda \otimes \mathrm{Id}_Z)$ for all $\lambda : P \rightarrow Y$ in $\mathcal{M}$, we then have $f = g$.
\end{definition}

\begin{proposition}{\cite[Prop. 4.3]{BT}}\label{propsep}
Let $(\mathcal{M}, \otimes, \underline{1})$ be a monoidal category such that $\underline{1}$ is a left $\otimes$-generator for $\mathcal{M}$. Then the following are equivalent
\begin{enumerate}
\item $A$ is separable.
\item The forgetful functor $F:\mathcal{M}_A \rightarrow \mathcal{M}$ is separable.
\item There exists a morphism $e: \underline{1} \rightarrow A \otimes A$ in $\mathcal{M}$, such that
\[(m_A \otimes \mathrm{Id}_A)(\mathrm{Id}_A \otimes e)=(\mathrm{Id}_A \otimes m_A)(e \otimes \mathrm{Id}_A) \quad \textrm{and} \quad m_A e =u_A.\]
Such $e$ is called the separability morphism of $A$.
\end{enumerate}
\end{proposition}
If we want to dualize Proposition \ref{propsep} to the case of coalgebras, the definition of coseparable coalgebra needs to be presented. It was introduced by Larson in \cite{Lar}.

\begin{definition}{\cite[p. 262]{Lar}}
A coalgebra $C$ in a monoidal category $(\mathcal{M}, \otimes, \underline{1})$ is coseparable if it is a relative injective $C$-bicomodule in $C$, which comes down to the following property. If $q : M \rightarrow N$ in ${^C}\mathcal{M}{^C}$ has a left inverse
$p : M \rightarrow N$ in $\mathcal{M}$, then every $f : M \rightarrow C$ in ${^C}\mathcal{M}{^C}$ factors through $q$ in ${^C}\mathcal{M}{^C}$: there exists a $C$-bicolinear morphism $g : N \rightarrow C$ such that $g q = f$.
\end{definition}
\begin{proposition}{\cite[Prop. 7.3]{bct}}\label{eqcosep}
Let $(\mathcal{M}, \otimes, \underline{1})$ be a monoidal category. Then the following are equivalent
\begin{enumerate}
\item $C$ is coseparable.
\item There exists a morphism $\gamma: C \otimes C \rightarrow C$ of bicomodules such that $\gamma \Delta_C= \mathrm{Id}_C$.
\item There exists a morphism $B: C \otimes C \rightarrow \underline{1}$ in $\mathcal{M}$, such that
\begin{equation*}
(\mathrm{Id}_C \otimes B)(\Delta_C \otimes \mathrm{Id}_C)=(B \otimes \mathrm{Id}_C)(\mathrm{Id}_C \otimes \Delta_C) \quad \mathrm{and} \quad B \Delta_C =\varepsilon_C.
\end{equation*}
\end{enumerate}
A morphism $B: C \otimes C \rightarrow \underline{1}$ satisfying the first of the two conditions above is called a Casimir morphism for $C$. A Casimir morphism satisfying the latter is said to be normalized. 
\end{proposition}

Notice that Proposition \ref{eqcosep} is not completely dual to Proposition \ref{propsep}, in that it lacks an interpretation of coseparability of $C$ in terms of a functorial property and the fact that $\underline{1}$ be a left $\otimes$-generator is not required. The picture becomes complete once we take into account a fundamental result contained in \cite{mt}.

\begin{theorem}
\label{Theorem:B} \cite[Thm. 2.5]{mt} Assume that $\underline{1}$ is a left $\otimes $-generator for a monoidal category $(\mathcal{M}, \otimes, \underline{1})$, and let $C$ be a coalgebra in $\mathcal{M}$. The forgetful functor $F:\mathcal{M}^C \rightarrow \mathcal{M}$ is separable if and only if the coalgebra $C$ is coseparable.
\end{theorem}

\medskip

Starting from an example related to the tensor algebra, A. Ardizzoni and C. Menini introduced in \cite{am} a stronger version of the notion of separability for functors, called heavy separability (for short, $h$-separability). This led to the discovery of a Rafael-type theorem for h-separable functors \cite[Thm. $2.1$]{am} and to the definition of h-coseparable coalgebra. It was initially proved in \cite{am} that the only h-coseparable coalgebras $C$ over a field $\Bbbk$ are $C=\Bbbk$ and $C=\lbrace 0 \rbrace$, while non-trivial examples of h-coseparable coalgebras were later exhibited in \cite{mt} as instance of cowreaths of type $(A \otimes H^{\op},H,\psi)$.

\begin{theorem}\label{Chcosep}\cite[Thm. 2.8 and Def. 2.2]{mt}
Assume that $\underline{1}$ is a left $\otimes $-generator for a monoidal category $(\mathcal{M}, \otimes, \underline{1})$, and let $C$ be a coalgebra in $\mathcal{M}$. The forgetful functor $F:\mathcal{M}^C \rightarrow \mathcal{M}$ is heavily separable if and only if there exists a normalized Casimir morphism $B: C \otimes C \rightarrow \underline{1}$ which fulfills \[(B \otimes B)(\mathrm{Id}_C \otimes \Delta_C \otimes \mathrm{Id}_C) = B(\mathrm{Id}_C \otimes \varepsilon_C \otimes \mathrm{Id}_C) \quad \textrm{(h-coseparability condition)}.\]
In this case $C$ is called an h-coseparable coalgebra.

\end{theorem}

\subsection{Frobenius (co)algebras and Frobenius pairs}\label{subsec:Frob}

An extensive treatment of Frobenius algebras in monoidal categories is found in \cite{BT-Frob}, to which we owe the first part of what follows. Let us consider a monoidal category $(\mathcal{M},\otimes, \underline{1})$ and an algebra $A$ in $\mathcal{M}$. The algebra $A$ is called Frobenius if there exist morphisms $\vartheta:  A \rightarrow \underline{1}$ and $e: \underline{1} \rightarrow A \otimes A$ in $\mathcal{M}$ such that
\begin{align*}
&(m_A \otimes \mathrm{Id}_A)(\mathrm{Id}_A \otimes e)=(\mathrm{Id}_A \otimes m_A)(e \otimes \mathrm{Id}_A),\\
&(\vartheta \otimes \mathrm{Id}_A)e=u_A=(\mathrm{Id}_A \otimes \vartheta)e.
\end{align*}
The pair $(\vartheta, e)$ is usually called a Frobenius pair (or Frobenius system) for $A$; $\vartheta$, $e$ are called Frobenius, respectively Casimir, morphism for $A$.
A standard result for Frobenius algebras is that Frobenius pairs are essentially unique. To see this, let $A_0:=\mathrm{Hom}_{\mathcal{M}}(\underline{1},A)$ with the monoid structure given by $m_{A_0}(\tilde{a} \otimes \tilde{b}):=m_A(\tilde{a} \otimes \tilde{b})$ for every $\tilde{a}, \tilde{b} \in A_0$. The invertible elements of $A_0$ form a group denoted by $A_0^{\times}$ and we write $\tilde{a}^{-1}$ for the inverse of an element $\tilde{a} \in A_0^{\times}$.

\begin{proposition}\cite[Prop. 2]{BT-Frob}\label{prop:FSalg}
Let $(\phi,e)$, $(\psi,f)$ be two Frobenius pairs for a Frobenius algebra $A$ in the monoidal category $(\mathcal{M}, \otimes, \underline{1})$. Then $\tilde{a}:=(\psi \otimes \mathrm{Id}_A)e: \underline{1} \rightarrow A$ is invertible in $A_0$ with inverse given by $\tilde{a}^{-1}:=(\phi \otimes \mathrm{Id}_A)f: \underline{1} \rightarrow A$ and $f=(\mathrm{Id}_A \otimes m_A(\tilde{a}^{-1} \otimes \mathrm{Id}_A))e$, $\psi= \phi m_A(\tilde{a} \otimes \mathrm{Id}_A)$.
\end{proposition}
It easily follows that there is a bijective correspondence between the set of Frobenius pairs for $A$ and the set $A_0^{\times}$. Explicitly, once we fix a Frobenius pair $(\vartheta,e)$ for $A$, we have
\[
FS_{\mathcal{M}}(A)=\left\lbrace ( \vartheta m_A(\tilde{a} \otimes \mathrm{Id}_A), (\mathrm{Id}_A \otimes m_A(\tilde{a}^{-1} \otimes \mathrm{Id}_A))e) \ | \ \tilde{a} \in A_0^{\times} \right\rbrace,
\]
where $FS_{\mathcal{M}}(A)$ denotes the set of all Frobenius pairs for $A$ in $\mathcal{M}$.

\bigskip

If we consider algebras in the opposite category, we get the dual notion of Frobenius coalgebra. Let $C$ be a coalgebra in a monoidal category $(\mathcal{M},\otimes, \underline{1})$. The coalgebra $C$ is called Frobenius if there exist morphisms $\t:  \underline{1} \rightarrow C$ and $B: C \otimes C \rightarrow \underline{1}$ in $\mathcal{M}$ such that
\begin{align*}
&(B \otimes \mathrm{Id}_C)(\mathrm{Id}_C \otimes \Delta_C)=(\mathrm{Id}_C \otimes B)(\Delta_C \otimes \mathrm{Id}_C),\\
&B(\t \otimes \mathrm{Id}_C)=\varepsilon_C=B(\mathrm{Id}_C \otimes \t).
\end{align*}
Again, the pair $(B, \t)$ is called a Frobenius pair (or system) for $C$; $\t$, $B$ are called Frobenius, respectively Casimir, morphism for $C$. Let $C_0:=\mathrm{Hom}_{\mathcal{M}}(C,\underline{1})$ with the monoid structure given by $m_{C_0}(\tilde{c} \otimes \tilde{d}):=(\tilde{c} \otimes \tilde{d})\Delta$ for every $\tilde{c}, \tilde{d} \in C_0$. The invertible elements of $C_0$ form a group denoted by $C_0^{\times}$ and we write $\tilde{c}^{-1}$ for the inverse of an element $\tilde{c} \in C_0^{\times}$.
The analogous of Proposition \ref{prop:FSalg} is
\begin{proposition}\label{prop:FScoalg}
Let $(B,\t)$, $(B',\t')$ be two Frobenius pairs for a Frobenius coalgebra $C$ in the monoidal category $(\mathcal{M}, \otimes, \underline{1})$. Then $\tilde{c}:=B(\t' \otimes \mathrm{Id}_C): C \rightarrow \underline{1}$ is invertible in $C_0$ with inverse given by $\tilde{c}^{-1}:=B'(\t \otimes \mathrm{Id}_C): C \rightarrow \underline{1}$ and
$B'=B(\mathrm{Id}_C \otimes (\tilde{c}^{-1} \otimes \mathrm{Id}_C)\Delta_C)$, $\t'= (\tilde{c} \otimes \mathrm{Id}_C)\Delta_C \t$.
\end{proposition}
Once we fix a Frobenius pair $(B,\t)$ for $C$, we have that
\[
FS_{\mathcal{M}}(C)=\left\lbrace (B(\mathrm{Id}_C \otimes (\tilde{c}^{-1} \otimes \mathrm{Id}_C)\Delta_C), (\tilde{c} \otimes \mathrm{Id}_C)\Delta_C \t) \ | \ \tilde{c} \in C_0^{\times} \right\rbrace,
\]
where $FS_{\mathcal{M}}(C)$ denotes the set of all Frobenius pairs for $C$ in $\mathcal{M}$.

\subsection{Cowreaths in monoidal categories}\label{cowreathsec}
Wreaths and cowreaths are generalized entwining structures that can be introduced in the context of $2$-categories. The definition of wreath was given in \cite[p. 256]{ls}, while the dual notion of cowreath appeared in \cite[p. 1056]{bc}. 
In \cite{bc} the authors show that when $\mathcal{M}$ is a monoidal category, a cowreath in $\mathcal{M}$ is equivalently defined by the choice of an algebra $A$ in $\mathcal{M}$, an object $X$ in $\mathcal{M}$, and an entwining map $\psi: X \otimes A \rightarrow A \otimes X$ satisfying appropriate relations. Among these relations is the fact that $(X, \psi)$ is a coalgebra in a suitable monoidal category associated to the algebra $A$. Further details about (co)wreaths in a $2$-category, their (co)representations and motivations for their investigation can be found in \cite{bc,bct,bct2,BT2, BT-app}.

\begin{definition}\cite[Subsec. 3.1]{bct}\label{transf}
Let $(\mathcal{M}, \otimes, \underline{1})$ be a monoidal category and let $A$ be an algebra in $\mathcal{M}$. A (right) transfer morphism through $A$ is a pair $(X, \psi)$ with $X \in \mathcal{M}$ and $\psi : X \otimes A \rightarrow A \otimes X$ in $\mathcal{M}$
such that
\begin{align*}
\psi (\mathrm{Id}_X \otimes m_A) &= (m_A \otimes \mathrm{Id}_X) (\mathrm{Id}_A \otimes \psi) (\psi \otimes \mathrm{Id}_A),\\
\psi (\mathrm{Id}_X \otimes u_A) &= u_A \otimes \mathrm{Id}_X.
\end{align*}
The category of all right transfer morphism through $A$ is denoted $\mathcal{T}_A^{\#}$ and a morphism $f:X \rightarrow Y$ for this category is a morphism $f: X \rightarrow A \otimes Y$ in $\mathcal{M}$ such that
\begin{equation*}\label{morphT}
(m_A \otimes \mathrm{Id}_Y) (\mathrm{Id}_A \otimes f) \psi_X =(m_A \otimes \mathrm{Id}_Y) (\mathrm{Id}_A \otimes \psi_Y) (f \otimes \mathrm{Id}_A).
\end{equation*}
The composition of two morphisms $f : X \rightarrow Y$ and $g : Y \rightarrow Z$ in $\mathcal{T}_A^{\#}$ is
\begin{equation*}\label{compT}
g \odot f = (m_A \otimes \mathrm{Id}_Z) (\mathrm{Id}_A \otimes g) f
\end{equation*}
and we have
\[\mathrm{Id}_{(X,\psi)} = u_A \otimes \mathrm{Id}_X.\]
The tensor product of $(X, \psi_X)$ and $(Y, \psi_Y )$ is
\begin{equation*}\label{tensorpsi}
X \circledast Y = (X \otimes Y, \ \psi_{X \otimes Y} = (\psi_X \otimes \mathrm{Id}_Y )(\mathrm{Id}_X \otimes \psi_Y )).
\end{equation*} 
The tensor product of morphisms $f : X \rightarrow X'$ and $g : Y \rightarrow Y'$ is
\begin{equation*}\label{tensorfg}
f \circledast g = (m_A \otimes \mathrm{Id}_{X' \otimes Y'}) (\mathrm{Id}_A \otimes \psi_{X'} \otimes \mathrm{Id}_{Y'}) (f \otimes g).
\end{equation*}
The unit object of $\mathcal{T}_A^{\#}$ is
\[(\underline{1}, r_A^{-1} l_A : \underline{1} \otimes A \rightarrow A \otimes \underline{1}).\]
\end{definition}

\begin{definition}\cite[Subsec. 3.2]{bct}\label{defcowreath}
A cowreath in $\mathcal{M}$ is a triple $(A, X, \psi)$ where $A$ is an algebra in $\mathcal{M}$ and $(X, \psi)$ is a coalgebra in $\mathcal{T}_A^{\#}$. 
This means that $(X, \psi) \in \mathcal{T}_A^{\#}$ and there are morphisms
$\delta : X \rightarrow A \otimes X \otimes X$ and $\epsilon : X \rightarrow A$ in $\mathcal{M}$ such that 
\begin{equation*}\label{deltamor}
(m_A \otimes \mathrm{Id}_{X^{\otimes 2}})(\mathrm{Id}_A \otimes \psi \otimes \mathrm{Id}_X)(\mathrm{Id}_{A \otimes X} \otimes \psi)(\delta \otimes \mathrm{Id}_A)= (m_A \otimes \mathrm{Id}_{X^{\otimes 2}})(\mathrm{Id}_A \otimes \delta)\psi \\
\end{equation*}
($\delta$ is a morphism in $\mathcal{T}_A^{\#}$),
\begin{equation*}\label{cowcoass}
(m_A \otimes \mathrm{Id}_{X^{\otimes 3}})(\mathrm{Id}_A \otimes \delta \otimes \mathrm{Id}_X)\delta = (m_A \otimes \mathrm{Id}_{X^{\otimes 3}})(\mathrm{Id}_A \otimes \psi \otimes \mathrm{Id}_{X^{\otimes 2}})(\mathrm{Id}_{A \otimes _X} \otimes \delta)\delta
\end{equation*}
(coassociativity),
\begin{equation*}\label{epsmor} 
m_A (\mathrm{Id}_A \otimes \epsilon) \psi = m_A (\epsilon \otimes \mathrm{Id}_A) 
\end{equation*}
($\epsilon$ is a morphism in $\mathcal{T}_A^{\#}$),
\begin{equation*}\label{leftcou}
(m_A \otimes \mathrm{Id}_X) (\mathrm{Id}_A \otimes \epsilon \otimes \mathrm{Id}_X) \delta = u_A \otimes \mathrm{Id}_X 
\end{equation*}
(left counit property),
\begin{equation*}\label{rightcou}
(m_A \otimes \mathrm{Id}_X)(\mathrm{Id}_A \otimes \psi)(\mathrm{Id}_{A \otimes X} \otimes \epsilon) \delta = u_A \otimes \mathrm{Id}_X 
\end{equation*}
 (right counit property).
\end{definition}

\subsection{(h-)Separable and Frobenius cowreaths in monoidal categories}

Given that cowreaths in $\mathcal{M}$ can be regarded as coalgebras in the category $\mathcal{T}^{\#}_A$, the notions of (h-)separable cowreath and Frobenius cowreath can be introduced and examples of such structures can be presented, see \cite{bct2,BT,mt,mt2,FR3}.
\begin{definition}\cite[Def. 2.5]{BT}\label{def:cowprop}
Let $(A, X, \psi)$ be a cowreath in a monoidal category $(\mathcal{M}, \otimes, \underline{1})$. We say that
\begin{itemize}
\item A morphism $B: X \otimes X \rightarrow A$ in $\mathcal{M}$ is a Casimir morphism for the coalgebra $(X, \psi) \in \mathcal{T}^{\#}_A$ if $B$ satisfies
\begin{equation*}
m_A(B \otimes \mathrm{Id}_A) = m_A(\mathrm{Id}_A \otimes B)(\psi \otimes \mathrm{Id}_X)(\mathrm{Id}_X \otimes \psi)
\end{equation*}
($B$ is a morphism in $\mathcal{T}_A^{\#}$),
\begin{align*}
&(m_A \otimes \mathrm{Id}_X)(\mathrm{Id}_A \otimes \psi)(\mathrm{Id}_{A \otimes X} \otimes B)(\delta \otimes \mathrm{Id}_X)\\
=&(m_A \otimes \mathrm{Id}_X)(\mathrm{Id}_A  \otimes B \otimes \mathrm{Id}_X)(\psi \otimes \mathrm{Id}_{X^{\otimes 2}})(\mathrm{Id}_X \otimes \delta),
\end{align*}
($B$ satisfies the Casimir condition).

\item The cowreath $(A, X, \psi)$ is called separable if there exists a Casimir morphism $B: X \otimes X \rightarrow A$ for $(X, \psi)$ that also satisfies
\begin{equation*}
m_A(\mathrm{Id}_A \otimes B)\delta =\epsilon,
\end{equation*}
($B$ is normalized).

\item The cowreath $(A, X, \psi)$ is called h-separable if there exists a normalized Casimir morphism $B: X \otimes X \rightarrow A$ such that
\begin{align*}
&m_A[m_A(\mathrm{Id}_A \otimes B) \otimes B](\psi \otimes \mathrm{Id}_{X^{ \otimes  3}})(\mathrm{Id}_X \otimes \delta \otimes\mathrm{Id}_X)\\
=&m_A(\mathrm{Id}_A  \otimes B)( \psi \otimes \mathrm{Id}_X)(\mathrm{Id}_X \otimes \epsilon \otimes \mathrm{Id}_X),
\end{align*}
(h-coseparability condition).

\item The cowreath $(A, X, \psi)$ is called Frobenius if $(X,\psi)$ is a Frobenius coalgebra in $\mathcal{T}^{\#}_A$. This means that there exist a Casimir morphism $B: X \otimes X \rightarrow A$ for $(X, \psi)$ and a morphism $\t: \underline{1} \rightarrow A  \otimes X$ in $\mathcal{M}$ such that
\begin{equation*}
(m_A \otimes \mathrm{Id}_X)(\mathrm{Id}_A \otimes \t)=(m_A \otimes \mathrm{Id}_X )(\mathrm{Id}_A  \otimes \psi)(\t \otimes \mathrm{Id}_A),
\end{equation*}
($\t$ is a morphism in $\mathcal{T}_A^{\#}$),
\begin{equation*}
m_A(\mathrm{Id}_A \otimes B)(\t \otimes \mathrm{Id}_X) =\epsilon= m_A(\mathrm{Id}_A \otimes B)(\psi \otimes \mathrm{Id}_X)(\mathrm{Id}_X \otimes \t).
\end{equation*}
The second condition assures that $\t$ is a Frobenius morphism, i.e. that $(B,\t)$ is a Frobenius pair for the coalgebra $(X, \psi)$.
\end{itemize}
\end{definition}

\subsection{Cowreaths of type \texorpdfstring{$(A \otimes H^{\op},H,\psi)$}{}}

From now on we will consider $\mathcal{M}=\mathrm{Vec}_{\Bbbk}$ the monoidal category of vector spaces over a fixed field $\Bbbk$. Let $A$ be an algebra in $\mathrm{Vec}_{\Bbbk}$ and let $H$ be a Hopf algebra in $\mathrm{Vec}_{\Bbbk}$. Given a $\Bbbk$-linear map $\rho_A: A \rightarrow A \otimes H$, the pair $(A, \rho_A)$ is called a (right) $H$-comodule algebra if the following conditions are satisfied.
\begin{align*}
(A, \rho_A) &  \textrm{ is a (right)} \ H\textrm{-comodule,}\\
\rho_A (ab) &= a_0b_0 \otimes a_1b_1,\\
\rho_A (1_A)&=1_A \otimes 1_H,
\end{align*}
for all $a,b \in A$.
The map $\rho_A$ is also called a (right) $H$-coaction on $A$. The following proposition shows how to build a cowreath in $\mathrm{Vec}_{\Bbbk}$ from a Hopf algebra $H$ and an $H$-comodule algebra $(A,\rho_A)$.
\begin{proposition}\cite[Sec. 5]{mt}\label{propXA}
Let $H$ be a Hopf algebra and let $\left(A,\rho_A \right)$ be a right $H$-comodule algebra.
Let $\psi: H \otimes A \otimes H^{\op} \rightarrow A \otimes H^{\op} \otimes H$ be defined by
\begin{equation*}\label{psiHAHop}
\psi(h \otimes a \otimes l)=a_0 \otimes l_1 \otimes l_2 h a_1
\end{equation*}
for every $h \in H$ and every $a \otimes l \in A \otimes H^{\op}$.
Then $(H, \psi) \in \mathcal{T}_{A\otimes H^{\op}}^{\#}$ and $(H, \psi)$ is a coalgebra in $\mathcal{T}_{A \otimes H^{\op}}^{\#}$, with comultiplication $\delta_H:  H \rightarrow A \otimes H^{\op} \otimes H \otimes H$ and counit $\epsilon_H: H \rightarrow A \otimes H^{\op}$ given by
\begin{eqnarray*}
\delta_H(h)=1_A \otimes 1_H\otimes h_1 \otimes h_2, \label{deltaH}\\
\epsilon_H(h)=\varepsilon_H(h)1_A \otimes 1_H, \label{epsH}
\end{eqnarray*}
for every $h \in H$.
\end{proposition}
At this point, by explicitly writing the conditions in Definition \ref{def:cowprop}, we can describe when a cowreath of this type is (h-)separable or Frobenius (cf. \cite{BT,mt}). 

\begin{enumerate}[\textbf{\roman*})]
\item A Casimir morphism for the coalgebra $(H,\psi)$ is a $\Bbbk$-linear map $B:H\otimes H\rightarrow A\otimes H^{\op}$ given by
\begin{equation*}
h\otimes h^{\prime }\mapsto B(h \otimes h')= B^{A}\left( h\otimes h^{\prime }\right)
\otimes B^{H}\left( h\otimes h^{\prime }\right)
\end{equation*}%
and such that
\begin{align}
(a_{0} \otimes 1_H)B(g_{2}ha_{1}\otimes g_{3}h^{\prime }a_{2})(1 \otimes g_{1})=(1_A \otimes g)B(h\otimes h^{\prime })(a \otimes 1_H),  \label{eq:Casimir1}\\
B^{A}(h_{2}\otimes h^{\prime })_{0}\otimes B^{H}(h_{2}\otimes h^{\prime
})_{1}\otimes B^{H}(h_{2}\otimes h^{\prime })_{2}h_{1}B^{A}(h_{2}\otimes
h^{\prime })_{1}=B(h\otimes h_{1}^{\prime })\otimes h_{2}^{\prime },  \label{eq:Casimir2}
\end{align}
for every $a \in A$ and every $g,h,h' \in H$.

\item The cowreath $(A \otimes H^{\op}, H, \psi)$ is separable if and only if there is a Casimir morphism $B:H\otimes H\rightarrow A\otimes H^{\op}$ for $(H,\psi)$ such that
\begin{equation}\label{eq:sepcas}
B(h_1 \otimes h_2)=\varepsilon(h)1_A \otimes 1_H,
\end{equation}
for every $h \in H$.

\item The cowreath $(A \otimes H^{\op}, H, \psi)$ is h-separable if and only if there is a normalized Casimir morphism $B:H\otimes H\rightarrow A\otimes H^{\op}$ for $(H,\psi)$ such that
\begin{equation}\label{eq:hsepcas}
B^{A}\left( h\otimes h_{1}^{\prime }\right) B^{A}\left( h_{2}^{\prime
}\otimes h^{\prime \prime }\right) \otimes B^{H} \left( h_{2}^{\prime }\otimes h^{\prime \prime}\right) B^{H}\left( h\otimes h_{1}^{\prime }\right) =\varepsilon_H \left( h^{\prime }\right) B\left( h\otimes h^{\prime \prime }\right),
\end{equation}
for every $h,h',h'' \in H$.

\item The cowreath $(A \otimes H^{\op}, H, \psi)$ is Frobenius if and only if there is a Casimir morphism $B:H\otimes H\rightarrow A\otimes H^{\op}$ for $(H,\psi)$ and an element $\t=\t^1 \otimes \t^2 \otimes \t^3 \in A \otimes H \otimes H$ such that
\begin{align}
a\t^1 \otimes \t^2h \otimes \t^3=\t^1 a_0 \otimes h_1\t^2 \otimes h_2 \t^3 a_1, \label{eq:Frobel1}\\
(\t^1 \otimes 1_H)B(\t^3 \otimes h)(1_A \otimes \t^2)=\varepsilon(h)1_A \otimes 1_H=(\t^1_0 \otimes 1_H)B(\t^2_2h\t^1_1 \otimes \t^3)(1_A \otimes \t^2_1),\label{eq:Frobel2}
\end{align}
for every $a \in A$ and every $h \in H$. In this case $\t$ is called a Frobenius element for $(A \otimes H^{\op}, H, \psi)$. 
\end{enumerate}

\subsection{Integrals on cowreaths of type \texorpdfstring{$(A \otimes H^{\op},H,\psi)$}{}}

Cowreaths of the form $(A \otimes H^{\op},H,\psi)$ defined in the above proposition, are Galois cowreath, see \cite[Prop. 12.3]{BT3}. (Pre-)Galois cowreaths in monoidal categories were introduced in \cite{BT3} and further investigated in \cite{BT}. An integral theory for this kind of cowreaths was developed in \cite{BT2, BT} and it will prove very useful in the sequel. In fact, we will show that the (h-)separability and the Frobenius property of cowreaths of type $(A \otimes H^{\op},H,\psi)$ can be described in terms of integrals instead of Casimir morphisms, if we assume that $H$ has bijective antipode. We present the definition of integral on cowreaths of our interest, obtained from the general one given in \cite{BT}.

\begin{definition}\cite[Def. 2.2]{BT}
An integral on $(A \otimes H^{\op}, H,\psi)$ \--- also called an integral on the coalgebra $(H,\psi)$ \--- is a $\Bbbk$-linear map $\T: H \rightarrow A \otimes H^{\op}$ given by $\T(h)=\T^A(h) \otimes \T^H(h)$ and such that
\begin{align}
\T^A(h_2)_0 \otimes \T^H(h_2)_1 \otimes \T^H(h_2)_2 h_1\T^A(h_2)_1=\T(h) \otimes 1_H,\label{eq:weakint}\\
(a_0 \otimes 1_H)\T(S^{-1}(a_2)ha_1)(1_A \otimes S^{-1}(a_3))=(1_A \otimes S^{-1}(a_1))\T(h)(a_0 \otimes 1_H),\label{eq:int}
\end{align}
for every $a \in A$ and every $h \in H$. $\T$ is called a total integral if in addition
\begin{equation}\label{eq:totint}
\T(1_H)=1_A \otimes 1_H.
\end{equation}
\end{definition}

\begin{proposition}\label{prop:weakint2}
If $H$ has bijective antipode, conditions \eqref{eq:weakint} and \eqref{eq:int} are equivalent to
\begin{align}
\T(h)=\varepsilon(\T^H(h_2))\T^A(h_2)_0 \otimes S^{-1}(h_1\T^A(h_2)_1),\label{eq:weakint2}\\
a_0[(\mathrm{Id}_A \otimes \varepsilon)\T(S^{-1}(a_2)ha_1)]=[(\mathrm{Id}_A \otimes \varepsilon)\T(h)]a.\label{eq:int2}
\end{align}
\end{proposition}

\begin{proof}
To obtain \eqref{eq:weakint2} from \eqref{eq:weakint}, it is enough to apply $(\mathrm{Id}_A \otimes m^{\op}_H)(\mathrm{Id}_A \otimes \mathrm{Id}_H \otimes S^{-1})$.
Conversely, if \eqref{eq:weakint2} holds, then
\begin{align*}
&\quad \T^A(h_2)_0 \otimes \T^H(h_2)_1 \otimes \T^H(h_2)_2 h_1\T^A(h_2)_1\\
\overset{\eqref{eq:weakint2}}&{=}\varepsilon(\T^H(h_{2_2}))\T^A(h_{2_2})_{0_0} \otimes S^{-1}(h_{2_1}\T^A(h_{2_2})_1)_1 \otimes S^{-1}(h_{2_1}\T^A(h_{2_2})_1)_2 h_1\T^A(h_{2_2})_{0_1} \\
&=\varepsilon(\T^H(h_3))\T^A(h_3)_0 \otimes S^{-1}(h_{2_2}\T^A(h_3)_{2_2}) \otimes S^{-1}(h_{2_1}\T^A(h_3)_{2_1}) h_1\T^A(h_3)_1\\
&=\varepsilon(\T^H(h_4))\T^A(h_4)_0 \otimes S^{-1}(h_3\T^A(h_4)_3) \otimes S^{-1}(\T^A(h_4)_2)S^{-1}(h_2) h_1\T^A(h_4)_1\\
&=\varepsilon(\T^H(h_2))\T^A(h_2)_0 \otimes S^{-1}(h_1\T^A(h_2)_3) \otimes S^{-1}(\T^A(h_2)_2)\T^A(h_2)_1\\
&=\varepsilon(\T^H(h_2))\T^A(h_2)_0 \otimes S^{-1}(h_1\T^A(h_2)_1) \otimes 1_H\\
\overset{\eqref{eq:weakint2}}&{=}\T(h) \otimes 1_H,
\end{align*}
thus \eqref{eq:weakint} is verified.
Similarly \eqref{eq:int2} can be obtained from \eqref{eq:int} by applying $(\mathrm{Id}_A \otimes \varepsilon)$ on both sides, whereas if \eqref{eq:weakint2} and \eqref{eq:int2} holds we find
\begin{align*}
&\quad (a_0 \otimes 1_H)\T(S^{-1}(a_2)ha_1)(1_A \otimes S^{-1}(a_3))\\
\overset{\eqref{eq:weakint2}}&{=}(a_0 \otimes 1_H)\left[\varepsilon(\T^H(S^{-1}(a_3)h_2a_2))\T^A(S^{-1}(a_3)h_2a_2)_0 \otimes S^{-1}(a_5S^{-1}(a_4)h_1a_1\T^A(S^{-1}(a_3)h_2a_2)_1)\right]\\
&=\left[\varepsilon(\T^H(S^{-1}(a_3)h_2a_2))a_0\T^A(S^{-1}(a_3)h_2a_2)_0 \otimes S^{-1}(a_1\T^A(S^{-1}(a_3)h_2a_2)_1)\right](1_A \otimes S^{-1}(h_1))\\
&=\varepsilon(\T^H(S^{-1}(a_2)h_2a_1))\left[(\mathrm{Id}_A \otimes S^{-1})\rho\left(a_0\T^A(S^{-1}(a_2)h_2a_1)\right)\right](1_A \otimes S^{-1}(h_1))\\
\overset{\eqref{eq:int2}}&{=}\varepsilon(\T^H(h_2))\left[(\mathrm{Id}_A \otimes S^{-1})\rho\left(\T^A(h_2)a\right)\right](1_A \otimes S^{-1}(h_1))\\
&=(1_A \otimes S^{-1}(a_1))\varepsilon(\T^H(h_2))\left[\T^A(h_2)_0 \otimes S^{-1}(h_1\T^A(h_2)_1)\right](a_0 \otimes 1_H)\\
\overset{\eqref{eq:weakint2}}&{=}(1_A \otimes S^{-1}(a_1))\T(h)(a_0 \otimes 1_H),
\end{align*}
i.e. \eqref{eq:int} is satisfied.
\end{proof}

\begin{remark}
Notice that while \eqref{eq:weakint} and \eqref{eq:weakint2} are completely equivalent, \eqref{eq:int} and \eqref{eq:int2} can be used interchangeably only if one of the preceding already holds.
\end{remark}

\subsection{Classic integrals in Hopf algebras}

For a thorough treatment of integrals in Hopf algebras we refer to \cite[Ch. 10]{Rad3}. 
We recall that a left integral in a Hopf algebra $H$ is an element $\Lambda \in H$  such that $h\Lambda=\varepsilon(h)\Lambda$ for every $h \in H$. A right integral in $H^*$ is a $\Bbbk$-linear map $\lambda \in H^*$ such that $\lambda(h_1)h_2=\lambda(h)1_H$ for every $h \in H$. Integrals in $H$ (and in $H^*$) exist if and only if $H$ is finite-dimensional. In that case each space of integrals is one-dimensional and $\Lambda$ and $\lambda$ can be chosen such that $\lambda(\Lambda)=1$. Given a left integral $\Lambda \in H$, there exists an algebra map $\mu \in H^*$ such that $\Lambda h= \mu(h)\Lambda$ for every $h\ \in H$. This map is called the distinguished grouplike element of $H^*$ \cite{Rad3} or the modular element in $H^*$ \cite{BT} and it is invertible (with respect to the convolution product in $H^*$) with inverse $\mu^{-1}=\mu S$. The Hopf algebra $H$ is called unimodular when $\mu=\varepsilon_H$. The element $S(\Lambda)$ is a right integral in $H$, i.e. it satisfies $S(\Lambda)h=\varepsilon(h)S(\Lambda)$ for every $h \in H$. We have $\lambda(S(\Lambda))=1$, provided that $\lambda(\Lambda)=1$.

\section{Casimir morphisms vs integrals}\label{Casvsint}

Let $H$ be a Hopf algebra and $(A,\rho_A)$ be an $H$-comodule algebra. It is easy to check that if $B:H \otimes H \rightarrow A \otimes H^{\op}$ is a Casimir morphism for the coalgebra $(H,\psi)$, then $B(\mathrm{Id}_H \otimes u_H):H \rightarrow A \otimes H^{\op}$ is an integral on $(H,\psi)$. This was already established in \cite[Prop. 3.5]{BT2} in a more general context. Our first goal is to show that if $H$ has invertible antipode, the correspondence between Casimir morphisms for $(H,\psi)$ and integrals on $(H,\psi)$ is bijective. 
First, we prove a useful lemma.

\begin{lemma}\label{Lemma:B(hx1)}
Let $H$ be a Hopf algebra with bijective antipode and let $B:H \otimes H \rightarrow A \otimes H^{\op}$ be a map satisfying \eqref{eq:Casimir1}. Then every element $B(h \otimes h')$ can be written on elements of the form $B(h \otimes 1)$. Namely
\begin{equation}\label{eq:B(hx1)}
B(h \otimes h')=(1_A \otimes h'_3)B(S^{-1}(h'_1)h \otimes 1_H)(1_A \otimes S^{-1}(h'_2))
\end{equation}
for every $h,h' \in H$.
\end{lemma}

\begin{proof}
We calculate
\begin{align*}
B(h \otimes h')&=B(\varepsilon(h'_1)h \otimes h'_2)=B(h'_2S^{-1}(h'_1)h \otimes h'_3)=B(h'_3S^{-1}(h'_1)h \otimes h'_4)(1_A \otimes \varepsilon(h'_2)1_H)\\
&=B(h'_4S^{-1}(h'_1)h \otimes h'_5)(1_A \otimes h'_3S^{-1}(h'_2))\\\overset{\eqref{eq:Casimir1}}&{=}(1_A \otimes h'_3)B(S^{-1}(h'_1)h \otimes 1_H)(1_A \otimes S^{-1}(h'_2))
\end{align*}
\end{proof}

\begin{remark}
Similarly, one can prove that under the same hypothesis
\begin{equation}\label{eq:B(1xh)}
B(h \otimes h')=(1_A \otimes h_2)B(1_H \otimes S(h_3)h')(1_A \otimes S^{-1}(h_1))
\end{equation}
also holds true for every $h,h' \in H$.
\end{remark}
Equality \eqref{eq:B(hx1)} suggests how we can invert the assignment $B \mapsto \T_B:=B(\mathrm{Id}_H \otimes u_H)$ described at the beginning of the section. We denote by $\mathrm{Casim}(H,\psi)$ the set of Casimir morphisms for the coalgebra $(H,\psi)$
and by $\int_{(H,\psi)}$ the set of all integrals on $(H,\psi)$.

\begin{theorem}\label{thm:maincas}
Let $H$ be a Hopf algebra with bijective antipode. Then 
\[
\begin{matrix}
\Phi: & \mathrm{Casim}(H,\psi) & \longrightarrow & \int_{(H,\psi)}\\
&B(h \otimes h') & \longmapsto & B(h \otimes 1_H)=:\T_B(h)
\end{matrix}
\]
is a bijective map. Moreover, it induces a bijective correspondence
\[
\begin{matrix}
 \mathrm{Casim_{norm}}(H,\psi) & \overset{\cong}{\longleftrightarrow} & \int^{\mathrm{tot}}_{(H,\psi)}
\end{matrix}
\]
between the sets of normalized Casimir morphisms for the coalgebra $(H,\psi)$ and of total integrals on $(H,\psi)$.

\end{theorem}

\begin{proof}
The map $\Phi$ is injective. Indeed if $B$ and $B'$ are Casimir morphisms for $(H,\psi)$ and $B(h \otimes 1)=B'(h \otimes 1_H)$ for every $h \in H$, then $B$ and $B'$ must coincide, by Lemma \ref{Lemma:B(hx1)}.
To prove that $\Phi$ is also surjective we consider an integral $\T \in \int_{(H,\psi)}$ and define a map $B_{\T}: H \otimes H \rightarrow A \otimes H^{\op}$ such that 
\[B_{\T}(h \otimes h')=(1_A \otimes h'_3)\T(S^{-1}(h'_1)h)(1_A \otimes S^{-1}(h'_2))\]
for every $h, h' \in H$. We are going to prove that $B_{\T}$ is a Casimir morphism for $(H,\psi)$. We have
\begin{align*}
&\quad (a_{0} \otimes 1_H)B_{\T}(g_{2}ha_{1}\otimes g_{3}h^{\prime }a_{2})(1_A \otimes g_{1})\\
&=(a_{0} \otimes (g_{3}h^{\prime }a_{2})_3)\T(S^{-1}((g_{3}h^{\prime }a_{2})_1)g_{2}ha_{1})(1_A \otimes S^{-1}((g_{3}h^{\prime }a_{2})_2)g_{1})\\
&=(a_{0} \otimes g_5h'_3a_4)\T(S^{-1}(g_{3}h'_1 a_{2})g_{2}ha_{1})(1_A \otimes S^{-1}(g_4 h'_2a_3)g_{1})\\
&=(a_{0} \otimes g_{4}h^{\prime }_3a_{4})\T(S^{-1}(h^{\prime }_1a_{2})ha_{1})(1_A \otimes S^{-1}(\varepsilon(g_{2})g_{3}h^{\prime }_2a_{3})g_{1})\\
&=(a_{0} \otimes g_{3}h^{\prime }_3a_{4})\T(S^{-1}(h^{\prime }_1a_{2})ha_{1})(1_A \otimes S^{-1}(g_{2}h^{\prime }_2a_{3})g_{1})\\
&=(a_{0} \otimes g_{2}h^{\prime }_3a_{4})\T(S^{-1}(h^{\prime }_1a_{2})ha_{1})(1_A \otimes S^{-1}(h^{\prime }_2a_{3})\varepsilon(g_{1}))\\
&=(a_{0} \otimes gh^{\prime }_3a_{4})\T(S^{-1}(a_2)S^{-1}(h^{\prime }_1)ha_{1})(1_A \otimes S^{-1}(a_3)S^{-1}(h^{\prime }_2)))\\
&=(a_{0_0} \otimes gh^{\prime }_3a_{1})\T(S^{-1}(a_{0_2})S^{-1}(h^{\prime }_1)ha_{0_1})(1_A \otimes S^{-1}(a_{0_3})S^{-1}(h^{\prime }_2)))\\
\overset{\eqref{eq:int}}&{=}(1_A \otimes gh^{\prime }_3a_{1}S^{-1}(a_{0_1}))\T(S^{-1}(h^{\prime }_1)h)(a_{0_0} \otimes S^{-1}(h^{\prime }_2)))\\
&=(1_A \otimes gh'_3)\T(S^{-1}(h'_1)h)(a \otimes S^{-1}(h'_2))\\
&=(1_A \otimes g)B_{\T}(h\otimes h^{\prime })(a \otimes 1_H),
\end{align*}
thus $B_{\T}$ satisfies \eqref{eq:Casimir1}. Moreover
\begin{align*}
&\quad \psi(h_1 \otimes B_{\T}(h_2 \otimes h'))\\
&=\psi \left(h_1 \otimes  (1_A \otimes h'_3)\T(S^{-1}(h'_1)h_2)(1_A \otimes S^{-1}(h'_2))\right)\\
&=\psi \left(h_1 \otimes  \T^A(S^{-1}(h'_1)h_2) \otimes h'_3\T^H(S^{-1}(h'_1)h_2)S^{-1}(h'_2)\right)\\
&=\T^A(S^{-1}(h'_1)h_2)_0 \otimes h'_{4}\T^H(S^{-1}(h'_1)h_2)_1S^{-1}(h'_{3}) \otimes h'_{5}\T^H(S^{-1}(h'_1)h_2)_2S^{-1}(h'_{2})h_1\T^A(S^{-1}(h'_1)h_2)_1\\
&=(1_A \otimes h'_4 \otimes h'_5)\psi \left(S^{-1}(h'_2)h_1 \otimes \T(S^{-1}(h'_1)h_2)\right)(1_A \otimes S^{-1}(h'_3) \otimes  1_H)\\
&=(1_A \otimes h'_3 \otimes h'_4)\psi \left(S^{-1}(h'_1)_1h_1 \otimes \T(S^{-1}(h'_1)_2 h_2)\right)(1_A \otimes S^{-1}(h'_2) \otimes  1_H)\\
\overset{\eqref{eq:weakint}}&{=}(1_A \otimes h'_3 \otimes h'_4)\left(\T(S^{-1}(h'_1) h) \otimes 1_H\right)(1_A \otimes S^{-1}(h'_2) \otimes  1_H)\\
&=(1_A \otimes h'_3)\T(S^{-1}(h'_1)h)(1_A \otimes S^{-1}(h'_2)) \otimes h'_4\\
&=B_{\T}(h \otimes h'_1) \otimes h'_2,
\end{align*}
which means that $B_{\T}$ also satisfies \eqref{eq:Casimir2}. Clearly the map $\Phi$ has inverse $\Phi^{-1}$ defined by the assignment $\T \mapsto B_{\T}$.
To conclude, consider a normalized Casimir morphism $B$. Then
\[\T_{B}(1_H)=\Phi(B)(1_H)=B(1_H \otimes 1_H)\overset{\eqref{eq:sepcas}}{=}1_A \otimes 1_H,\]
thus $\T_{B}$ is total. Conversely, consider a total integral $\T$. Then
\[
B_{\T}(h_1 \otimes h_2)=(1_A \otimes h_4)\T(S^{-1}(h_2)h_1)(1_A \otimes S^{-1}(h_3)) \overset{\eqref{eq:totint}}{=}1_A \otimes h_3 S^{-1}(\varepsilon(h_1)h_2)=\varepsilon(h) 1_A \otimes 1_H,
\]
thus $B_{\T}$ is normalized.
\end{proof}

\subsection{Separable cowreaths of type \texorpdfstring{$(A \otimes H^{\op},H,\psi)$}{}}

Remember that a cowreath $(A \otimes H^{\op},H,\psi)$ is separable if and only if $\mathrm{Casim_{norm}}(H,\psi) \neq \emptyset$. We claim that this is true whenever $H$ has bijective antipode and therefore that in this case all the sets mentioned in the statement of Theorem \ref{thm:maincas} are non-empty.
\begin{proposition}\label{prop:principal}
Let $H$ be a Hopf algebra with bijective antipode. The map 
\[
\begin{matrix}
\underline{B}:& H \otimes H &\longrightarrow & A \otimes H^{\op}\\
& h \otimes h' &\longmapsto &1_A \otimes h'S^{-1}(h)
\end{matrix}
\]
is a normalized Casimir morphism for $(H,\psi)$. We will call it the canonical separability morphism for $(H,\psi)$. Equivalently, the map
\[
\begin{matrix}
\Phi(\underline{B})=:\underline{\T}:& H &\longrightarrow & A \otimes H^{\op}\\
& h &\longmapsto &1_A \otimes S^{-1}(h)
\end{matrix}
\]
is a total integral on $(H,\psi)$. We will call it the canonical total integral on $(H,\psi)$.
\end{proposition}
\begin{proof}
It is an easy computation we leave for the reader to check.
\end{proof}

\begin{corollary}\label{cor:mainsep}
Let $H$ be a Hopf algebra with bijective antipode. The cowreath $(A \otimes H^{\op},H,\psi)$ is separable.
\end{corollary}

Theorem 	\ref{thm:maincas} can be of help to obtain classification results for Casimir morphisms like the one contained in \cite{mt2}. This will be shown in detail in Section \ref{sec:integrals}. In addition, the map $\Phi$ induces a bijection between the set $\mathrm{Casim^h_{norm}}(H,\psi)$ of normalized Casimir morphisms realizing h-separability and a suitable subset of $\int^{\mathrm{tot}}_{(H,\psi)}$ as well. To identify $\Phi \left(\mathrm{Casim^h_{norm}}(H,\psi) \right)$ we will use the following lemma.

\begin{lemma}\label{lemma:newhsepcondition}
Let $H$ be a Hopf algebra with bijective antipode and let $B:H \otimes H \rightarrow A \otimes H^{\op}$ be a normalized Casimir morphism for the coalgebra $(H,\psi)$. Then $B$ satisfies \eqref{eq:hsepcas} for every $h' \in H$ if and only if it satisfies \eqref{eq:hsepcas} with $h'=1_H$.
\end{lemma}

\begin{proof}
We need to prove that if
\begin{equation}\label{eq:13}
B^{A}\left( h\otimes 1_H \right) B^{A}\left( 1_H \otimes h^{\prime \prime }\right) \otimes B^{H} \left( 1_H \otimes h^{\prime \prime}\right) B^{H}\left( h\otimes 1_H \right) =B\left( h\otimes h^{\prime \prime }\right),
\end{equation}
then \eqref{eq:hsepcas} is satisfied for every $h, h', h'' \in H$.
We calculate
\begin{align*}
&\quad B^{A}\left( h\otimes h'_1 \right) B^{A}\left( h'_2 \otimes h^{\prime \prime }\right) \otimes B^{H} \left( h'_2 \otimes h^{\prime \prime}\right) B^{H}\left( h\otimes h'_1 \right)\\
\overset{\eqref{eq:B(hx1)}}&{=}B^A(S^{-1}(h'_1)h \otimes 1_H)B^A\left( h'_{4}\otimes h''\right) \otimes B^H\left( h'_{4}\otimes h''\right)h'_3B^H(S^{-1}(h'_1)h \otimes 1_H)S^{-1}(h'_2)\\
\overset{\eqref{eq:B(1xh)}}&{=}B^A(S^{-1}(h'_1)h \otimes 1_H)B^A(1_H \otimes S(h'_6)h'') \otimes\\
&\quad \otimes h'_5B^H(1_H \otimes S(h'_6)h'')S^{-1}(h'_4)h'_3B^H(S^{-1}(h'_1)h \otimes 1_H)S^{-1}(h'_2)\\
&=B^A(S^{-1}(h'_1)h \otimes 1_H)B^A(1_H \otimes S(h'_4)h'') \otimes h'_3B^H(1_H \otimes S(h'_4)h'')B^H(S^{-1}(h'_1)h \otimes 1_H)S^{-1}(h'_2)\\
\overset{\eqref{eq:13}}&{=}B^A(S^{-1}(h'_1)h \otimes S(h'_4)h'') \otimes h'_3B^H(S^{-1}(h'_1)h \otimes S(h'_4)h'')S^{-1}(h'_2)\\
&=(1_A \otimes h'_3)B(S^{-1}(h'_1)h \otimes S(h'_4)h'')(1_A \otimes S^{-1}(h'_2))\\
\overset{\eqref{eq:Casimir1}}&{=}B(h'_{4}S^{-1}(h'_1)h \otimes h'_{5}S(h'_6)h'')(1_A \otimes h'_3S^{-1}(h'_2))\\
&=B(h'_{3}S^{-1}(h'_1)h \otimes \varepsilon(h'_4)h'')(1_A \otimes \varepsilon(h'_2)1_H)\\
&=B(h'_{2}S^{-1}(h'_1)h \otimes h'')\\
&=\varepsilon(h')B(h \otimes h'')
\end{align*}
and thus we are done.
\end{proof}

By the previous lemma, if $H$ has bijective antipode and $B$ is a normalized Casimir morphism for $(H,\psi)$ satisfying \eqref{eq:hsepcas}, then
\[\T_B(h)\bullet[(1_A \otimes h'_3)\T_B(S^{-1}(h'_1))(1_A \otimes S^{-1}(h'_2))]=(1_A \otimes h'_3)\T_B(S^{-1}(h'_1)h)(1_A \otimes S^{-1}(h'_2)),\]
where $\bullet=m_A \otimes m_H^{\op}$ denotes the product in $A \otimes H^{\op}$, and $\T_B=\Phi(B)$. 
To see this, one simply rewrites \eqref{eq:13} using \eqref{eq:B(hx1)}. If we further apply $(\mathrm{Id}_A \otimes \varepsilon)$ on both sides of the equality and make use of the invertibility of $S$, then we get
\[
(\mathrm{Id}_A \otimes \varepsilon)(\T_B(h)\T_B(h'))=(\mathrm{Id}_A \otimes \varepsilon)(\T_B(h'h))
\]
for every $h,h' \in H$.
Total integrals satisfying this condition are of particular interest to us, thus we introduce the following definition.
\begin{definition}
Let $H$ be a Hopf algebra with bijective antipode and let $\T:H \rightarrow A \otimes H^{\op}$ be a total integral on $(H,\psi)$. We say that $\T$ satisfies the \emph{h-property} whenever 
\begin{equation}\label{eq:hprop2}
(\mathrm{Id}_A \otimes \varepsilon)(\T(h)\T(h'))=(\mathrm{Id}_A \otimes \varepsilon)(\T(h'h))
\end{equation}
holds
for every $h,h' \in H$.
\end{definition}
\begin{remark}
\begin{enumerate}[1)]
\item Notice that a total integral satisfies the h-property if and only if $(\mathrm{Id}_A \otimes \varepsilon)\T$ is an anti-homomorphism of algebras.
\item We have seen that a total integral $\T$ satisfying
\begin{equation}\label{eq:hprop1}
\T(h)\bullet[(1_A \otimes h'_3)\T(S^{-1}(h'_1))(1_A \otimes S^{-1}(h'_2))]=(1_A \otimes h'_3)\T(S^{-1}(h'_1)h)(1_A \otimes S^{-1}(h'_2))
\end{equation}
clearly verifies \eqref{eq:hprop2}, but the two conditions are actually equivalent. In fact suppose that \eqref{eq:hprop2} is verified. Then
\begin{align*}
&\quad \T(h)\bullet[(1_A \otimes h'_3)\T(S^{-1}(h'_1))(1_A \otimes S^{-1}(h'_2))]\\
\overset{\eqref{eq:weakint2}}&{=}[\varepsilon(\T^H(h_2))\T^A(h_2)_0 \otimes S^{-1}(h_1\T^A(h_2)_1)] \bullet\\
&\quad \bullet[\varepsilon(\T^H(S^{-1}(h'_1)_2))\T^A(S^{-1}(h'_1)_2)_0 \otimes h'_3S^{-1}(h'_2S^{-1}(h'_1)_1\T^A(S^{-1}(h'_1)_2)_1)] \\
&=\varepsilon(\T^H(h_2))\varepsilon(\T^H(S^{-1}(h'_1)))\left[\T^A(h_2)_0\T^A(S^{-1}(h'_1))_0 \otimes h'_2S^{-1}\left(h_1\T^A(h_2)_1\T^A(S^{-1}(h'_1))_1 \right)\right]\\
\overset{\eqref{eq:hprop2}}&{=}\varepsilon(\T^H(S^{-1}(h'_1)h_2))\left[\T^A(S^{-1}(h'_1)h_2)_0 \otimes h'_2S^{-1}\left(h_1\T^A(S^{-1}(h'_1)h_2)_1)\right)\right]\\
\overset{\eqref{eq:weakint2}}&{=}(1_A \otimes h'_3)\T(S^{-1}(h'_1)h)(1_A \otimes S^{-1}(h'_2)).
\end{align*}
\end{enumerate}
\end{remark}

Remember that $\mathrm{Casim^h_{norm}}(H,\psi)$ denotes the set of all normalized Casimir morphisms for $(H,\psi)$ that satisfy condition \eqref{eq:hsepcas}. Let $\int^{\mathrm{h-tot}}_{(H,\psi)}$ denote the set of all total integrals on $(H,\psi)$ that satisfy the h-property \eqref{eq:hprop2}.
\begin{theorem}
Let $H$ be a Hopf algebra with bijective antipode. The map $\Phi$ of Theorem \ref{thm:maincas}, induces a bijective correspondence
\[
\begin{matrix}
\mathrm{Casim^h_{norm}}(H,\psi) & \overset{\cong}{\longleftrightarrow} & \int^{\mathrm{h-tot}}_{(H,\psi)}.
\end{matrix}
\]
\end{theorem}

\begin{proof}
In view of the comments after the proof of Lemma \ref{lemma:newhsepcondition}, it is enough to check that if $\T \in \int^{\mathrm{h-tot}}_{(H,\psi)}$, then the map 
\[B_{\T}(h \otimes h')=(1_A \otimes h'_3)\T(S^{-1}(h'_1)h)(1_A \otimes S^{-1}(h'_2))\]
is an element of $\mathrm{Casim^h_{norm}}(H,\psi)$. We already know, by Theorem \ref{thm:maincas}, that $B_{\T} \in \mathrm{Casim_{norm}}(H,\psi)$, thus, thanks to Lemma \ref{lemma:newhsepcondition}, everything boils down to verify that $B_{\T}$ satisfies \eqref{eq:13}. Given that $\T \in \int^{\mathrm{h-tot}}_{(H,\psi)}$, it satisfies \eqref{eq:hprop1}. Hence calculate
\begin{align*}
&\quad B_{\T}^{A}\left( h\otimes 1_H \right) B_{\T}^{A}\left( 1_H \otimes h^{\prime \prime }\right) \otimes B_{\T}^{H} \left( 1_H \otimes h^{\prime \prime}\right) B_{\T}^{H}\left( h\otimes 1_H \right)\\
&=\T^A(h)\T^A(S^{-1}(h''_1)) \otimes h''_3\T^H(S^{-1}(h''_1))S^{-1}(h''_2)\T^H(h)\\
\overset{\eqref{eq:hprop1}}&{=}(1_A \otimes h''_3)\T(S^{-1}(h''_1)h)(1_A \otimes S^{-1}(h''_2))\\
&=B_{\T}(h \otimes h'').
\end{align*}
\end{proof}
Now we can strengthen Proposition \ref{prop:principal} and Corollary \ref{cor:mainsep}.
\begin{proposition}\label{Prop:hsepcow}
Let $H$ be a Hopf algebra with bijective antipode. Then $\underline{B} \in \mathrm{Casim^h_{norm}}(H,\psi)$ --- and equivalently $\underline{\T} \in \int^{\mathrm{h-tot}}_{(H,\psi)}$ --- therefore the cowreath $(A \otimes H^{\op},H,\psi)$ is h-separable.
\end{proposition}
\begin{proof}
An easy calculation left to the reader.
\end{proof}

\subsection{Frobenius cowreaths of type \texorpdfstring{$(A \otimes H^{\op},H,\psi)$}{}}

We can assume that $H$ is a finite-dimensional Hopf algebra throughout this subsection. This is justified by the following remark.
\begin{remark}\label{rmk:Frob-fd}
The fact that $(A \otimes H^{\op}, H,\psi)$ is Frobenius automatically guarantees that $H$ is finite-dimensional thanks to \cite[Prop. 1.3]{bct2} and thus that $H$ has a right dual in $\mathcal{M}$ (i.e. the classical dual $H^*$), and that the inverse of the antipode $S^{-1}$ is also well-defined.
\end{remark}

Pairing this with Proposition \ref{Prop:hsepcow}, we can immediately deduce 
\begin{theorem}
If $(A \otimes H^{\op},H,\psi)$ is a Frobenius cowreath, then it is also h-separable.
\end{theorem}
This is an extension of \cite[Prop. 6.11]{BT}, where it is shown that a Frobenius cowreath $(A \otimes H^{\op},H,\psi)$ is automatically separable.

\medskip

Our next goal is to prove that the Frobenius property for a cowreath $(A \otimes H^{\op}, H,\psi)$ is equivalent to the existence of a distinguished element in the algebra $A$. We start with a proposition proved in \cite{BT}.

\begin{proposition}\cite[Prop. 6.10]{BT}\label{prop:6.10}
Let $H$ be a finite-dimensional Hopf algebra and $(A,\rho_A)$ a right $H$-comodule algebra. If there exists an invertible element $\A \in A$ satisfying $a\A = \mu(a_1)\A a_0$ for all $a \in A$, then the cowreath $(A\otimes H^{\op}, H,\psi)$ is Frobenius.
\end{proposition}
In the same paper it was shown that the converse also holds when $H=H_4$, Sweedler's Hopf algebra, see \cite[Prop. 6.7]{BT}. We extend this result to any finite-dimensional Hopf algebra $H$.

\begin{theorem}\label{thm:Frob-cow}
Let $H$ be a finite-dimensional Hopf algebra and $(A,\rho_A)$ a right $H$-comodule algebra. The cowreath $(A \otimes H^{\op},H,\psi)$ is Frobenius if and only if there is an invertible element $\A \in A$ such that 
\begin{equation}\label{eq:propA}
a\A=\mu(a_1)\A a_0 \quad \textrm{for every } a \in A.
\end{equation}
Equivalently $(A \otimes H^{\op},H,\psi)$ is Frobenius if and only if the morphism $(\mathrm{Id}_A \otimes \mu)\rho_A$ is inner.
\end{theorem}

\begin{proof}
Suppose the cowreath $(A \otimes H^{\op}, H, \psi)$ is Frobenius with Frobenius pair $(B,\t)$. Mimicking the steps performed in the proof of \cite[Prop. 6.7]{BT}, we can write $\mathfrak{t}= \sum a_i \otimes T_i$ and assume that the $a_i$'s are linearly independent. Then \eqref{eq:Frobel1} implies that for every $i$ we have
\begin{equation}\label{eq:T_i}
T_i(h \otimes 1_H)=(h_1 \otimes h_2)T_i
\end{equation}
for every $h \in H$. By applying $(\varepsilon \otimes \mathrm{Id}_H)$ on both sides we find that the element $(\varepsilon \otimes \mathrm{Id}_H)(T_i)$ is a left integral in $H$, thus we can write $(\varepsilon \otimes \mathrm{Id}_H)(T_i)=\eta_i \Lambda$ for some $\eta_i \in \Bbbk$ and a fixed left integral $\Lambda \in H$. Then we have
\[
(\mathrm{Id}_A \otimes \varepsilon \otimes \lambda)(\mathfrak{t})=\sum a_i \otimes \lambda(\eta_i\Lambda)=\sum \eta_i a_i,\]
for a right integral $\lambda \in H^*$ such that $\lambda(\Lambda)=1$.
Let $\A:=\sum \eta_i a_i$. It follows that
\begin{align*}
a \A &= a\t^1 \lambda(\t^2 \Lambda)\lambda(\t^3)=(\mathrm{Id}_A \otimes \lambda \otimes \lambda)[(a \t^1 \otimes \t^2 \otimes \t^3)(1_A \otimes \Lambda \otimes 1_H) ]\\
\overset{\eqref{eq:Frobel1}}&{=}(\mathrm{Id}_A \otimes \lambda \otimes \lambda)[ (\t^1a_0 \otimes \t^2 \otimes \t^3 a_1)(1_A \otimes \Lambda \otimes 1_H) ]=\sum(\mathrm{Id}_A \otimes \lambda \otimes \lambda)[ a_ia_0 \otimes T_i(\Lambda \otimes a_1) ]\\
&=\sum(\mathrm{Id}_A \otimes \lambda \otimes \lambda)[ \eta_i a_ia_0 \otimes \Lambda \otimes \Lambda a_1]=\sum(\mathrm{Id}_A \otimes \lambda \otimes \lambda)[ \eta_i \mu(a_1)a_ia_0 \otimes \Lambda \otimes \Lambda]=\sum \eta_i \mu(a_1)a_ia_0\\
&=\mu(a_1)\A a_0
\end{align*}
for every $a \in A$, i.e. \eqref{eq:propA} holds. Now we prove that $\A$ is invertible. Remember that $(B,\t)$ is a Frobenius system for $(A \otimes H^{\op}, H, \psi)$, then consider $\mathfrak{B}:=(\mathrm{Id}_A \otimes \varepsilon)B(\Lambda \otimes 1_H)$. We have
\[
(\A \otimes 1_H)B(\Lambda \otimes 1_H)=\sum(a_i \otimes 1_H)B(\eta_i\Lambda \otimes 1_H)=(\t^1 \otimes 1_H)B(\varepsilon(\t^2)\t^3 \otimes 1_H)=\varepsilon(\t^2)(\t^1 \otimes 1_H)B(\t^3 \otimes 1_H),
\]
thus
\begin{align*} 
\A \mathfrak{B}&=(\mathrm{Id}_A \otimes \varepsilon)[(\A \otimes 1_H)B(\Lambda \otimes 1_H)]=(\mathrm{Id}_A \otimes \varepsilon)[(\t^1 \otimes 1_H)B(\t^3 \otimes 1_H)(1_A \otimes \t^2)]\\
\overset{\eqref{eq:Frobel2}}&{=}(\mathrm{Id}_A \otimes \varepsilon)(1_A \otimes 1_H)=1_A,
\end{align*}
which means that $\mathfrak{B}$ is a right inverse for $\A$. Let $\mathfrak{B}':=(\mathrm{Id} \otimes \mu)B(S^{-1}(\Lambda) \otimes 1_H)$. We will prove that $\mathfrak{B}'$ is a left inverse for $\A$.
First of all, observe that, in view of \eqref{eq:Casimir2}, we have
\[
B^{A}(\Lambda_{2}\otimes 1_H)_{0}\otimes B^{H}(\Lambda_{2}\otimes 1_H)_{1}\otimes B^{H}(\Lambda_{2}\otimes 1_H)_{2}\Lambda_{1}B^{A}(\Lambda_{2}\otimes
1_H)_{1}=B(\Lambda\otimes 1_H)\otimes 1_H.
\]
If we apply $(\mathrm{Id}_A \otimes \varepsilon \otimes \mu)$ on both sides we find
\[
B^{A}(\Lambda_{2}\otimes 1_H)_{0} \mu \left( B^{H}(\Lambda_{2}\otimes 1)\Lambda_{1}B^{A}(\Lambda_{2}\otimes 1_H)_{1} \right)=(\mathrm{Id}_A \otimes \varepsilon)B(\Lambda\otimes 1_H)=\mathfrak{B}.
\]
Since $\mu$ is an algebra map and $\mu(\Lambda_1)\Lambda_2=S^{-1}(\Lambda)$ by \cite[Thm. 10.5.4]{Rad3}, the first term can be rewritten as 
\[\mu \left(B^{A}( S^{-1}(\Lambda)\otimes 1_H)_1 \right)B^{A}( S^{-1}(\Lambda)\otimes 1_H)_{0} \mu \left( B^{H}( S^{-1}(\Lambda)\otimes 1_H) \right)=\mu(\mathfrak{B}'_1)\mathfrak{B}'_0.\]
It follows that $\mu(\mathfrak{B}'_1)\mathfrak{B}'_0=\mathfrak{B}$ and therefore
\[
\mathfrak{B}'\A=\mu(\mathfrak{B}'_1)\A\mathfrak{B}'_0=\A \mathfrak{B}=1_A.
\]
This shows that $\A$ is invertible and that $\A^{-1}=\mathfrak{B}=\mathfrak{B}'$. 
\end{proof}

A straightforward consequence (already pointed out in \cite[Sec 6.3]{bct2}) is that the unimodularity of $H$ implies that the cowreath $(A \otimes H^{\op},H,\psi)$ is Frobenius. More examples of Frobenius cowreaths can be obtained from central simple algebras. When $A$ is a central simple algebra over $\Bbbk$ the endomorphism $(\mathrm{Id}_A \otimes \mu)\rho_A:A \rightarrow A$ is automatically inner, thanks to the Skolem-Noether Theorem \cite[Cor. IV.1.9]{La}.

\begin{corollary}\label{cor:simpleA}
Let $H$ be a finite-dimensional Hopf algebra and $A$ a right $H$-comodule algebra, such that $A$ is central simple over $\Bbbk$. The cowreath $(A \otimes H^{\op},H,\psi)$ is Frobenius.
\end{corollary}

Nonetheless the cowreath $(A \otimes H^{\op},H,\psi)$ can be Frobenius even if $H$ is not unimodular and $A$ is not central simple over $\Bbbk$. If $A$ is endowed with the trivial $H$-coaction $\rho_A=\mathrm{Id}_A \otimes u_H$, then $(\mathrm{Id}_A \otimes \mu)\rho_A=\mathrm{Id}_A$ for any choice of $H$ and any algebra $A$. 

\begin{remark}\label{rmk:ess-uniq-A}
If it exists, the element $\A$ of Theorem \ref{thm:Frob-cow} is essentially unique in the following sense. Let $\A' \in A$ be another invertible element satisfying \eqref{eq:propA}. Then
\[a\A'(\A)^{-1}=\mu(a_1)\A'a_0(\A)^{-1}=\A'(\A)^{-1}a\]
for every $a \in A$, thus $\A$ is unique up to multiplication by an invertible element of the center of $A$.
\end{remark}

The previous remark describes what is actually a consequence of the essential uniqueness of the Frobenius pair for a Frobenius coalgebra that we have recalled in Subsection \ref{subsec:Frob}. Let us write the details to make this clear. We begin by showing how one can produce a Frobenius pair $(B_{\A},\t_{\A})$, given the element $\A$ of the previous proposition. This will give a slightly different proof of Proposition \ref{prop:6.10}.

\begin{lemma}\label{lemma:Frobpair}
Let $\A \in A$ be an invertible element satisfying \eqref{eq:propA} and let $\Lambda \in H$ be a left integral, $\lambda \in H^*$ be a right integral such that $\lambda(\Lambda)=1$. The element $\t_{\A}=\A \otimes 1_H \otimes \Lambda$ and the map $B_{\A}: H \otimes H \rightarrow A \otimes H^{\op}$ defined by $h \otimes h' \mapsto \lambda(S(h \A_1^{-1})h')\A_0^{-1}\otimes 1_H$ are part of a Frobenius pair $(B_{\A},\t_{\A})$ for the coalgebra $(H,\psi)$.
\end{lemma}

\begin{proof}
We have
\begin{align*}
a\A \otimes h \otimes \Lambda\overset{\eqref{eq:propA}}&{=}\mu(a_1)\A a_0 \otimes h \otimes \Lambda=\A a_0 \otimes h \otimes \mu(a_1)\Lambda=\A a_0 \otimes h \otimes \Lambda a_1\\
&=\A a_0 \otimes h_1  \otimes \varepsilon(h_2)\Lambda a_1=\A a_0 \otimes h_1 \otimes h_2\Lambda a_1
\end{align*}
for every $a \in A$, $h \in H$, and therefore \eqref{eq:Frobel1} is satisfied.
Then observe that $\lambda(xy_1)y_2=\lambda(x_1y)S^{-1}(x_2)$ for every $x,y \in H$, since $\lambda$ is a right integral in $H^*$. Hence
\begin{align*}
B_{\A}(h\otimes h_1')\otimes h_{2}^{\prime }&=\lambda(S(h \A_1^{-1})h'_1)\A_0^{-1}\otimes 1_H \otimes h_2'\\
&=\A_0^{-1}\otimes 1_H \otimes \lambda(S(h \A_1^{-1})h'_1)h_2'\\
&=\A_0^{-1}\otimes 1_H \otimes \lambda(S(h \A_1^{-1})_1h')S^{-1}(S(h \A_1^{-1})_2)\\
&=\A_0^{-1}\otimes 1_H \otimes \lambda(S(h_2 \A_2^{-1})h')h_1 \A_1^{-1}\\
&=\lambda(S(h_2 \A_2^{-1})h')\A_0^{-1}\otimes 1_H \otimes h_1 \A_1^{-1}\\
&=\psi(h_1 \otimes \lambda(S(h_2 \A_1^{-1})h')\A_0^{-1}\otimes 1_H)\\
&=\psi(h_1 \otimes B_{\A}(h_2 \otimes h')).
\end{align*}
The above chain of equalities shows that \eqref{eq:Casimir2} holds (see Proposition \ref{propXA} for the definition of $\psi$). Furthermore
\begin{align*}
(\t_{\A}^1 \otimes 1_H)B_{\A}(\t_{\A}^3 \otimes h)(1_A \otimes \t_{\A}^2)&=\lambda(S(\Lambda \A_1^{-1})h)\A\A_0^{-1}\otimes 1_H=\lambda(S(\Lambda )h)\mu(\A_1^{-1})\A\A_0^{-1}\otimes 1_H\\
&=\lambda(S(\Lambda )h)\A^{-1}\A\otimes 1_H=\varepsilon(h)\lambda(S(\Lambda )) 1_A \otimes 1_H\\
&=\varepsilon(h)1_A \otimes 1_H
\end{align*}
and
\begin{align*}
((\t^1_{\A})_0 \otimes 1_H)B_{\A}((\t^2_{\A})_2h(\t^1_{\A})_1 \otimes \t^3_{\A})(1_A \otimes (\t^2_{\A})_1)&=(\A_0 \otimes 1_H)B_{\A}(h\A_1 \otimes \Lambda)\\
&=\lambda(S(h)\Lambda)1_A\otimes 1_H=\varepsilon(h)1_A \otimes 1_H,
\end{align*}
therefore \eqref{eq:Frobel2} is also verified. In order to prove that \eqref{eq:Casimir1} holds, we need a couple of auxiliary identities. First of all remember that we have $\A^{-1}a=\mu(a_1)a_0 \A^{-1}$ for every $a \in A$ and therefore also
\begin{equation}\label{eq:aux1}
\A_0^{-1}a_0 \otimes \A_1^{-1}a_1 =\mu(a_1)a_{0_0} \A_0^{-1} \otimes a_{0_1} \A_1^{-1}
\end{equation}
for every $a \in A$. Since $\lambda$ is a right integral in $H^*$ we also have that
\begin{equation}\label{eq:aux2}
\lambda(yx)=\lambda(S^2(\mu(x_1)x_2)y)
\end{equation}
for every $x,y \in H$, by \cite[Thm. 10.5.4]{Rad3}. Then we can see that
\begin{align*}
(a_{0} \otimes 1_H)B_{\A}(g_{2}ha_{1}\otimes g_{3}h'a_{2})(1 \otimes g_{1})&=\lambda(S(g_{2}ha_{1} \A_1^{-1})g_{3}h'a_{2})a_0\A_0^{-1}\otimes g_1\\
&=\lambda(S(ha_{1} \A_1^{-1})h'a_{2})a_0\A_0^{-1}\otimes \varepsilon(g_2)g_1\\
&=\lambda(S(ha_{1} \A_1^{-1})h'a_{2})a_0\A_0^{-1}\otimes g\\
\overset{\eqref{eq:aux2}}&{=}\lambda(S^2(\mu(a_2)a_3)S(ha_1 \A_1^{-1})h')a_0\A_0^{-1}\otimes g\\
&=\lambda(S^2(a_3)S(h \mu(a_2)a_1\A_1^{-1})h')a_0\A_0^{-1}\otimes g\\
&=\lambda(S^2(a_2)S(h \mu(a_1) a_{0_1}\A_1^{-1})h')a_{0_0}\A_0^{-1}\otimes g\\
\overset{\eqref{eq:aux1}}&{=}\lambda(S^2(a_2)S(h \A_1^{-1}a_1)h')\A_0^{-1}a_0\otimes g\\
&=\lambda(S(h \A_1^{-1}a_1S(a_2))h')\A_0^{-1}a_0\otimes g\\
&=\lambda(S(h \A_1^{-1})h')\A_0^{-1}a\otimes g\\
&=(1 \otimes g)B_{\A}(h \otimes h')(a \otimes 1).
\end{align*}
\end{proof}
Once we fix an invertible element $\A \in A$ satisfying \eqref{eq:propA} (and a left integral $\Lambda \in H$) we can define the Frobenius pair $(B_{\A},\t_{\A})$ and use it to give a presentation of the whole set $FS_{\T^{\#}_{A \otimes H^{\op}}}(H,\psi)$ of Frobenius pairs for $(H,\psi)$, using the description reported in Subsection \ref{subsec:Frob}. First of all we need a convenient way to display the set $C_0^{\times}=(\mathrm{Hom}_{\T^{\#}_{A \otimes H^{\op}}}(H,\Bbbk))^{\times}$.
\begin{proposition}\label{prop:C-CoH}
Let 
\[
C_{A \otimes H^{\op}}(coH):=\lbrace b \otimes u \in A \otimes H^{\op} \ | \  ba_0 \otimes S^{-1}(a_1)u=a_0b \otimes u S^{-1}(a_1), \textrm{ for every } a \in A \rbrace.
\]
There is a bijective correspondence
\[
\begin{matrix}
\Theta: &(\mathrm{Hom}_{\T^{\#}_{A \otimes H^{\op}}}(H,\Bbbk))^{\times} & \overset{\cong}{\longleftrightarrow} & (C_{A \otimes H^{\op}}(coH))^\times\\
&\tilde{c} & \longmapsto & \tilde{c}(1_H)
\end{matrix}
\]
\end{proposition}

\begin{proof}
Let $\tilde{c} \in (\mathrm{Hom}_{\T^{\#}_{A \otimes H^{\op}}}(H,\Bbbk))^{\times} $. Then $\tilde{c}:H \rightarrow A \otimes H^{\op}$ is a $\Bbbk$-linear map satisfying
\begin{equation}\label{eq:c-morph}
(a_0 \otimes 1_H)\tilde{c}(l_2ha_1)(1_A \otimes l_1) =(1_A \otimes l)\tilde{c}(h)(a \otimes 1_H)
\end{equation}
for every $a \in A$ and every $h,l \in H$. If we write $b \otimes u=\tilde{c}(1_H)$ and put $a \otimes l=\mathfrak{a}_0 \otimes S^{-1}(\mathfrak{a}_1)$ we find that
\begin{align*}
b\mathfrak{a}_0 \otimes S^{-1}(\mathfrak{a}_1)u&=({\mathfrak{a}_0}_0 \otimes 1_H)\tilde{c}(S^{-1}(\mathfrak{a}_1)_2{\mathfrak{a}_0}_1)(1_A \otimes S^{-1}(\mathfrak{a}_1)_1)\\
&=(\mathfrak{a}_0 \otimes 1_H)\tilde{c}(S^{-1}(\mathfrak{a}_2)\mathfrak{a}_1)(1_A \otimes S^{-1}(\mathfrak{a}_3))\\
&=(\mathfrak{a}_0 \otimes 1_H)\tilde{c}(1_H)(1_A \otimes S^{-1}(\mathfrak{a}_1))\\
&=\mathfrak{a}_0b \otimes uS^{-1}(\mathfrak{a}_1)
\end{align*}
for every $\mathfrak{a} \in A$. Moreover there is a $\tilde{c}^{-1} \in \mathrm{Hom}_{\T^{\#}_{A \otimes H^{\op}}}(H,\Bbbk)$ such that $(\tilde{c} \circledast \tilde{c}^{-1})\odot \delta=\epsilon$, i.e.
\[
m_{A \otimes H^{\op}}(\mathrm{Id}_{A \otimes H^{\op}} \otimes m_{A \otimes H^{\op}}(\tilde{c} \otimes \tilde{c}^{-1}))\delta=\epsilon.
\]
Evaluating both sides on $1_H$ we find
\[
m_{A \otimes H^{\op}}(\tilde{c}(1_H) \otimes \tilde{c}^{-1}(1_H))=1_A \otimes 1_H,
\]
which implies that $\tilde{c}(1_H)$ is a right-invertible element of $A \otimes H^{\op}$. Similarly, left-invertibility of $\tilde{c}$ implies that $\tilde{c}^{-1}(1_H)$ is a left inverse of $\tilde{c}(1_H)$. We can conclude that $\tilde{c}(1_H)$ belongs to $(C_{A \otimes H^{\op}}(coH))^\times$. The map $\Theta$ of the statement is injective, in fact
\begin{align*}
\tilde{c}(h)&=\varepsilon(h_1)\tilde{c}(h_2)=\tilde{c}(h_2)(1_A \otimes h_{1_2}S^{-1}(h_{1_1}))=\tilde{c}(h_{2_2})(1_A \otimes h_{2_1}S^{-1}(h_1))\\
&\overset{\eqref{eq:c-morph}}{=}(1_A \otimes h_2)\tilde{c}(1_H)(1_A \otimes S^{-1}(h_1))
\end{align*}
for every $h \in H$. To conclude the proof it is enough to show that the map defined by
\[\tilde{d}(h):=b \otimes h_2uS^{-1}(h_1)\]
for every $h \in H$, is in $(\mathrm{Hom}_{\T^{\#}_{A \otimes H^{\op}}}(H,\Bbbk))^{\times}$, provided that $b \otimes u \in (C_{A \otimes H^{\op}}(coH))^\times$.
We calculate
\begin{align*}
(a_0 \otimes 1_H)\tilde{d}(l_2ha_1)(1_A \otimes l_1)&=(a_0 \otimes 1_H)(b \otimes (l_2ha_1)_2uS^{-1}((l_2h a_1)_1))(1_A \otimes l_1)\\
&=a_0b \otimes l_3h_2a_2 uS^{-1}(l_2h_1 a_1)l_1\\
&=ba_0 \otimes l_3h_2a_2S^{-1}(a_1) uS^{-1}(l_2h_1)l_1\\
&=ba \otimes l_3h_2uS^{-1}(h_1)S^{-1}(l_2)l_1\\
&=ba \otimes l h_2uS^{-1}(h_1)\\
&=(1_A \otimes l)(b \otimes h_2uS^{-1}(h_1))(a \otimes 1_H)\\
&=(1_A \otimes l)\tilde{d}(h)(a \otimes 1_H).
\end{align*}
Then
\begin{align*}
[(\tilde{d} \circledast \tilde{d'})\odot \delta](h)&=m_{A \otimes H^{\op}}(\mathrm{Id}_{A \otimes H^{\op}} \otimes m_{A \otimes H^{\op}}(\tilde{d} \otimes \tilde{d'}))\delta(h)\\
&=m_{A \otimes H^{\op}}(\tilde{d}(h_1) \otimes \tilde{d'}(h_2))\\
&=m_{A \otimes H^{\op}}(b \otimes h_2uS^{-1}(h_1) \otimes b' \otimes h_4u'S^{-1}(h_3))\\
&=bb' \otimes h_4u'S^{-1}(h_3)h_2uS^{-1}(h_1)\\
&=bb' \otimes h_2u'uS^{-1}(h_1),
\end{align*}
thus if $\tilde{d'}$ is defined via $b'\otimes u'=(b \otimes u)^{-1}$, we find that $[(\tilde{d} \circledast \tilde{d'})\odot \delta](h)=\varepsilon(h)1_A \otimes 1_H=\epsilon(h)$. By swapping the roles of $\tilde{d}$ and $\tilde{d}'$ it follows that $\tilde{d}$ is an invertible element of $\mathrm{Hom}_{\T^{\#}_{A \otimes H^{\op}}}(H,\Bbbk)$.
\end{proof}

\begin{remark}\label{rmk:coinv}
The set $C_{A \otimes H^{\op}}(coH)$ corresponds to the centralizer in $A \otimes H^{\op}$ of the subset
\[coH=\lbrace a_0 \otimes S^{-1}(a_1) \ | \ a \in A \rbrace \subseteq A \otimes H^{\op},\]
which is actually the so-called subalgebra of coinvariants in the Galois context for cowreaths introduced in \cite{BT2}. For further reference on this see \cite[Subss. 2.3 and 6.3]{BT}. This set is tipically denoted by the letter $B$, but we used $coH$ to avoid confusion with the notation for Casimir morphisms.
\end{remark}

Now we are ready to give a complete description of $FS_{\T^{\#}_{A \otimes H^{\op}}}(H,\psi)$.

\begin{theorem}\label{thm:Frobpair}
Fix an invertible element $\A \in A$ satisfying \eqref{eq:propA}, a left integral $\Lambda \in H$ and a right integral $\lambda \in H^*$ such that $\lambda(\Lambda)=1$. Then all Frobenius pairs $(B',\t')$ for the cowreath $(A \otimes H^{\op},H,\psi)$ are given by
\begin{align*}
B'(h \otimes h')&=\lambda(S(\overline{u}_2S^{-1}(h'_1) h \overline{b}_1\A_1^{-1}))\overline{b}_0\A_0^{-1}\otimes  h'_3\overline{u}_1 S^{-1}(h'_2)\\
\t'&=\A b \otimes \Lambda_2 u S^{-1}(\Lambda_1) \otimes \Lambda_3,
\end{align*}
with $b \otimes u \in (C_{A \otimes H^{\op}}(coH))^{\times}$ and $\overline{b} \otimes \overline{u}=(b \otimes u)^{-1}$.
\end{theorem}

\begin{proof}
Let $(A,X,\psi)$ be a Frobenius cowreath in the monoidal category $\mathcal{T}^{\#}_A$ built from $\mathrm{Vec}_{\Bbbk}$ and let $(B,\t)$ be a fixed Frobenius pair. Then every Frobenius pair $(B',\t')$ for $(A,X,\psi)$ is of the form
\begin{align*}
B'&=B \odot (\mathrm{Id}_{(X,\psi)} \circledast \tilde{c}^{-1} \circledast \mathrm{Id}_{(X,\psi)} ) \odot(\mathrm{Id}_{(X,\psi)} \circledast \delta),\\
\t'&= (\tilde{c} \circledast \mathrm{Id}_{(X,\psi)}) \odot \delta \odot \t,
\end{align*} 
for an invertible morphism $\tilde{c}: X \rightarrow \Bbbk$ in $\mathcal{T}^{\#}_A$ (see Subsection \ref{subsec:Frob}).
Explicitly, these two equalities read
\begin{align*}
B'&=m_A(m_A \otimes B)(\mathrm{Id}_A \otimes \psi_X \otimes \textrm{Id}_X)(\psi_X \otimes \tilde{c}^{-1} \otimes \textrm{Id}_X)(\textrm{Id}_X\otimes \delta),\\
\t'&= (m_A \otimes \mathrm{Id}_X)(m_A \otimes \tilde{c}\otimes \textrm{Id}_X)(\mathrm{Id}_A \otimes \delta)\t.
\end{align*} 
By specializing to the case of cowreaths $(A \otimes H^{\op},H,\psi)$, we get to \begin{align*}
B'(h \otimes h')&=m_{A \otimes H^{\op}}(\mathrm{Id}_{A \otimes H^{\op}} \otimes B)(\psi(h \otimes \tilde{c}^{-1}(h'_1)) \otimes h'_2),\\
\t'&=(\t^1 \otimes 1_H)\tilde{c}(\t^3_1)(1_A \otimes \t^2) \otimes \t^3_2
\end{align*}
for $\tilde{c} \in (\mathrm{Hom}_{\T^{\#}_{A \otimes H^{\op}}}(H,\Bbbk))^{\times}$, i.e.
\begin{align*}
B'(h \otimes h')&=m_{A \otimes H^{\op}}(\mathrm{Id}_{A \otimes H^{\op}} \otimes B)(\psi(h \otimes \overline{b} \otimes h'_2\overline{u}S^{-1}(h'_1)) \otimes h'_3),\\
\t'&=\t^1b \otimes \t^3_2uS^{-1}(\t^3_1)\t^2 \otimes \t^3_3
\end{align*}
for $b \otimes u \in (C_{A \otimes H^{\op}}(coH))^\times$, thanks to Proposition \ref{prop:C-CoH}; here $\overline{b} \otimes \overline{u}$ is used to denote $(b \otimes u)^{-1}$. Once we choose $(B,\t)=(B_{\A},\t_{\A})$ we find
\begin{align*}
\t'&=\A b \otimes \Lambda_2 u S^{-1}(\Lambda_1) \otimes \Lambda_3
\end{align*}
and
\begin{align*}
B'(h \otimes h')&=m_{A \otimes H^{\op}}(\mathrm{Id}_{A \otimes H^{\op}} \otimes B_{\A})(\overline{b}_0 \otimes h'_3\overline{u}_1 S^{-1}(h'_2) \otimes h'_4 \overline{u}_2S^{-1}(h'_1) h \overline{b}_1  \otimes h'_5)\\
&=(\overline{b}_0 \otimes 1_H)B_{\A}(h'_4 \overline{u}_2S^{-1}(h'_1) h \overline{b}_1  \otimes h'_5)(1_A \otimes h'_3\overline{u}_1 S^{-1}(h'_2))\\
&=(\overline{b}_0 \otimes 1_H)(\lambda(S(h'_4 \overline{u}_2S^{-1}(h'_1) h\overline{b}_1 \A_1^{-1})h'_5)\A_0^{-1}\otimes 1_H)(1_A \otimes h'_3\overline{u}_1 S^{-1}(h'_2))\\
&=\lambda(S(\overline{u}_2S^{-1}(h'_1) h\overline{b}_1 \A_1^{-1}))\overline{b}_0\A_0^{-1}\otimes  h'_3\overline{u}_1 S^{-1}(h'_2).
\end{align*}
\end{proof}

\begin{remark}
Following the proof of Theorem \ref{thm:Frob-cow} we set $\A':=(\mathrm{Id}_A \otimes \varepsilon \otimes \lambda)(\t')$. Then, from the expression of $\t'$ we have found, we get
\[\A':=\A b\varepsilon(\Lambda_2 u S^{-1}(\Lambda_1))\lambda(\Lambda_3)=\A b\varepsilon(u),\]
while it is easy to see that the image of $(C_{A \otimes H^{\op}}(coH))^\times$ via $(\mathrm{Id}_A \otimes \varepsilon)$ is the group of invertible elements in the center of $A$. Thus we recover the content of Remark \ref{rmk:ess-uniq-A}.
\end{remark}

Given that for a Frobenius cowreath $(A \otimes H^{\op},H,\psi)$ the Hopf algebra $H$ has a bijective antipode, Theorem \ref{thm:maincas} applies and we can describe Frobenius pairs for $(A \otimes H^{\op},H,\psi)$ using integrals in place of Casimir morphisms.

\begin{corollary}
Fix an invertible element $\A \in A$ satisfying \eqref{eq:propA}, a left integral $\Lambda \in H$ and a right integral $\lambda \in H^*$ such that $\lambda(\Lambda)=1$. Then all Frobenius pairs $(\T',\t')$ for the cowreath $(A \otimes H^{\op},H,\psi)$ are given by
\begin{align*}
\T'(h)&=\lambda(S(\overline{u}_2h \overline{b}_1\A_1^{-1}))\overline{b}_0\A_0^{-1}\otimes  \overline{u}_1\\
\t'&=\A b \otimes \Lambda_2 u S^{-1}(\Lambda_1) \otimes \Lambda_3,
\end{align*}
with $b \otimes u \in (C_{A \otimes H^{\op}}(coH))^{\times}$ and $\overline{b} \otimes \overline{u}=(b \otimes u)^{-1}$.
\end{corollary}

\begin{remark}
Our result agrees with \cite[Thm. 3.6]{BT2}, that implies that integrals on $(H,\psi)$ are in bijective correspondence with elements of $C_{A \otimes H^{\op}}(coH)$, provided $(A \otimes H^{\op},H,\psi)$ is Frobenius.
\end{remark}

\subsubsection{Frobenius pairs vs normalized Casimir morphisms}

We have seen that when the cowreath $(A \otimes H^{\op},H,\psi)$ is Frobenius, the canonical separability morphism $\underline{B}$ is well-defined and thus $(A \otimes H^{\op},H,\psi)$ is (h-)separable. We should note however that $\underline{B}$ is almost never part of a Frobenius pair. 

\begin{proposition}
The canonical separability morphism $\underline{B}$ is part of a Frobenius pair $(\underline{B},\t)$ for $(A \otimes H^{\op},H,\psi)$ only if $H \cong \Bbbk$.
\end{proposition}

\begin{proof}
Suppose $(\underline{B},\t)$ is a Frobenius pair for $(A \otimes H^{\op},H,\psi)$. Then \eqref{eq:Frobel2} implies
\[
\t^1 \otimes S^{-1}(\t^3)\t^2=1_A \otimes 1_H
\]
for every $h \in H$. Hence
\begin{align*}
1_A \otimes h &=\t^1 \otimes S^{-1}(\t^3)\t^2h \overset{\eqref{eq:Frobel1}}{=}\t^1 \otimes S^{-1}(h_2\t^3)h_1\t^2=\varepsilon(h)\t^1 \otimes S^{-1}(\t^3)\t^2=\varepsilon(h)1_A \otimes 1_H,
\end{align*}
i.e. $h=\varepsilon(h)1_H$ for every $h \in H$.
\end{proof}

Nonetheless, normalized Casimir morphisms can be part of a Frobenius pair. As a concrete example, consider $H=H_4$ and $A=Cl(-1,-1,0)$ the four-dimensional Clifford algebra over $\Bbbk$, generated by $G$ and $X$ with $G^2=X^2=-1$ and $GX+XG=0$. $A$ is an $H_4$-comodule algebra with coaction given by $\rho(G)=G \otimes g$ and $\rho(X)=X \otimes g + 1_A \otimes x$, see \cite{PVO2}. The Casimir morphism $B: H_4 \otimes H_4 \rightarrow A \otimes H_4^{\op} $ for $(H_4,\psi)$ defined by
\begin{align*}
&B(1_{H_4} \otimes 1_{H_4})=1_A \otimes 1_{H_4}, &&B(g \otimes 1_{H_4})=-G \otimes 1_{H_4},\\
&B(x \otimes 1_{H_4})=-GX \otimes 1_{H_4}, &&B(gx \otimes 1_{H_4})=1_A \otimes x,
\end{align*}
and \eqref{eq:B(hx1)} is clearly normalized (see also \cite[Thm. 1]{mt2}). Moreover, the reader can check that $(B,\t)$ is a Frobenius system for $(A \otimes H_4^{\op},H_4, \psi)$, if $\t=GX \otimes 1_{H_4} \otimes (1_{H_4}+g)x+1_A \otimes x \otimes (1_{H_4}-g)x$. It can also be recovered from the formulas in the statement of Theorem \ref{thm:Frobpair} with $\A=GX$, $\Lambda=x+gx$, $\lambda(h)=\delta_{h,gx}$ and $b \otimes u =1_A \otimes 1_{H_4} -GX \otimes x$.
%

\section{Integrals on \texorpdfstring{$(E(n),\psi)$}{} when \texorpdfstring{$A=Cl(\alpha,\beta_i, \gamma_i, \lambda_{ij})$}{}}\label{sec:integrals}

In this section we will determine all integrals on cowreaths of type $(A \otimes H^{\op},H,\psi)$ where $H=E(n)$ and $A=Cl(\alpha,\beta_i, \gamma_i, \lambda_{ij})$. We recall the definitions of these algebras and some of their useful properties. From now on we will always assume $\mathrm{char} (\Bbbk) \neq 2$.

\bigskip

\textbf{The Hopf algebras $E(n)$}.
The Hopf algebras denoted $E(n)$ are a family that generalizes Sweedler's Hopf algebra $H_4=E(1)$ and that was introduced in \cite[p. 755]{BDG} and studied in \cite{CD, CC, PVO, PVO2}. 
\begin{definition}{\cite[p. 18]{CD}}\label{defEn} $E(n)$ will denote the $2^{n+1}$-dimensional Hopf algebra over $\Bbbk$, generated by elements $g, x_1, \ldots, x_n$ such that $g^2=1$, $x_i^2=0$, $gx_i=-x_ig$ for any $i=1, \ldots, n$ and $x_ix_j=-x_jx_i$ for $i,j=1, \ldots, n$, $i <j$.
\end{definition}
The canonical Hopf algebra structure of $E(n)$ is given by
\[
\def\arraystretch{1.5}
\begin{matrix*}[l]
\Delta(g)=g \otimes g, & \varepsilon(g)=1, & S(g)=g^{-1}=g, &\\
\Delta(x_i)=x_i \otimes g +1 \otimes x_i, & \varepsilon(x_i)=0, & S(x_i)=gx_i, &  i=1, \ldots, n.
\end{matrix*}
\]
For $P=\lbrace i_1,i_2,\ldots,i_s \rbrace \subseteq \lbrace 1,2, \ldots,n \rbrace$ such that $i_1 < i_2 < \ldots < i_s$, we write $x_P := x_{i_1}x_{i_2}\cdots x_{i_s}$. If $P=\emptyset$ then $x_{\emptyset} = 1$. The set $\mathbb{B}_H=\lbrace g^jx_P \ | \ P\subseteq \lbrace 1, \ldots ,n \rbrace, j \in \lbrace 0,1 \rbrace \rbrace$ is a basis of $E(n)$.
Let $F=\lbrace i_{j_1}, i_{j_2}, \ldots, i_{j_r} \rbrace$ be a subset of $P$ and define
\[S(F,P):=(j_1+\cdots+j_r)-\frac{r(r+1)}{2} \quad \textrm{and} \quad S(\emptyset,P):=0.\]
Then computations show that
\begin{equation}\label{deltagx}
\begin{split}
\Delta(g^jx_P)&=\sum_{F \subseteq P} (-1)^{S(F,P)} g^j x_F \otimes g^{|F|+j} x_{P \setminus F}\\
&=\sum_{F \subseteq P} (-1)^{S(P \setminus F,P)} g^j x_{P \setminus F} \otimes g^{|P|-|F|+j} x_F,
\end{split}
\end{equation}
\begin{equation}\label{antipode}
S(g^jx_P) = (-1)^{j|P|}g^{j+|P|}x_P,
\end{equation}
\begin{equation}\label{reorder}
(-1)^{S(F,P)}x_Fx_{P \setminus F}=x_P.
\end{equation}

\bigskip

\textbf{Clifford algebras}. A Clifford algebra $Cl(V,q)$ is usually defined as a quotient of the tensor algebra $T(V)$ built on a free vector space $V$ by an ideal whose generators depend on a quadratic form $q$ \cite{C,La}. The following is an equivalent definition \cite[Thm. 3.3]{FR2}.
\begin{definition}\cite[Def. 1]{PVO2}\label{clt}
Let $\alpha, \beta_i,\gamma_i \in \Bbbk$ for $i=1, \ldots, n$ and $\lambda_{ij} \in \Bbbk$ for $i, j \in \lbrace 1, \ldots, n \rbrace$, $i<j$.
The Clifford algebra $Cl(\alpha,\beta_i,\gamma_i, \lambda_{ij})$ is the unital associative algebra generated by elements $G, X_1, \ldots, X_n$ such that $G^2=\alpha$, $X_i^2=\beta_i$, $GX_i+X_iG=\gamma_i$ for all $i=1, \ldots, n$ and $X_iX_j+X_jX_i=\lambda_{ij}$ for all $i,j \in \lbrace 1, \ldots, n \rbrace$ with $i<j$. A $\Bbbk$-basis for this algebra is given by $\mathbb{B}_A=\lbrace G^j X_P \ | \ P \subseteq \lbrace 1, \ldots, n \rbrace, j \in \lbrace 0,1 \rbrace \rbrace$, where $X_P:=X_{i_1} \cdots X_{i_s}$ when $P=\lbrace i_1 < i_2 < \ldots < i_s \rbrace$. 
\end{definition}
\begin{notation}
We warn the reader about the slightly misleading notation $Cl(\alpha,\beta_i,\gamma_i, \lambda_{ij})$. This should always be read as
\[Cl(\alpha, \beta_1, \beta_2, \ldots, \beta_n, \gamma_1, \gamma_2, \ldots, \gamma_n, \lambda_{12}, \ldots, \lambda_{1n}, \lambda_{23}, \ldots, \lambda_{nn})\]
and considered a single algebra (not a family) once all the defining scalars are fixed.
\end{notation}

In \cite{PVO2} it was proved that a $2^{n+1}$-dimensional Clifford algebra $A=Cl(\alpha, \beta_i, \gamma_i, \lambda_{ij})$ admits a canonical $E(n)$-comodule algebra structure $\rho: A \rightarrow A \otimes E(n)$ that makes it a cleft extension of the ground field $\Bbbk$. This coaction is given by
\[
\def\arraystretch{1.5}
\begin{matrix*}[l]
\rho(1_A)=1_A \otimes 1_{E(n)}, &\rho(G) =G \otimes g,\\
\rho(X_i)= X_i \otimes g + 1_A \otimes x_i, &\rho(GX_i) = GX_i \otimes 1_{E(n)} +  G \otimes gx_i, & i=1, \ldots, n.
\end{matrix*}
\]
Computations show that in general
\begin{equation}\label{rhoGX}
\begin{split}
\rho(G^jX_P)&=\sum_{F \subseteq P} (-1)^{S(F,P)} G^j X_F \otimes g^{|F|+j} x_{P \setminus F}\\
&=\sum_{F \subseteq P} (-1)^{S(P \setminus F,P)} G^j X_{P \setminus F} \otimes g^{|P|-|F|+j} x_F.
\end{split}
\end{equation}

Clifford algebras are $\mathbb{Z}_2$-graded algebras (superalgebras) and elements $G^jX_P$ can be regarded as monomials in $G,X_1, \ldots, X_n$ with degree $j+|P|$. The same can be said for each Hopf algebra $E(n)$, as it is isomorphic to the $2^{n+1}$-dimensional algebra $Cl(1,0,0,0)$.

\bigskip

\textbf{Determining the integrals}. Let $A=Cl(\alpha, \beta_i, \gamma_i, \lambda_{ij})$ be a $2^{n+1}$-dimensional Clifford algebra, equipped with the $E(n)$-comodule algebra structure defined by \eqref{rhoGX} and consider the cowreath $(A \otimes E(n)^{\op},E(n),\psi)$.
Since $E(n)$ has bijective antipode, by Proposition \ref{prop:weakint2} integrals on the coalgebra $(E(n),\psi)$ correspond to $\Bbbk$-linear maps $\T:E(n) \rightarrow A \otimes E(n)^{\op}$ that satisfy \eqref{eq:weakint2} and \eqref{eq:int2} for every element of the basis $\mathbb{B}_H$ of $E(n)$ and every element of the basis $\mathbb{B}_A$ of $A$. As a matter of fact, equation \eqref{eq:int2} is multiplicative with respect to $a$, therefore a $\Bbbk$-linear map $\T:E(n) \rightarrow A \otimes E(n)^{\op}$ is an integral on $(E(n),\psi)$ if and only if it satisfies \eqref{eq:weakint2} and \eqref{eq:int2} for every element $h \in \mathbb{B}_H$ and every generator $a \in \lbrace G, X_1, \ldots, X_n \rbrace$. The study of the elements $\T(g^jx_P)$ can follow two different approaches, depending on the parity of the degree $j+|P|$.

\subsection{Monomials with odd degree}

We are going to prove that the element $\T(g)$ completely determines every other element $\T(g^jx_P)$ with $j+|P|$ odd. Notice that these elements can also be written in the form $\T(g^{|P|+1}x_P)$. Suppose $\T$ is an integral on $(E(n),\psi)$ and let $g^{|P|+1}x_P$ be an element of $\mathbb{B}_H$ with $P \neq \emptyset$. Let $r$ be the greatest element in $P$ so that $x_P=x_{P'}x_r$ with $P'=P \setminus \lbrace r \rbrace$. Then \eqref{eq:int} with $a=X_r$ and $h=g^{|P|}x_{P'}$ yields
\begin{align}
2\T(g^{|P|+1}x_P)= &(1_A \otimes g)\T(g^{|P'|+1}x_{P'})(X_r \otimes g)-(1_A \otimes gx_r)\T(g^{|P'|+1}x_{P'})(1_A \otimes g)+\label{eq:T-odd}\\
&-\T(g^{|P'|+1}x_{P'})(1_A \otimes x_r)+(-1)^{|P'|+1}(X_r \otimes 1_{E(n)})\T(g^{|P'|+1}x_{P'}).\nonumber
\end{align}
If we iterate this process, it becomes clear that $\T(g^{|P|+1}x_P)$ can be ultimately expressed in terms of $\T(g)$ in $|P|$ steps.
Conversely, the following proposition assures that if $\T(g)$ satisfies \eqref{eq:weakint2} and \eqref{eq:int2} with $a=G$ and every other $\T(g^{|P|+1}x_P)$ is built using \eqref{eq:T-odd}, then the map $\T$ behaves like an integral on monomials with odd degree.
\begin{proposition}
Suppose $\T:E(n) \rightarrow A \otimes E(n)^{\op}$ is a $\Bbbk$-linear map such that \eqref{eq:weakint2} and \eqref{eq:int2} are satisfied for $h=g$ and $a=G$. Suppose further that \eqref{eq:T-odd} holds for every element $g^{|P|+1}x_P \in \mathbb{B}_H$ with $P \neq \emptyset$. Then \eqref{eq:weakint2} and \eqref{eq:int2} hold true for every $h=g^{|P|+1}x_P \in \mathbb{B}_H$ and every generator $a \in \lbrace G, X_1, \ldots, X_n \rbrace$ of $A$.
\end{proposition}

\begin{proof}
The proof is by induction on $m=|P|$. When $|P|=0$ the statement is true thanks to our set of hypotheses. In particular, \eqref{eq:int2} for $a=X_i$ corresponds to \eqref{eq:T-odd} with $x_r=x_i$ and $x_{P'}=1_{E(n)}$, once we apply $(\mathrm{Id}_A \otimes \varepsilon)$ on both sides. 
Now suppose \eqref{eq:weakint2} and \eqref{eq:int2} are satisfied for every generator $a \in \lbrace G, X_1, \ldots, X_n \rbrace$, and every element $h=g^{|L|+1}x_L$ with $L \leq m$, for some fixed $m \geq 0$. Explicitly, this means
\begin{equation}\label{eq:IndHyp0}
\T\left(g^{|L|+1}x_L\right)=\varepsilon\left(\T^H\left((g^{|L|+1}x_L)_2\right)\right)\T^A \left((g^{|L|+1}x_L)_2 \right)_0 \otimes S^{-1}\left((g^{|L|+1}x_L)_1\T^A\left((g^{|L|+1}x_L)_2\right)_1\right),
\end{equation}
\begin{equation}\label{eq:IndHyp1}
(-1)^{|L|}G [(\mathrm{Id}_A \otimes \varepsilon)\T(g^{|L|+1}x_L)]=[(\mathrm{Id}_A \otimes \varepsilon)\T(g^{|L|+1}x_L)]G,
\end{equation}
\begin{equation}\label{eq:IndHyp2}
\begin{split}
&(-1)^{|L|}X_i[(\mathrm{Id}_A \otimes \varepsilon)\T(g^{|L|+1}x_L)]+2(\mathrm{Id}_A \otimes \varepsilon)\T(g^{|L|}x_L x_i)\\
&=[(\mathrm{Id}_A \otimes \varepsilon)\T(g^{|L|+1}x_L)]X_i, \quad i=1, \ldots, n.
\end{split}
\end{equation}
Let $g^{|P|+1}x_P \in \mathbb{B}_H$ be an element with $|P|=m+1$. We have to prove that \eqref{eq:weakint2} and \eqref{eq:int2} are verified for $h=g^{|P|+1}x_P$ and every generator $a \in \lbrace G, X_1, \ldots, X_n \rbrace$.

Let $r$ be the greatest element in $P$ so that $x_P=x_{P'}x_r$ with $P'=P \setminus \lbrace r \rbrace$. It follows that
\begin{align*}
&2G[(\mathrm{Id}_A \otimes \varepsilon)\T(gg^{|P|+1}x_Pg)]\\
&=(-1)^{|P|}2G [(\mathrm{Id}_A \otimes \varepsilon)\T(g^{|P|+1}x_P)]\\
\overset{\eqref{eq:T-odd}}&{=}(-1)^{|P|}G [(\mathrm{Id}_A \otimes \varepsilon)\T(g^{|P|}x_{P'})]X_r+GX_r[(\mathrm{Id}_A \otimes \varepsilon)\T(g^{|P|}x_{P'})]\\
\overset{\eqref{eq:IndHyp1}}&{=}-[(\mathrm{Id}_A \otimes \varepsilon)\T(g^{|P|}x_{P'})]GX_r+\gamma_r[(\mathrm{Id}_A \otimes \varepsilon)\T(g^{|P|}x_{P'})]-(-1)^{|P'|}X_r [(\mathrm{Id}_A \otimes \varepsilon)\T(g^{|P|}x_{P'})]G\\
&=[(\mathrm{Id}_A \otimes \varepsilon)\T(g^{|P|}x_{P'})]X_rG+(-1)^{|P'|+1}X_r [(\mathrm{Id}_A \otimes \varepsilon)\T(g^{|P|}x_{P'})]G\\
\overset{\eqref{eq:T-odd}}&{=}2[(\mathrm{Id}_A \otimes \varepsilon)\T(g^{|P|+1}x_P)]G,
\end{align*}
thus \eqref{eq:int2} holds for $h=g^{|P|+1}x_P$ and $a=G$. Now consider a generator $X_i$ with $i \in \lbrace 1, \ldots, n \rbrace$. Equality \eqref{eq:int2} with $h=g^{|P|+1}x_P$ and $a=X_i$ reads
\begin{align}
(-1)^{|P|}X_i[(\mathrm{Id}_A \otimes \varepsilon)\T(g^{|P|+1}x_P)]+2(\mathrm{Id}_A \otimes \varepsilon)\T(g^{|P|}x_Px_i)=[(\mathrm{Id}_A \otimes \varepsilon)\T(g^{|P|+1}x_P)]X_i,\label{eq:T30}
\end{align}
which is true by \eqref{eq:T-odd} if $i>r$. 
If $i\leq r$, we can still use \eqref{eq:T-odd} to rewrite \eqref{eq:T30} as
\begin{align*}
&\frac{1}{2}(-1)^{|P|}X_i[(\mathrm{Id}_A \otimes \varepsilon)\T(g^{|P|}x_{P'})]X_r+\frac{1}{2}X_iX_r [(\mathrm{Id}_A \otimes \varepsilon)\T(g^{|P|}x_{P'})]+2(\mathrm{Id}_A \otimes \varepsilon)\T(g^{|P|}x_Px_i)\\
&=\frac{1}{2}[(\mathrm{Id}_A \otimes \varepsilon)\T(g^{|P|}x_{P'})]X_rX_i+\frac{1}{2}(-1)^{|P'|+1}X_r [(\mathrm{Id}_A \otimes \varepsilon)\T(g^{|P|}x_{P'})]X_i
\end{align*}
Next we can use \eqref{eq:IndHyp2} to expand terms of the form $X_i[(\mathrm{Id}_A \otimes \varepsilon)\T(g^{|P|}x_{P'})]\mathfrak{a}$ and obtain that \eqref{eq:T30} is also equivalent to
\begin{equation}\label{eq:T31}
[(\mathrm{Id}_A \otimes \varepsilon)\T(g^{|P'|}x_{P'} x_i)]X_r+(-1)^{|P'|}X_r [(\mathrm{Id}_A \otimes \varepsilon)\T(g^{|P'|}x_{P'} x_i)]+2(\mathrm{Id}_A \otimes \varepsilon)\T(g^{|P|}x_Px_i)=0.
\end{equation}
Now we have to deal with three cases:
\begin{itemize}
\item If $i \notin P$ we can write $\overline{P}=P \sqcup \lbrace i \rbrace$ and $\overline{P'}=P' \sqcup \lbrace i \rbrace$ and \eqref{eq:T31} becomes
\begin{equation*}
\begin{split}
&(-1)^{S(P',\overline{P'})}[(\mathrm{Id}_A \otimes \varepsilon)\T(g^{|P'|}x_{\overline{P'}})]X_r+(-1)^{|P'|+S(P',\overline{P'})}X_r[(\mathrm{Id}_A \otimes \varepsilon)\T(g^{|P'|}x_{\overline{P'}})]+\\
&+(-1)^{S(P,\overline{P})}2(\mathrm{Id}_A \otimes \varepsilon)\T(g^{|P|}x_{\overline{P}})=0.
\end{split}
\end{equation*}
This is true by \eqref{eq:T-odd}, given that $S(P,\overline{P})=S(P' \sqcup \lbrace r \rbrace,\overline{P'} \sqcup \lbrace r \rbrace)=S(P',\overline{P'})+1$, in view of \cite[Lem. 3.6]{FR3}.
\item If $i \in P$, but $i \neq r$, then $i \in P'$ and consequently $x_Px_i=0=x_{P'}x_i$. In this case \eqref{eq:T31} is trivially true.
\item If $i=r$, \eqref{eq:T31} changes into
\begin{equation*}
[(\mathrm{Id}_A \otimes \varepsilon)\T(g^{|P|+1}x_P)]X_r+(-1)^{|P'|}X_r[(\mathrm{Id}_A \otimes \varepsilon)\T(g^{|P|+1}x_P)]=0.
\end{equation*}
By employing \eqref{eq:T-odd} once again, we get that the above equation is an identity.
\end{itemize}
We can conclude that \eqref{eq:int2} also holds for $h=g^{|P|+1}x_P$ and $a=X_i$ with $i=1, \ldots, n$. 

Finally we check that \eqref{eq:weakint2} is verified for $h=g^{|P|+1}x_P$. We calculate
\begin{align*}
&\varepsilon\left(\T^H\left((g^{|P|+1}x_P)_2\right)\right)\T^A\left((g^{|P|+1}x_P)_2\right)_0 \otimes S^{-1}\left((g^{|P|+1}x_P)_1\T^A\left((g^{|P|+1}x_P)_2\right)_1\right)\\
&=\varepsilon\left(\T^H\left((g^{|P|+1}x_{P'}x_r)_2\right)\right)\T^A\left((g^{|P|+1}x_{P'}x_r)_2\right)_0 \otimes S^{-1}\left((g^{|P|+1}x_{P'}x_r)_1\T^A\left((g^{|P|+1}x_{P'}x_r)_2\right)_1\right)\\
&=\varepsilon\left(\T^H\left((g^{|P|+1}x_{P'})_2g\right)\right)\T^A\left((g^{|P|+1}x_{P'})_2g\right)_0 \otimes S^{-1}\left((g^{|P|+1}x_{P'})_1x_r\T^A\left((g^{|P|+1}x_{P'})_2g\right)_1\right)\\
&+\varepsilon\left(\T^H\left((g^{|P|+1}x_{P'})_2x_r\right)\right)\T^A\left((g^{|P|+1}x_{P'})_2x_r\right)_0 \otimes S^{-1}\left((g^{|P|+1}x_{P'})_1\T^A\left((g^{|P|+1}x_{P'})_2x_r\right)_1\right)\\
\overset{\eqref{deltagx}}&{=}(-1)^{|P'|+|P|}\varepsilon\left(\T^H\left((g^{|P|}x_{P'})_2\right)\right)\T^A\left((g^{|P|}x_{P'})_2\right)_0 \otimes S^{-1}\left(gx_r(g^{|P|}x_{P'})_1\T^A\left((g^{|P|}x_{P'})_2\right)_1\right)\\
&+\sum_{F \subseteq P'} (-1)^{S(F,P')}\left[(\mathrm{Id}_A \otimes S^{-1})(\rho \otimes \varepsilon)\T\left(g^{|P'|+|F|} x_{P' \setminus F}x_r\right)\right]\left[1_A \otimes S^{-1}(g^{|P'|}x_{F'})\right]\\
\overset{\eqref{eq:IndHyp2}}&{=}-\varepsilon\left(\T^H\left((g^{|P|}x_{P'})_2\right)\right)\T^A\left((g^{|P|}x_{P'})_2\right)_0 \otimes S^{-1}\left(gx_r(g^{|P|}x_{P'})_1\T^A\left((g^{|P|}x_{P'})_2\right)_1\right)\\
&+\frac{1}{2}\sum_{F \subseteq P'} (-1)^{S(F,P')}\left[(\mathrm{Id}_A \otimes S^{-1})\rho\left((\mathrm{Id}_A \otimes \varepsilon)\T\left(g^{|P|+|F|} x_{P' \setminus F}\right)X_r\right)\left(1_A \otimes S^{-1}(g^{|P'|}x_{F'})\right)\right.\\
&+\left.(-1)^{|P'|+|F|+1}(\mathrm{Id}_A \otimes S^{-1})\rho\left(X_r(\mathrm{Id}_A \otimes \varepsilon)\T\left(g^{|P|+|F|} x_{P' \setminus F}\right)\right)\left(1_A \otimes S^{-1}(g^{|P'|}x_{F'})\right)\right]\\
\overset{\eqref{deltagx}}&{=}-\varepsilon\left(\T^H\left((g^{|P|}x_{P'})_2\right)\right)\T^A\left((g^{|P|}x_{P'})_2\right)_0 \otimes S^{-1}\left((g^{|P|}x_{P'})_1\T^A\left((g^{|P|}x_{P'})_2\right)_1\right)x_r\\
&+\frac{1}{2}\varepsilon\left(\T^H\left((g^{|P|} x_{P'})_2 \right)\right)\T^A\left((g^{|P|} x_{P'})_2\right)_0X_r \otimes g S^{-1} \left((g^{|P|} x_{P'})_1\T^A\left((g^{|P|} x_{P'})_2\right)_1 \right)g\\
&-\frac{1}{2}\varepsilon\left(\T^H\left((g^{|P|} x_{P'})_2 \right)\right)\T^A\left((g^{|P|} x_{P'})_2\right)_0 \otimes gx_r S^{-1} \left((g^{|P|} x_{P'})_1\T^A\left((g^{|P|} x_{P'})_2\right)_1 \right)g\\
&+\frac{1}{2}(-1)^{|P|}\varepsilon\left(\T^H\left((g^{|P|} x_{P'})_2 \right)\right)X_r\T^A\left((g^{|P|} x_{P'})_2\right)_0 \otimes S^{-1} \left((g^{|P|} x_{P'})_1\T^A\left((g^{|P|} x_{P'})_2\right)_1  \right)\\
&+\frac{1}{2}\varepsilon\left(\T^H\left((g^{|P|} x_{P'})_2 \right)\right)\T^A\left((g^{|P|} x_{P'})_2\right)_0 \otimes S^{-1} \left((g^{|P|} x_{P'})_1\T^A\left((g^{|P|} x_{P'})_2\right)_1 \right)x_r\\
\overset{\eqref{eq:IndHyp0}}&{=}-\frac{1}{2}\T(g^{|P|}x_{P'})(1_A \otimes x_r)+\frac{1}{2}(1_A \otimes g)\T(g^{|P|}x_{P'})(X_r \otimes g)\\
&-\frac{1}{2}(1_A \otimes gx_r)\T(g^{|P|}x_{P'})(1_A \otimes g)+\frac{1}{2}(-1)^{|P'|+1}(X_r \otimes 1_{E(n)})\T(g^{|P|}X_{P'})\\
\overset{\eqref{eq:T-odd}}&{=}\T(g^{|P|+1}x_P).
\end{align*}
\end{proof}
The next step is to determine the explicit form of $\T(g)$.
\begin{proposition}
The element $\T(g)$ satisfies \eqref{eq:weakint2} and \eqref{eq:int2} with $a=G$ if and only if
\begin{equation}\label{eq:Tg}
\T(g)=\sum_{j,P}\sum_{Q \subseteq P} (-1)^{S(P \setminus Q,P)+(j+|P|+1)|Q|} \mu_{j,P}G^jX_{P \setminus Q} \otimes g^{j+|P|+1}x_Q,
\end{equation}
for $\mu_{j,P} \in \Bbbk$ such that
\begin{align}
&(1+(-1)^{|P|+1}) \mu_{0,P}=\sum_{i \in P^c}  (-1)^{S(P,P \sqcup \lbrace i \rbrace)}\gamma_i \mu_{1, P \sqcup \lbrace i \rbrace}, \label{eq:mu-0P}\\
&(1+(-1)^{|P|+1}) \mu_{1,P}\alpha=\sum_{i \in P^c}  (-1)^{S(P,P \sqcup \lbrace i \rbrace)}\gamma_i \mu_{0, P \sqcup \lbrace i \rbrace}, \label{eq:mu-1P}
\end{align}
for every $P \subseteq \lbrace 1, \ldots, n \rbrace$.
\end{proposition}

\begin{proof}
From \eqref{eq:weakint2} we see that $\T(g)=a_0 \otimes S^{-1}(a_1)g$ for some $a \in A$. Write $a=\sum_{j,P} \mu_{j,P}G^jX_P$, then $\T(g)$ is of the form displayed in \eqref{eq:Tg}. Equality \eqref{eq:int2} with $h=g$ and $a=G$ now reads
\[
\sum_{j,P} \mu_{j,P}G^{j+1}X_P=\sum_{j,P} \mu_{j,P}G^jX_PG.
\]
Thanks to \cite[Lem. 3.8]{FR3} we can further rewrite this as 
\[
\sum_{j,P} (1+(-1)^{|P|+1})\mu_{j,P}G^{j+1}X_P=\sum_{j,P} \sum_{i \in P} (-1)^{S(P \setminus \lbrace i \rbrace,P)}\mu_{j,P}\gamma_{i} G^j X_{P \setminus \lbrace i \rbrace}.
\]
A suitable rearrangement of indexes yields
\[
\sum_{j,P} (1+(-1)^{|P|+1})\mu_{j,P}G^{j+1}X_P=\sum_{j,R} \sum_{i \in R^c} (-1)^{S(R,R \sqcup \lbrace i \rbrace)}\mu_{j,R \sqcup \lbrace i \rbrace} \gamma_{i}G^j X_R,
\]
which, by linear independence, is equivalent to \eqref{eq:mu-0P}, \eqref{eq:mu-1P}.
\end{proof}

\begin{remark}\label{rmk:principal1}
The canonical total integral $\underline{\T}$ defined in Proposition \ref{prop:principal} corresponds to the following solution of \eqref{eq:mu-0P} and \eqref{eq:mu-1P}: $\mu_{0,\emptyset}=1$ and $\mu_{j,P}=0$ for $(j,P) \neq (0, \emptyset)$.
\end{remark}

\subsection{Monomials with even degree}

Now we study elements $\T(g^jx_P)$ with $j+|P|$ even, i.e. elements of the form $\T(g^{|P|}x_P)$. If we write \eqref{eq:weakint2} and \eqref{eq:int2} for $h=g^{|P|}x_P \in \mathbb{B}_H$ and $a \in \lbrace G, X_1, \ldots, X_n \rbrace$ we find
\begin{align}
&\T(g^{|P|}x_P)=\varepsilon(\T^H((g^{|P|}x_P)_2))\T^A((g^{|P|}x_P)_2)_0 \otimes S^{-1}((g^{|P|}x_P)_1\T^A((g^{|P|}x_P)_2)_1),\label{eq:10_on_gen}\\
&(-1)^{|P|}G [(\mathrm{Id}_A \otimes \varepsilon)\T(g^{|P|}x_P)]=[(\mathrm{Id}_A \otimes \varepsilon)\T(g^{|P|}x_P)]G,\label{eq:11_on_G}\\
&(-1)^{|P|}X_i[(\mathrm{Id}_A \otimes \varepsilon)\T(g^{|P|}x_P)]=[(\mathrm{Id}_A \otimes \varepsilon)\T(g^{|P|}x_P)]X_i.\label{eq:11_on_X}
\end{align}

Let $\mathfrak{U}_P:=[(\mathrm{Id}_A \otimes S^{-1})(\rho \otimes \varepsilon)\T(g^{|P|}x_P)](1_A \otimes g^{|P|})$. Then, by means of \eqref{deltagx} and \eqref{antipode}, we can rephrase \eqref{eq:10_on_gen} in the following way.
\begin{equation}\label{eq41}
\begin{split}
\T(g^{|P|}x_P)&=\sum_{R \subseteq P} (-1)^{S(P \setminus R,P)}\varepsilon(\T^H(g^{|R|}x_R))\T^A(g^{|R|}x_R)_0 \otimes S^{-1}(g^{|P|}x_{P \setminus R}\T^A(g^{|R|}x_R)_1)\\
&=\sum_{R \subseteq P} (-1)^{S(P \setminus R,P)+|R|(|P|-|R|)}\mathfrak{U}_R(1_A \otimes x_{P \setminus R}).
\end{split}
\end{equation}
\begin{lemma}
The family of equations \eqref{eq:10_on_gen} for $P \subseteq \lbrace 1, \ldots, n \rbrace$ is equivalent to the family
\begin{equation}\label{eq:T-even}
\begin{split}
\mathfrak{U}_P=\sum_{F \subseteq P}(-1)^{S(F,P)+|P|+|F|}\T(g^{|F|}x_{F}) (1_A \otimes x_{P \setminus F}).
\end{split}
\end{equation}
for $P \subseteq \lbrace 1, \ldots, n \rbrace$. Each element $\T(g^{|P|}x_P)$ satisfy \eqref{eq:11_on_G} and \eqref{eq:11_on_X} if and only if
\begin{equation}\label{eq:UP_G}
(-1)^{|P|}G[(\mathrm{Id}_A \otimes \varepsilon)\mathfrak{U}_P]=[(\mathrm{Id}_A \otimes \varepsilon)\mathfrak{U}_P]G,
\end{equation}
\begin{equation}\label{eq:UP_X}
(-1)^{|P|}X_i[(\mathrm{Id}_A \otimes \varepsilon)\mathfrak{U}_P]=[(\mathrm{Id}_A \otimes \varepsilon)\mathfrak{U}_P]X_i,
\end{equation}
for $i=1, \ldots, n$.
\end{lemma}

\begin{proof}
Throughout this proof, $(\star)$ refers to \cite[Lem. 3.6]{FR3}.
Suppose \eqref{eq:10_on_gen}, i.e. \eqref{eq41}, is verified.  Then
\begin{align*}
&\sum_{F \subseteq P}(-1)^{S(F,P)+|P|+|F|}\T(g^{|F|}x_{F}) (1_A \otimes x_{P \setminus F})\\
\overset{\eqref{eq41}}&{=}\sum_{F \subseteq P}\sum_{R \subseteq F}(-1)^{S(F,P)+|P|+|F|+S(F \setminus R,F)+|R|(|F|-|R|)}\mathfrak{U}_R(1_A \otimes x_{F \setminus R}x_{P \setminus F})\\
\overset{(\star)}&{=}\sum_{F \subseteq P}\sum_{R \subseteq F}(-1)^{S(R,P)+|P|+|F|}\mathfrak{U}_R(1_A \otimes x_{P \setminus R})\\
&=\sum_{R \subseteq P}\left[\sum_{T \subseteq P \setminus R}(-1)^{|T|}\right](-1)^{S(R,P)+|P|+|R|}\mathfrak{U}_R(1_A \otimes x_{P \setminus R})\\
&=\mathfrak{U}_P.
\end{align*}
Conversely, suppose \eqref{eq:T-even} holds. Then
\begin{align*}
&\sum_{R \subseteq P} (-1)^{S(P \setminus R,P)+|R|(|P|-|R|)}\mathfrak{U}_R(1_A \otimes x_{P \setminus R})\\
&\overset{\eqref{eq:T-even}}{=}\sum_{R \subseteq P} (-1)^{S(P \setminus R,P)+|R|(|P|-|R|)}\sum_{F \subseteq R}(-1)^{S(F,R)+|R|+|F|}\T(g^{|F|}x_{F}) (1_A \otimes x_{R \setminus F}x_{P \setminus R})\\
&\overset{(\star)}{=}\sum_{R \subseteq P}\sum_{F \subseteq R} (-1)^{S(F,P)+|R|+|F|}\T(g^{|F|}x_{F}) (1_A \otimes x_{P \setminus F})\\
&=\sum_{F \subseteq P}\left[\sum_{T \subseteq P \setminus F}(-1)^{|T|} \right] (-1)^{S(F,P)}\T(g^{|F|}x_{F}) (1_A \otimes x_{P \setminus F})\\
&=\T(g^{|P|}x_P),
\end{align*}
thus \eqref{eq41} is verified. As shown before, this is equivalent to \eqref{eq:10_on_gen}. Since $(\mathrm{Id}_A \otimes \varepsilon)\T(g^{|P|}x_P)=(\mathrm{Id}_A \otimes \varepsilon)\mathfrak{U}_P$ (for example, using \eqref{eq41}), the last part easily follows.
\end{proof}

The previous lemma shows that each $\T(g^{|P|}x_P)$ is determined by the images of $\T$ on monomials of lower degree and the additional term $\mathfrak{U}_P$. In fact, \eqref{eq:T-even} implies that
\[
\T(g^{|P|}x_P)=\sum_{F \subsetneq P}(-1)^{S(F,P)+|P|+|F|+1}\T(g^{|F|}x_{F}) (1_A \otimes x_{P \setminus F})+\mathfrak{U}_P.
\]
As a consequence, to determine all the elements $\T(g^{|P|}x_P)$ is equivalent to determine every $\mathfrak{U}_P$. To identify the exact form of $\mathfrak{U}_P$ we recall that $\mathfrak{U}_P=[(\mathrm{Id}_A \otimes S^{-1})(\rho \otimes \varepsilon)\T(g^{|P|}x_P)](1_A \otimes g^{|P|})$ by definition and thus $\mathfrak{U}_P=a_0 \otimes S^{-1}(a_1)g^{|P|}$ for some $a \in A$. Let $a=\sum_{k,Q} \eta(P;k,Q)G^kX_Q$ for some $\eta(P;k,Q) \in \Bbbk$, so that
\begin{equation}\label{eq:UP-form}
\mathfrak{U}_P=\sum_{k,Q}\sum_{F \subseteq Q}  \eta(P;k,Q)(-1)^{S(Q \setminus F,Q)+(|P|+|Q|+k)|F|} G^k X_{Q \setminus F} \otimes g^{|P|+|Q|+k}x_F.
\end{equation}

\begin{proposition}
The element $\mathfrak{U}_P$ satisfies \eqref{eq:UP_G} if and only if
\begin{equation}\label{eq:etaG}
[(-1)^{|P|}+(-1)^{|Q|+1}]\alpha^k \eta(P;k,Q)=\sum_{i \in Q^c}  (-1)^{S(Q,Q \sqcup \lbrace i \rbrace)} \gamma_i\eta(P;1-k,Q \sqcup \lbrace i \rbrace).
\end{equation}
for every $Q \subseteq \lbrace 1, \ldots, n \rbrace$ and $k=0,1$.
\end{proposition}

\begin{proof}
If we substitute \eqref{eq:UP-form} into \eqref{eq:UP_G} we find
\[
\sum_{k,Q}(-1)^{|P|}\eta(P;k,Q)G^{k+1} X_{Q}=\sum_{k,Q}\eta(P;k,Q)G^k X_{Q}G,
\]
which, by \cite[Lem. 3.6]{FR3}, is equivalent to
\begin{equation*}
\sum_{k,Q}[(-1)^{|P|}-(-1)^{|Q|}]\eta(P;k,Q)G^{k+1} X_{Q}=\sum_{k,Q}\sum_{i \in Q} (-1)^{S(Q \setminus \lbrace i \rbrace,Q)}\eta(P;k,Q)\gamma_{i} G^kX_{Q \setminus \lbrace i \rbrace}.
\end{equation*}
By further rearranging indexes on the right we get
\begin{equation*}
\sum_{k,Q}[(-1)^{|P|}-(-1)^{|Q|}]\eta(P;k,Q)G^{k+1} X_{Q}=\sum_{k,R}\sum_{i \in R^c} (-1)^{S(R,R \sqcup  \lbrace i \rbrace)}\eta(P;k,R \sqcup  \lbrace i \rbrace)\gamma_{i} G^kX_R
\end{equation*}
and by linear independence we can conclude.
\end{proof}

\begin{proposition}
The element $\mathfrak{U}_P$ satisfies \eqref{eq:UP_X} for a fixed $i \in \lbrace 1, \ldots, n \rbrace$ if and only if
\begin{equation}\label{eq:eta1.1}
\begin{split}
&\left[(-1)^{|P|+1}+(-1)^{|Q|+1} \right](-1)^{S(Q,Q \sqcup \lbrace i \rbrace)}\eta(P;1,Q)\\
&=\sum_{\underset{l<i}{l \in Q^c}}(-1)^{S(\lbrace l \rbrace,Q \sqcup \lbrace i,l \rbrace)+|P|+|Q|}\lambda_{li} \eta(P;1,Q \sqcup \lbrace i,l \rbrace)+\\
&+\sum_{\underset{l>i}{l \in Q^c}}(-1)^{S(\lbrace l \rbrace,Q \sqcup \lbrace i,l \rbrace)+1}\lambda_{li} \eta(P;1,Q \sqcup \lbrace i,l \rbrace),
\end{split}
\end{equation}
\begin{equation}\label{eq:eta1.2}
\begin{split}
&\left[(-1)^{|P|}+(-1)^{|Q|}\right](-1)^{S(Q, Q \sqcup \lbrace i \rbrace)}\beta_i\eta(P;1,Q \sqcup \lbrace i \rbrace)\\
&=\sum_{\underset{l<i}{l \in Q^c}} (-1)^{S(\lbrace l \rbrace,Q \sqcup \lbrace l \rbrace)+|P|+|Q|+1}\lambda_{li}\eta(P;1,Q \sqcup \lbrace l \rbrace)+\\
&+\sum_{\underset{l>i}{l \in Q^c}}(-1)^{S(\lbrace l \rbrace,Q \sqcup \lbrace l \rbrace)+1}\lambda_{li}\eta(P;1,Q \sqcup \lbrace l \rbrace),
\end{split}
\end{equation}
\begin{equation}\label{eq:eta0.1}
\begin{split}
&\left[(-1)^{|P|+1}+(-1)^{|Q|}\right](-1)^{S(Q,Q \sqcup \lbrace i \rbrace)}\eta(P;0,Q)+\gamma_i\eta(P;1,Q \sqcup \lbrace i \rbrace)\\
&=\sum_{\underset{l<i}{l \in Q^c}} (-1)^{S(\lbrace l \rbrace,Q \sqcup \lbrace i,l \rbrace)+1}\lambda_{li}\eta(P;0,Q \sqcup \lbrace i,l \rbrace)+\\
&+\sum_{\underset{l>i}{l \in Q^c}}(-1)^{S(\lbrace l \rbrace,Q \sqcup \lbrace i,l \rbrace)+|P|+|Q|+1}\lambda_{li}\eta(P;0,Q \sqcup \lbrace i,l \rbrace),
\end{split}
\end{equation}
\begin{equation}\label{eq:eta0.2}
\begin{split}
&\left[(-1)^{|P|+1}+(-1)^{|Q|} \right](-1)^{S(Q, Q \sqcup \lbrace i \rbrace)}\beta_i\eta(P;0,Q \sqcup \lbrace i \rbrace)+\gamma_i\eta(P;1,Q)\\
&=\sum_{\underset{l<i}{l \in Q^c}}(-1)^{S(\lbrace l \rbrace,Q \sqcup \lbrace l \rbrace)+1}\lambda_{li}\eta(P;0,Q \sqcup \lbrace l \rbrace)+\\
&+\sum_{\underset{l>i}{l \in Q^c}}(-1)^{S(\lbrace l \rbrace,Q \sqcup \lbrace l \rbrace)+|P|+|Q|}\lambda_{li} \eta(P;0,Q \sqcup \lbrace l \rbrace),
\end{split}
\end{equation}
hold for every $Q \subseteq \lbrace 1, \ldots, n \rbrace$ such that $i \notin Q$.
\end{proposition}

\begin{proof}
Condition \eqref{eq:UP_X} reads
\[
\sum_{k,Q}(-1)^{|P|}\eta(P;k,Q)X_iG^kX_Q=\sum_{k,Q} \eta(P;k,Q)G^kX_QX_i,
\]
which is equivalent to
\[
\begin{cases}
\sum_{Q}(-1)^{|P|+1}\eta(P;1,Q)GX_iX_Q=\sum_{Q} \eta(P;1,Q)GX_QX_i,\\
\sum_{Q}(-1)^{|P|}\eta(P;0,Q)X_iX_Q+\sum_{Q}(-1)^{|P|}\eta(P;1,Q)\gamma_iX_Q=\sum_{Q} \eta(P;k,Q)X_QX_i,
\end{cases}
\]
if we make a distinction between terms that contain $G$ and the rest. Then we can use \cite[Lem. 3.8]{FR3} and further split both equalities by looking at terms where $X_i$ appears (or does not appear):
\begin{equation*}
\begin{split}
\sum_{Q \not \ni i}\eta(P;1,Q)\left[(-1)^{|P|+1+S(\lbrace i \rbrace, Q \sqcup \lbrace i \rbrace)}GX_{Q \sqcup \lbrace i \rbrace}+ (-1)^{S(Q,Q \sqcup \lbrace i \rbrace)+1}GX_{Q \sqcup \lbrace i \rbrace}\right]\\
=\sum_{Q \ni i} \eta(P;1,Q)\left[\sum_{\underset{l>i}{l \in Q}}(-1)^{S(Q \setminus \lbrace l \rbrace,Q)}\lambda_{li} GX_{Q \setminus \lbrace l \rbrace}+\sum_{\underset{l<i}{l \in Q}}(-1)^{S(\lbrace l \rbrace,Q)+|P|}\lambda_{li} GX_{Q \setminus \lbrace l \rbrace}\right]
\end{split}
\end{equation*}

\bigskip

\begin{equation*}
\begin{split}
&\sum_{Q \ni i}\eta(P;1,Q)\left[(-1)^{|P|+1+S(\lbrace i \rbrace, Q)}\beta_iGX_{Q \setminus \lbrace i \rbrace}+ (-1)^{S(Q \setminus \lbrace i \rbrace, Q)+1}\beta_i GX_{Q \setminus \lbrace i \rbrace}\right]\\
&=\sum_{Q \not\ni i}\eta(P;1,Q)\left[\sum_{\underset{l>i}{l \in Q}}(-1)^{S(Q \setminus \lbrace l \rbrace,Q)}\lambda_{li} GX_{Q \setminus \lbrace l \rbrace}+\sum_{\underset{l<i}{l \in Q}} (-1)^{|P|+S(\lbrace l \rbrace,Q)}\lambda_{li} GX_{Q \setminus \lbrace l \rbrace}\right],
\end{split}
\end{equation*}

\bigskip

\begin{equation*}
\begin{split}
&\sum_{Q \ni i}\eta(P;0,Q)\left[\sum_{\underset{l<i}{l \in Q}} (-1)^{S(\lbrace l \rbrace,Q)+|P|}\lambda_{li} X_{Q \setminus \lbrace l \rbrace}+\sum_{\underset{l>i}{l \in Q}}(-1)^{S(Q \setminus \lbrace l \rbrace,Q)+1}\lambda_{li} X_{Q \setminus \lbrace l \rbrace}\right]\\
&+\sum_{Q \ni i}(-1)^{|P|}\eta(P;1,Q)\gamma_i X_Q\\
&=\sum_{Q \not\ni i}\eta(P;0,Q)\left[(-1)^{S(Q,Q \sqcup \lbrace i \rbrace)} X_{Q \sqcup \lbrace i \rbrace}+(-1)^{|P|+S(\lbrace i \rbrace, Q \sqcup \lbrace i \rbrace)+1}X_{Q \sqcup \lbrace i \rbrace}\right],
\end{split}
\end{equation*}

\bigskip

\begin{equation*}
\begin{split}
&\sum_{Q \ni i}\eta(P;0,Q)\left[(-1)^{|P|+S(\lbrace i \rbrace, Q)}\beta_i X_{Q \setminus \lbrace i \rbrace}+(-1)^{S(Q \setminus \lbrace i \rbrace, Q)+1}\beta_i X_{Q \setminus \lbrace i \rbrace}\right]\\
&=\sum_{Q \not\ni i} \eta(P;0,Q)\left[\sum_{\underset{l>i}{l \in Q}}(-1)^{S(Q \setminus \lbrace l \rbrace,Q)}\lambda_{li} X_{Q \setminus \lbrace l \rbrace}+\sum_{\underset{l<i}{l \in Q}}(-1)^{|P|+S(\lbrace l \rbrace,Q)+1}\lambda_{li} X_{Q \setminus \lbrace l \rbrace}\right]+\\
&+\sum_{Q \not\ni i}(-1)^{|P|+1}\eta(P;1,Q)\gamma_iX_Q.
\end{split}
\end{equation*}
Now we recall that each subset $Q$ of $\lbrace 1, \ldots n \rbrace$ containing $i$ is in bijection with a subset of $\lbrace 1, \ldots, n \rbrace$ that does not contain $i$ (namely $Q \setminus \lbrace i \rbrace$). Thus we can rewrite each equality so that every sum runs over subsets $Q$ not containing $i$. Then we can rearrange indexes of the kind $Q \setminus \lbrace l \rbrace$ and make use of linear independence to find 
\begin{equation*}
\begin{split}
&\left[(-1)^{|P|+1+S(\lbrace i \rbrace, Q \sqcup \lbrace i \rbrace)}+ (-1)^{S(Q,Q \sqcup \lbrace i \rbrace)+1}\right]\eta(P;1,Q)\\
&=\sum_{\underset{l<i}{l \in Q^c}}(-1)^{S(\lbrace l \rbrace,Q \sqcup \lbrace i,l \rbrace)+|P|} \eta(P;1,Q \sqcup \lbrace i,l \rbrace)\lambda_{li}+\\
&+\sum_{\underset{l>i}{l \in Q^c}}(-1)^{S(Q \sqcup \lbrace i \rbrace,Q \sqcup \lbrace i,l \rbrace)} \eta(P;1,Q \sqcup \lbrace i,l \rbrace)\lambda_{li},
\end{split}
\end{equation*}

\bigskip

\begin{equation*}
\begin{split}
&\left[(-1)^{S(Q, Q \sqcup \lbrace i \rbrace)}+(-1)^{|P|+S(\lbrace i \rbrace, Q \sqcup \lbrace i \rbrace)}\right]\eta(P;1,Q \sqcup \lbrace i \rbrace)\beta_i\\
&=\sum_{\underset{l<i}{l \in Q^c}} (-1)^{S(\lbrace l \rbrace,Q \sqcup \lbrace l \rbrace)+|P|+1}\eta(P;1,Q \sqcup \lbrace l \rbrace)\lambda_{li}+\\
&+\sum_{\underset{l>i}{l \in Q^c}}(-1)^{S(Q,Q \sqcup \lbrace l \rbrace)+1}\eta(P;1,Q \sqcup \lbrace l \rbrace)\lambda_{li},
\end{split}
\end{equation*}

\bigskip

\begin{equation*}
\begin{split}
&\left[(-1)^{S(Q,Q \sqcup \lbrace i \rbrace)}+(-1)^{|P|+S(\lbrace i \rbrace, Q \sqcup \lbrace i \rbrace)+1}\right]\eta(P;0,Q)+(-1)^{|P|+1}\eta(P;1,Q \sqcup \lbrace i \rbrace)\gamma_i\\
&=\sum_{\underset{l<i}{l \in Q^c}} (-1)^{S(\lbrace l \rbrace,Q \sqcup \lbrace i,l \rbrace)+|P|}\eta(P;0,Q \sqcup \lbrace i,l \rbrace)\lambda_{li}+\\
&+\sum_{\underset{l>i}{l \in Q^c}}(-1)^{S(Q \sqcup \lbrace i \rbrace,Q \sqcup \lbrace i,l \rbrace)+1}\eta(P;0,Q \sqcup \lbrace i,l \rbrace)\lambda_{li},
\end{split}
\end{equation*}

\bigskip

\begin{equation*}
\begin{split}
&\left[(-1)^{|P|+S(\lbrace i \rbrace, Q \sqcup \lbrace i \rbrace)} +(-1)^{S(Q, Q \sqcup \lbrace i \rbrace)+1} \right]\eta(P;0,Q \sqcup \lbrace i \rbrace)\beta_i+(-1)^{|P|}\eta(P;1,Q)\gamma_i\\
&=\sum_{\underset{l<i}{l \in Q^c}}(-1)^{|P|+S(\lbrace l \rbrace,Q \sqcup \lbrace l \rbrace)+1} \eta(P;0,Q \sqcup \lbrace l \rbrace)\lambda_{li}+\\
&+\sum_{\underset{l>i}{l \in Q^c}}(-1)^{S(Q,Q \sqcup \lbrace l \rbrace)} \eta(P;0,Q \sqcup \lbrace l \rbrace)\lambda_{li}.
\end{split}
\end{equation*}
for every $Q$ such that $i \notin Q$. Finally we can use \cite[Lem. 3.6]{FR3} to obtain \eqref{eq:eta1.1}-\eqref{eq:eta0.2}.
\end{proof}

\begin{remark}\label{rmk:principal2}
The canonical total integral $\underline{\T}$ defined in Proposition \ref{prop:principal} corresponds to the following solution of \eqref{eq:etaG} and \eqref{eq:eta1.1}-\eqref{eq:eta0.2}: $\eta(\emptyset;0,\emptyset)=1$ and $\eta(P;k,Q)=0$ for $(P,k,Q) \neq (\emptyset, 0, \emptyset)$.
\end{remark}

Now we can we give a complete characterization of integrals on the coalgebra $(E(n),\psi)$. 
\begin{theorem}\label{thm:mainint}
A $\Bbbk$-linear map $\T:E(n) \rightarrow A \otimes E(n)^{\op}$ is an integral on the coalgebra $(E(n),\psi)$ if and only if
\begin{itemize}\setlength\itemsep{1em}
\item $\displaystyle \T(g)=\sum_{j,P}\sum_{F \subseteq P} (-1)^{S(P \setminus F,P)+(j+|P|+1)|F|} \mu_{j,P}G^jX_{P \setminus F} \otimes g^{j+|P|+1}x_F$
\\[0.5em] with $\mu_{j,P} \in \Bbbk$ such that \eqref{eq:mu-0P}, \eqref{eq:mu-1P} hold for every $P \subseteq \lbrace 1, \ldots, n \rbrace$.
\item $\displaystyle \begin{aligned}[t]\T(g^{|P|}x_{P}x_r)&= \frac{1}{2}\left[(1_A \otimes g)\T(g^{|P|+1}x_{P})(X_r \otimes g)-(1_A \otimes gx_r)\T(g^{|P|+1}x_{P})(1_A \otimes g)+\right.\\
&\left.-\T(g^{|P|+1}x_{P})(1_A \otimes x_r)+(-1)^{|P|+1}(X_r \otimes 1_{E(n)})\T(g^{|P|+1}x_{P})\right]\end{aligned}$
\\[0.5em]for every $P \subseteq \lbrace 1, \ldots, n \rbrace \ni r$, such that $r >j$ for every $j \in P$.
\item $\displaystyle 
\T(g^{|P|}x_P)=\sum_{F \subsetneq P}(-1)^{S(F,P)+|P|+|F|+1}\T(g^{|F|}x_{F}) (1_A \otimes x_{P \setminus F})+\mathfrak{U}_P$
\\[0.5em] for every $P \subseteq \lbrace 1, \ldots, n\rbrace$, where
\[
\mathfrak{U}_P=\sum_{k,Q}\sum_{F \subseteq Q}  \eta(P;k,Q)(-1)^{S(Q \setminus F,Q)+(|P|+|Q|+k)|F|} G^k X_{Q \setminus F} \otimes g^{|P|+|Q|+k}x_F
\]
and $\eta(P;k,Q) \in \Bbbk$ satisfy \eqref{eq:etaG} and \eqref{eq:eta1.1}-\eqref{eq:eta0.2} for $i=1, \ldots, n$ and every $Q \not \ni i$.
\end{itemize}
\end{theorem}

\begin{corollary}\label{cor:totint}
The map $\T$ is a total integral on $(E(n),\psi)$ if and only if in addition $\eta(\emptyset;0,\emptyset)=1$ and $\eta(\emptyset;k,Q)=0$ for $(k,Q) \neq (0,\emptyset)$.
\end{corollary}
\begin{proof}
By definition, an integral on $(E(n),\psi)$ is total if and only if $\T(1_{E(n)})=1_A \otimes 1_{E(n)}$. Given that, by the previous theorem,
\[
\T(1_{E(n)})=\mathfrak{U}_{\emptyset}=\sum_{k,Q}\sum_{F \subseteq Q}  \eta(\emptyset;k,Q)(-1)^{S(Q \setminus F,Q)+(|Q|+k)|F|} G^k X_{Q \setminus F} \otimes g^{|Q|+k}x_F,
\]
we can easily conclude.
\end{proof}

\begin{remark}
In view of Theorem \ref{thm:maincas}, the above results are equivalent to a complete characterization of (normalized) Casimir morphisms for the coalgebra $(E(n),\psi)$. The system of equations presented above admit (at least) one solution, by Proposition \ref{prop:principal} (see Remarks \ref{rmk:principal1} and \ref{rmk:principal2}).
\end{remark}

\subsection{Integrals satisfying the h-property}

We can also determine what integrals $\T$ share the h-property \eqref{eq:hprop2}. The conditions to check are listed in the following proposition.
\begin{proposition}\label{prop:h-prop}
Let $\T:E(n) \rightarrow A \otimes E(n)^{\op}$ be a total integral on $(E(n),\psi)$. Then $\T$ satisfies the h-property if and only if
\begin{align}
&(\mathrm{Id}_A \otimes \varepsilon)(\T(g)\T(g))=1_A, \label{eq:Tgg}\\
&(\mathrm{Id}_A \otimes \varepsilon)(\T(x_P)\T(g))=(\mathrm{Id}_A \otimes \varepsilon)\T(gx_P),\label{eq:TgxP}
\end{align}
for every $x_P$ with $P \subseteq \lbrace 1, \ldots, n \rbrace$.
\end{proposition}

\begin{proof}
First of all, we will show that the h-property \eqref{eq:hprop2} is equivalent to $\T$ verifying
\begin{align}
&(\mathrm{Id}_A \otimes \varepsilon)(\T(g^jx_P)\T(g))=(\mathrm{Id}_A \otimes \varepsilon)\T(g^{j+1}x_P),\label{eq:hprop_g}\\
&(\mathrm{Id}_A \otimes \varepsilon)(\T(g^jx_P)\T(x_i))=(-1)^{j+|P|}(\mathrm{Id}_A \otimes \varepsilon)\T(g^jx_Px_i),\label{eq:hprop_x_i}
\end{align}
for every element $g^jx_P \in \mathbb{B}_H$ and every $x_i$ with $i=1, \ldots, n$.
Clearly, if $\T$ satisfies the h-property, then \eqref{eq:hprop_g} and \eqref{eq:hprop_x_i} are satisfied. Conversely, suppose \eqref{eq:hprop_g} and \eqref{eq:hprop_x_i} hold true. Then it is easy to see that \begin{equation}\label{eq:aux64}
(\mathrm{Id}_A \otimes \varepsilon)\T(g^kx_Q)=(\mathrm{Id}_A \otimes \varepsilon)\T(x_{i_r})\cdots (\mathrm{Id}_A \otimes \varepsilon)\T(x_{i_1})(\mathrm{Id}_A \otimes \varepsilon)\T(g^k),
\end{equation}
where $g^kx_Q=g^jx_{i_1}\cdots x_{i_r}$. Hence
\begin{align*}
&(\mathrm{Id}_A \otimes \varepsilon)(\T(g^jx_P)\T(g^kx_Q))\\
&=(\mathrm{Id}_A \otimes \varepsilon)\T(g^jx_P)(\mathrm{Id}_A \otimes \varepsilon)\T(g^kx_Q)\\
\overset{\eqref{eq:aux64}}&{=}(\mathrm{Id}_A \otimes \varepsilon)\T(g^jx_P)(\mathrm{Id}_A \otimes \varepsilon)\T(x_{i_r})\cdots (\mathrm{Id}_A \otimes \varepsilon)\T(x_{i_1})(\mathrm{Id}_A \otimes \varepsilon)\T(g^k)\\
\overset{\eqref{eq:hprop_x_i}}&{=}(\mathrm{Id}_A \otimes \varepsilon)\T(x_{i_r}g^jx_P)(\mathrm{Id}_A \otimes \varepsilon)\T(x_{i_{r-1}})\cdots (\mathrm{Id}_A \otimes \varepsilon)\T(x_{i_1})(\mathrm{Id}_A \otimes \varepsilon)\T(g^k)\\
&\ldots\\
\overset{\eqref{eq:hprop_x_i}}&{=}(\mathrm{Id}_A \otimes \varepsilon)\T(x_{i_1}\cdots x_{i_r}g^jx_P)(\mathrm{Id}_A \otimes \varepsilon)\T(g^k)\\
\overset{\eqref{eq:hprop_g}}&{=}(\mathrm{Id}_A \otimes \varepsilon)\T(g^kx_{i_1}\cdots x_{i_r}g^jx_P)\\
&=(\mathrm{Id}_A \otimes \varepsilon)\T(g^kx_Qg^jx_P).
\end{align*}
By linearity of $\T$, we can conclude that \eqref{eq:hprop2} holds for every $h,h' \in H$.

Now we observe that \eqref{eq:hprop_g} and \eqref{eq:hprop_x_i} for $P=\emptyset$ are equivalent to \eqref{eq:Tgg} and 
\begin{equation}
(\mathrm{Id}_A \otimes \varepsilon)(\T(g)\T(x_i))=-(\mathrm{Id}_A \otimes \varepsilon)\T(gx_i) \label{eq:trivial3},\\
\end{equation}
while for every other $P$ we find \eqref{eq:TgxP} and
\begin{align}
&(\mathrm{Id}_A \otimes \varepsilon)(\T(x_P)\T(x_i))=(-1)^{|P|}(\mathrm{Id}_A \otimes \varepsilon)\T(x_Px_i),\label{eq:trivial4}\\
&(\mathrm{Id}_A \otimes \varepsilon)(\T(gx_P)\T(g))=(\mathrm{Id}_A \otimes \varepsilon)\T(x_P), \label{eq:trivial1}\\
&(\mathrm{Id}_A \otimes \varepsilon)(\T(gx_P)\T(x_i))=(-1)^{|P|+1}(\mathrm{Id}_A \otimes \varepsilon)\T(gx_Px_i).\label{eq:trivial2}
\end{align}
To conclude, it is enough to show that \eqref{eq:trivial3}-\eqref{eq:trivial2} can all be deduced from \eqref{eq:Tgg} and \eqref{eq:TgxP}.
Indeed, Theorem \ref{thm:mainint} implies that
\begin{equation}\label{eq:aux77}
(\mathrm{Id}_A \otimes \varepsilon)\T(x_i)= \frac{1}{2}\left[(\mathrm{Id}_A \otimes \varepsilon)\T(g)X_i-X_i (\mathrm{Id}_A \otimes \varepsilon)\T(g)\right],
\end{equation}
therefore
\begin{align*}
(\mathrm{Id}_A \otimes \varepsilon)(\T(g)\T(x_i))\overset{\eqref{eq:aux77}}&{=}(\mathrm{Id}_A \otimes \varepsilon)\T(g)\cdot \frac{1}{2}\left[(\mathrm{Id}_A \otimes \varepsilon)\T(g)X_i-X_i(\mathrm{Id}_A \otimes \varepsilon)\T(g)\right]\\
\overset{\eqref{eq:Tgg}}&{=} \frac{1}{2}\left[X_i(\mathrm{Id}_A \otimes \varepsilon)\T(g)-(\mathrm{Id}_A \otimes \varepsilon)\T(g) X_i\right] \cdot (\mathrm{Id}_A \otimes \varepsilon)\T(g)\\
\overset{\eqref{eq:aux77}}&{=}-(\mathrm{Id}_A \otimes \varepsilon)(\T(x_i)\T(g))\\
\overset{\eqref{eq:TgxP}}&{=}-(\mathrm{Id}_A \otimes \varepsilon)\T(gx_i),
\end{align*}
and
\[
(\mathrm{Id}_A \otimes \varepsilon)(\T(gx_P)\T(g))\overset{\eqref{eq:TgxP}}{=}(\mathrm{Id}_A \otimes \varepsilon)(\T(x_P)\T(g)\T(g))\overset{\eqref{eq:Tgg}}{=}(\mathrm{Id}_A \otimes \varepsilon)\T(x_P),
\]
which means that \eqref{eq:trivial3} and \eqref{eq:trivial1} follow directly from \eqref{eq:Tgg} and \eqref{eq:TgxP}.
To prove that the same is true for \eqref{eq:trivial4} we need to write \eqref{eq:int} for $a=X_i$ and $h=x_P,gx_P$ and apply $(\mathrm{Id}_A \otimes \varepsilon)$ on both sides of these equalities:
\begin{align}
&(-1)^{|P|}X_i(\mathrm{Id}_A \otimes \varepsilon)\T(x_P)+[1+(-1)^{|P|+1}](\mathrm{Id}_A \otimes \varepsilon)\T(gx_Px_i)=(\mathrm{Id}_A \otimes \varepsilon)\T(x_P)X_i,\label{eq:aux73}\\
&(-1)^{|P|}X_i(\mathrm{Id}_A \otimes \varepsilon)\T(gx_P)+[1+(-1)^{|P|}](\mathrm{Id}_A \otimes \varepsilon)\T(x_Px_i)=(\mathrm{Id}_A \otimes \varepsilon)\T(gx_P)X_i. \label{eq:aux74}
\end{align}
Then we observe that
\begin{align*}
&(\mathrm{Id}_A \otimes \varepsilon)(\T(x_P)\T(x_i))\\
\overset{\eqref{eq:aux77}}&{=}(\mathrm{Id}_A \otimes \varepsilon)\T(x_P)\cdot \frac{1}{2}\left[(\mathrm{Id}_A \otimes \varepsilon)\T(g)X_i-X_i(\mathrm{Id}_A \otimes \varepsilon)\T(g)\right]\\
\overset{\eqref{eq:TgxP}}&{=}\frac{1}{2}\left[(\mathrm{Id}_A \otimes \varepsilon)\T(gx_P)X_i-(\mathrm{Id}_A \otimes \varepsilon)\T(x_P)X_i(\mathrm{Id}_A \otimes \varepsilon)\T(g)\right]\\
\overset{\eqref{eq:aux73},\eqref{eq:aux74}}&{=}\frac{1}{2}\left[(-1)^{|P|}X_i(\mathrm{Id}_A \otimes \varepsilon)\T(gx_P)+[1+(-1)^{|P|}](\mathrm{Id}_A \otimes \varepsilon)\T(x_Px_i)+\right.\\
&\left.-(-1)^{|P|}X_i(\mathrm{Id}_A \otimes \varepsilon)(\T(x_P)\T(g))-[1+(-1)^{|P|+1}](\mathrm{Id}_A \otimes \varepsilon)(\T(gx_Px_i)\T(g))\right]\\
\overset{\eqref{eq:TgxP}}&{=}\frac{1}{2}\left[[1+(-1)^{|P|}](\mathrm{Id}_A \otimes \varepsilon)\T(x_Px_i)-[1+(-1)^{|P|+1}](\mathrm{Id}_A \otimes \varepsilon)(\T(gx_Px_i)\T(g))\right]\\
\end{align*}
If $i \in P$, then $x_Px_i=0$ and
\[(\mathrm{Id}_A \otimes \varepsilon)(\T(x_P)\T(x_i))=0=(-1)^{|P|}(\mathrm{Id}_A \otimes \varepsilon)\T(x_Px_i).
\]
If instead $i \notin P$, we can write
\begin{align*}
&(\mathrm{Id}_A \otimes \varepsilon)(\T(x_P)\T(x_i))\\
&=\frac{1}{2}\left[[1+(-1)^{|P|}](\mathrm{Id}_A \otimes \varepsilon)\T(x_Px_i)-[1+(-1)^{|P|+1}](\mathrm{Id}_A \otimes \varepsilon)(\T(gx_Px_i)\T(g))\right]\\
\overset{\eqref{reorder}}&{=}\frac{(-1)^{S(P,P\sqcup \lbrace i \rbrace)}}{2}\left[[1+(-1)^{|P|}](\mathrm{Id}_A \otimes \varepsilon)\T(x_{P \sqcup \lbrace i \rbrace})-[1+(-1)^{|P|+1}](\mathrm{Id}_A \otimes \varepsilon)(\T(gx_{P \sqcup \lbrace i \rbrace})\T(g))\right]\\
\overset{\eqref{eq:trivial1}}&{=}(-1)^{S(P,P\sqcup \lbrace i \rbrace)}(-1)^{|P|}(\mathrm{Id}_A \otimes \varepsilon)\T(x_{P \sqcup \lbrace i \rbrace})\\
\overset{\eqref{reorder}}&{=}(-1)^{|P|}(\mathrm{Id}_A \otimes \varepsilon)\T(x_Px_i),
\end{align*}
thus also \eqref{eq:trivial4} follows from \eqref{eq:Tgg} and \eqref{eq:TgxP}.
Finally
\begin{align*}
(\mathrm{Id}_A \otimes \varepsilon)(\T(gx_P)\T(x_i)) \overset{\eqref{eq:TgxP}}&{=}
(\mathrm{Id}_A \otimes \varepsilon)(\T(x_P)\T(g)\T(x_i))\\
\overset{\eqref{eq:trivial3}}&{=}-(\mathrm{Id}_A \otimes \varepsilon)(\T(x_P)\T(gx_i))\\
\overset{\eqref{eq:TgxP}}&{=}-(\mathrm{Id}_A \otimes \varepsilon)(\T(x_P)\T(x_i)\T(g))\\
\overset{\eqref{eq:trivial4}}&{=}(-1)^{|P|+1}(\mathrm{Id}_A \otimes \varepsilon)(\T(x_Px_i)\T(g)),
\end{align*}
from which we can deduce that \eqref{eq:trivial2} is a consequence of \eqref{eq:Tgg} and \eqref{eq:TgxP}. In fact, if $i \in P$, then $(-1)^{|P|+1}(\mathrm{Id}_A \otimes \varepsilon)(\T(x_Px_i)\T(g))=0=(-1)^{|P|+1}(\mathrm{Id}_A \otimes \varepsilon)\T(gx_Px_i)$. On the other hand, if $i \notin P$, it follows that
\begin{align*}
(-1)^{|P|+1}(\mathrm{Id}_A \otimes \varepsilon)(\T(x_Px_i)\T(g))\overset{\eqref{reorder}}&{=}(-1)^{|P|+1+S(P,P \sqcup \lbrace i \rbrace)}(\mathrm{Id}_A \otimes \varepsilon)(\T(x_{P \sqcup \lbrace i \rbrace})\T(g))\\
\overset{\eqref{eq:TgxP}}&{=}(-1)^{|P|+1+S(P,P \sqcup \lbrace i \rbrace)}(\mathrm{Id}_A \otimes \varepsilon)\T(gx_{P \sqcup \lbrace i \rbrace})\\
\overset{\eqref{reorder}}&{=}(-1)^{|P|+1}(\mathrm{Id}_A \otimes \varepsilon)\T(gx_Px_i).
\end{align*}
\end{proof}

\begin{remark}
Equations \eqref{eq:Tgg} and \eqref{eq:TgxP} admit (at least) one solution, given by the canonical total integral $\underline{\T}$, in view of Proposition \ref{Prop:hsepcow}.
\end{remark}

We can immediately notice that the solutions of \eqref{eq:Tgg} are exactly the square roots of $1_A$ in $A=Cl(\alpha,\beta_i,\gamma_i, \lambda_{ij})$, although to identify these elements in full generality is a hard task. While separability conditions \eqref{eq:mu-0P}, \eqref{eq:mu-1P}, \eqref{eq:etaG} and \eqref{eq:eta1.1}-\eqref{eq:eta0.2} are linear equations in the coefficients $\mu_{j,P}$ and $\eta(P;k,Q)$, h-separability equations \eqref{eq:Tgg} and \eqref{eq:TgxP} are quadratic and their explicit form on coefficients is not easily determined in general. This is mostly due to the fact that writing the coordinates of a product $G^uX_{U}G^vX_{V}$ on the elements of the basis $\mathbb{B}_A$ of $A$ requires an iterated use of \cite[Lem. 3.8]{FR3} and calculations becomes chaotic pretty quickly. Indeed, it seems that the number and the form of square roots of $1_A$ depend heavily on the choice of the defining constants $\alpha,\beta_i,\gamma_i, \lambda_{ij}$ and the ground field $\Bbbk$. We include a couple of concrete examples in hope that they can better suggest the diversity that runs through the whole series of cases. 
\begin{example}
Consider equation \eqref{eq:Tgg} when $n=1$. An element $u=u_11_A+u_2G+u_3X+u_4GX \in Cl(\alpha, \beta,\gamma)$ squares to $1_A$ if and only if
\[
\begin{cases}
u_1^2+\alpha u_2^2+\beta u_3^2+\gamma  u_2 u_3-\alpha \beta  u_4^2=1\\ 
u_2(2 u_1+\gamma u_4)=0\\ 
u_3(2 u_1+\gamma u_4)=0\\ 
u_4(2 u_1+\gamma u_4)=0,
\end{cases}
\] 
from which we deduce that $u=\pm 1_A$ or $u \in \left\lbrace \frac{\gamma^2}{4}u_4^2+\alpha u_2^2+\beta u_3^2+\gamma  u_2 u_3-\alpha \beta  u_4^2=1 \right \rbrace$. Let us call the elements of the second family the \emph{non-trivial} roots of $1_A$. When $\alpha=\beta=-1$ and $\gamma=0$ there are no non-trivial roots of $1_A$ if $\Bbbk=\mathbb{R}$, and there is an infinitude of them if $\Bbbk=\mathbb{C}$. On the other hand, when $\alpha=\beta=1$ and $\gamma=0$ there are infinitely many non-trivial roots of $1_A$ both when $\Bbbk=\mathbb{R}$ and when $\Bbbk=\mathbb{C}$. Finally, if $\alpha=\beta=\gamma=0$ the set of non-trivial roots is empty, independently of the choice of $\Bbbk$. 
\end{example}
\begin{example}
When $n=2$, the Clifford algebra $A=Cl(\alpha, \beta_1,\beta_2, \gamma_1, \gamma_2, \lambda)$ contains three subalgebras isomorphic to the previous example, namely the one generated by $G$ and $X_1$, the one generated by $G$ and $X_2$, and the one generated by $X_1$ and $X_2$. This immediately implies that among the square roots of $1_A$ appear the three sets
\begin{align*}
&\left\lbrace \frac{\gamma_1^2}{4}u_4^2+\alpha u_2^2+\beta_1 u_3^2+\gamma_1  u_2 u_3-\alpha \beta_1  u_4^2=1 \right \rbrace,\\
&\left\lbrace \frac{\gamma_2^2}{4}u_6^2+\alpha u_2^2+\beta_2 u_5^2+\gamma_2  u_2 u_5-\alpha \beta_2  u_6^2=1 \right \rbrace,\\
&\left\lbrace \frac{\lambda^2}{4}u_7^2+\beta_1 u_3^2+\beta_2 u_5^2+\lambda  u_3 u_5-\beta_1 \beta_2  u_7^2=1 \right \rbrace,
\end{align*}
which, however, do not include every possible occurrence. For example, when $\alpha=0$ and $\gamma_1,\gamma_2 \neq 0$ the element $u=-1_A +\frac{2}{\gamma_1+\gamma_2}GX_1+\frac{2}{\gamma_1+\gamma_2}GX_2$  squares to $1_A$ but do not belong in any of the previous families if $\gamma_1 \neq -\frac{\gamma_2}{2}$ and $\gamma_2 \neq -\frac{\gamma_1}{2}$.
\end{example}

\section{Classification results}\label{sec:classification}

In this section we present the explicit classifications of integrals on $(E(n),\psi)$ we obtain from Theorem \ref{thm:mainint} when $n=1,2$. These can be compared with the partial results displayed in \cite{mt2,mt3}. We also show how to recover \cite[Thm. 5.7]{FR3} that contains a classification of Casimir morphisms of the form $B=B^A \otimes u_H$ for $(E(n),\psi)$.

\subsection{Integrals on \texorpdfstring{$(H_4,\psi)$}{} when \texorpdfstring{$A=Cl(\alpha,\beta, \gamma)$}{}}
Remember that $E(1)=H_4$, Sweedler's Hopf algebra. We use Theorem \ref{thm:mainint} to obtain a complete classification of integrals on $(H_4, \psi)$. We find
\begin{align*}
\T(g)&=\mu_{0,\emptyset}(1_A \otimes g)+\mu_{1,\emptyset}(G \otimes 1_{H_4})+\mu_{1,\lbrace 1 \rbrace}(GX \otimes g)- \mu_{1,\lbrace 1 \rbrace}(G \otimes gx),\\
\T(x)&=-\frac{\gamma}{2}\mu_{1,\emptyset} (1_A \otimes 1_{H_4})+\mu_{1,\emptyset}(GX \otimes 1_{H_4})+\beta\mu_{1,\lbrace 1 \rbrace}(G \otimes g)-\mu_{0,\emptyset}(1_A \otimes gx)+\\
&-\mu_{1,\lbrace 1 \rbrace}(GX \otimes gx),
\end{align*}
with $\mu_{1,\lbrace 1 \rbrace}\alpha=0$, $\mu_{1, \lbrace 1 \rbrace}\gamma=0$, and
\begin{align*}
\T(1_{H_4})&=\eta(\emptyset;0,\emptyset)1_A\otimes 1_{H_4}+\eta(\emptyset;1,\lbrace 1 \rbrace)GX \otimes 1_{H_4}+\eta(\emptyset;1,\lbrace 1 \rbrace) G \otimes x,\\
\T(gx)&=\left[\eta(\emptyset;0,\emptyset)+\eta(\lbrace 1 \rbrace;0,\lbrace 1 \rbrace) \right]1_A \otimes x+\eta(\emptyset;1,\lbrace 1 \rbrace)GX \otimes x-\frac{\gamma}{2}\eta(\lbrace 1 \rbrace;1,\lbrace 1 \rbrace) 1_A \otimes g+\\
&+\eta(\lbrace 1 \rbrace;1,\emptyset) G \otimes 1_{H_4}+\eta(\lbrace 1 \rbrace;0,\lbrace 1 \rbrace)X \otimes 1_{H_4}+\eta(\lbrace 1 \rbrace;1,\lbrace 1 \rbrace)GX \otimes g+\\
&-\eta(\lbrace 1 \rbrace;1,\lbrace 1 \rbrace) G \otimes gx,
\end{align*}
with $\eta(\emptyset;1, \lbrace 1 \rbrace)\alpha=0$, $\eta(\emptyset;1,\lbrace 1 \rbrace)\beta=0$, $\eta(\emptyset;1, \lbrace 1 \rbrace)\gamma=0$ and
\[
\begin{cases}
\eta(\lbrace 1 \rbrace;0,\lbrace 1 \rbrace)\gamma +2\eta(\lbrace 1 \rbrace;1,\emptyset)\alpha=0,\\
\eta(\lbrace 1 \rbrace;1,\emptyset)\gamma+2\eta(\lbrace 1 \rbrace;0, \lbrace 1 \rbrace)\beta=0.
\end{cases}
\]
We perform a relabelling, in order to compare our result with the one previously obtained in \cite{mt2}. Let $\textbf{U}:=\eta(\emptyset;0,\emptyset)$, $\textbf{V}:=\eta(\emptyset;1,\lbrace 1 \rbrace)$, $\textbf{A}:=\mu_{0,\emptyset}$, $\textbf{B}:=\mu_{1,\emptyset}$, $\textbf{Z}:=\mu_{1,\lbrace 1 \rbrace}$, $\textbf{C}:=-\eta(\lbrace 1 \rbrace;1,\emptyset)$, $\textbf{D}:=\eta(\lbrace 1 \rbrace;0,\lbrace 1 \rbrace)+\eta(\emptyset;0,\emptyset)$, $\textbf{E}:=-\eta(\lbrace 1 \rbrace; 1, \lbrace 1 \rbrace)$.

\begin{theorem}
The elements of $\int_{(H_4,\psi)}$ are $\Bbbk$-linear maps $\T:H_4 \rightarrow A \otimes H_4^{\op}$ defined by
\begin{align*}
\T(1_{H_4})&=\textbf{U}(1_A\otimes 1_{H_4})+\textbf{V}(GX \otimes 1_{H_4})+\textbf{V}(G \otimes x),\\
\T(g)&=\textbf{A}(1_A \otimes g)+\textbf{B}(G \otimes 1_{H_4})+\textbf{Z}(GX \otimes g)-\textbf{Z}(G \otimes gx),\\
\T(x)&=-\frac{\gamma}{2}\textbf{B}(1_A \otimes 1_{H_4})+\textbf{B}(GX \otimes 1_{H_4})+\beta\textbf{Z}(G \otimes g)-\textbf{A}(1_A \otimes gx)-\textbf{Z}(GX \otimes gx),\\
\T(gx)&=\textbf{D}(1_A \otimes x)+\textbf{V}(GX \otimes x)+\frac{\gamma}{2}\textbf{E}(1_A \otimes g)-\textbf{C}(G \otimes 1_{H_4})+[\textbf{D}-\textbf{U}](X \otimes 1_{H_4})+\\
&-\textbf{E}(GX \otimes g)+\textbf{E}(G \otimes gx),
\end{align*}
with $\textbf{A},\textbf{B},\textbf{C},\textbf{D},\textbf{E},\textbf{U}, \textbf{V},\textbf{Z} \in \Bbbk$ satisfying
\[
\begin{cases}
\textbf{Z}\alpha=\textbf{Z}\gamma=0,\\
\textbf{V}\alpha=\textbf{V}\beta=\textbf{V}\gamma=0,\\
(\textbf{D}-\textbf{U})\gamma -2\textbf{C}\alpha=0,\\
2(\textbf{D}-\textbf{U})\beta-\textbf{C}\gamma=0.
\end{cases}
\]
An integral of this kind is total if and only if $\textbf{U}=1$ and $\textbf{V}=0$.
\end{theorem} 

Remember that, by Theorem \ref{thm:maincas}, to each total integral $\T$ corresponds a normalized Casimir morphism $B_{\T}$ for $(H_4, \psi)$ and that $B_{\T}(h \otimes 1_{H_4})=\T(h)$ for every $h \in H_4$. Then such a morphism must be defined by
\begin{align*}
B_{\T}(1_{H_4} \otimes g)\overset{\eqref{eq:B(hx1)}}&{=}(1_A \otimes g)B_{\T}(g \otimes 1_{H_4})(1_A \otimes g)=(1_A \otimes g)\T(g)(1_A \otimes g)\\
&=\textbf{A}(1_A \otimes g)+\textbf{B}(G \otimes 1_{H_4})+\textbf{Z}(GX \otimes g)+\textbf{Z}(G \otimes gx),
\end{align*}
\begin{align*}
B_{\T}(x \otimes 1_{H_4})&=\T(x)\\
&=-\frac{\gamma}{2}\textbf{B}(1_A \otimes 1_{H_4})+\textbf{B}(GX \otimes 1_{H_4})+\beta\textbf{Z}(G \otimes g)-\textbf{A}(1_A \otimes gx)-\textbf{Z}(GX \otimes gx),
\end{align*}
\begin{align*}
B_{\T}(1_A \otimes x)\overset{\eqref{eq:B(hx1)}}&{=}-(1_A \otimes g)B_{\T}(gx \otimes 1_{H_4})(1_A \otimes g)=-(1_A \otimes g)\T(gx)(1_A \otimes g)\\
&=\textbf{D}(1_A \otimes x)-\frac{\gamma}{2}\textbf{E}(1_A \otimes g)+\textbf{C}(G \otimes 1_{H_4})+\\
&+[1-\textbf{D}](X \otimes 1_{H_4})+\textbf{E}(GX \otimes g)+\textbf{E}(G \otimes gx),
\end{align*}
with $\textbf{A},\textbf{B},\textbf{C},\textbf{D},\textbf{E},\textbf{Z} \in \Bbbk$ satisfying
\begin{equation}\label{sys1}
\begin{cases}
\textbf{Z}\alpha=\textbf{Z}\gamma=0,\\
(\textbf{D}-1)\gamma -2\textbf{C}\alpha=0,\\
2(\textbf{D}-1)\beta-\textbf{C}\gamma=0.
\end{cases}
\end{equation}
Now it is immediate to check that the above description complete the partial classification of normalized Casimir morphisms for $(H_4,\psi)$ presented by Menini and Torrecillas in \cite[Thm. 1]{mt2}, where the coefficient $\textbf{Z}$ vanishes.

\subsubsection{Integrals satisfying the h-property}
Thanks to Proposition \ref{prop:h-prop} we can also determine for what values of $\textbf{A},\textbf{B},\textbf{C},\textbf{D},\textbf{E},\textbf{Z}\in \Bbbk$ a total integral $\T$ satisfies the h-property. 
\begin{theorem}
An integral on $(H_4, \psi)$ satisfies the h-property if and only if either
\begin{enumerate}[1)]
\itemsep0.5em
\item $\textbf{A}=1$, $\textbf{B}=0$, $\textbf{C}=0$, $\textbf{D}=1$, $\textbf{E}=0$, $\textbf{Z}=0$ (the canonical total integral $\underline{\T}$).
\item $\textbf{A}=-1$, $\textbf{B}=0$, $\textbf{C}=0$, $\textbf{D}=1$, $\textbf{E}=0$, $\textbf{Z}=0$.
\item $\textbf{A}=0$, $\alpha\textbf{B}^2=1$, $\textbf{C}=-\frac{\gamma}{2\alpha}$, $\textbf{D}=0$, $\textbf{E}=0$, $\textbf{Z}=0$, assuming $\gamma^2-4\alpha\beta=0$ and $\alpha \in (\Bbbk^\times)^2$.
\end{enumerate}
\end{theorem}
\begin{proof}
Equality \eqref{eq:Tgg} becomes
\[(\textbf{A}^2+\alpha\textbf{B}^2-\alpha\beta \textbf{Z}^2)1_A +\textbf{B}(2\textbf{A}+\gamma\textbf{Z})G+\textbf{Z}(2\textbf{A}+\gamma \textbf{Z})GX=(\textbf{A}1_A +\textbf{B}G+\textbf{Z}GX)^2=1_A,\]
thus, considering \eqref{sys1}, we must have
\[
\begin{cases}
\textbf{A}^2+\alpha\textbf{B}^2=1,\\
\textbf{A}\textbf{B}=0,\\
\textbf{A}\textbf{Z}=0,\\
\textbf{Z}\alpha=\textbf{Z}\gamma=0,\\
(\textbf{D}-1)\gamma -2\textbf{C}\alpha=0,\\
2(\textbf{D}-1)\beta-\textbf{C}\gamma=0.
\end{cases}
\]
The family of equalities \eqref{eq:TgxP} reduces to
\[(-\frac{\gamma}{2}\textbf{B}1_A+\beta\textbf{Z}G+\textbf{B}GX)(\textbf{A}1_A +\textbf{B}G+\textbf{Z}GX)=\frac{\gamma}{2}\textbf{E}1_A -\textbf{C}G+(\textbf{D}-1)X-\textbf{E}GX,\]
i.e.
\[
\begin{cases}
\frac{\gamma}{2}\textbf{B}^2+\beta \textbf{A}\textbf{Z}=-\textbf{C},\\
\textbf{A}\textbf{B}=-\textbf{E},\\
\alpha\textbf{B}^2=1-\textbf{D},
\end{cases}
\]
in view of \eqref{sys1}. Putting everything altogether we find the solutions described in the statement.
\end{proof}
\begin{remark}
These values completely agree with the ones already found in the proof of \cite[Thm. 2]{mt2}.
\end{remark}

\subsection{Integrals on \texorpdfstring{$(E(2),\psi)$}{} when \texorpdfstring{$A=Cl(\alpha,\beta_1,\beta_2, \gamma_1,\gamma_2, \lambda)$}{}}\label{subsec:case2}

We use Theorem \ref{thm:mainint} to obtain a complete classification of integrals on $(E(2), \psi)$. In this case $A=Cl(\alpha,\beta_1, \beta_2, \gamma_1, \gamma_2, \lambda)$ and an integral $\T:E(2) \rightarrow A \otimes E(2)^{\op}$ is defined by
\begin{equation}\label{Tgcase2}
\begin{split}
\T(g)&=\mu_{0,\emptyset}(1_A \otimes g)+\frac{\gamma_2}{2}\mu_{1, \lbrace 1,2 \rbrace}(1_A \otimes x_1)-\frac{\gamma_1}{2}\mu_{1, \lbrace 1,2 \rbrace}(1_A \otimes x_2)+\mu_{0,\lbrace 1,2 \rbrace}(1_A \otimes gx_1x_2)+\\
&+\mu_{1,\emptyset}(G \otimes 1_{E(2)})-\mu_{1,\lbrace 1 \rbrace}(G \otimes gx_1)-\mu_{1,\lbrace 2 \rbrace}(G \otimes gx_2)+\mu_{1,\lbrace 1,2 \rbrace}(G \otimes x_1x_2)+\\
&+\frac{\gamma_2}{2}\mu_{1, \lbrace 1,2 \rbrace}(X_1 \otimes 1_{E(2)})-\mu_{0,\lbrace 1,2 \rbrace}(X_1 \otimes gx_2)+\mu_{1,\lbrace 1 \rbrace}(GX_1 \otimes g)+\mu_{1,\lbrace 1,2 \rbrace}(GX_1 \otimes x_2)+\\
&-\frac{\gamma_1}{2}\mu_{1, \lbrace 1,2 \rbrace}(X_2 \otimes 1_{E(2)})+\mu_{0,\lbrace 1,2 \rbrace}(X_2 \otimes gx_1)+\mu_{1,\lbrace 2 \rbrace}(GX_2 \otimes g)-\mu_{1,\lbrace 1,2 \rbrace}(GX_2 \otimes x_1)+\\
&+\mu_{0,\lbrace 1,2 \rbrace}(X_1X_2 \otimes g)+\mu_{1,\lbrace 1,2 \rbrace}(GX_1X_2 \otimes 1_{E(2)}),
\end{split}
\end{equation}
\begin{equation}\label{Tx1case2}
\begin{split}
\T(x_1)&= -\frac{\gamma_1}{2}\left(\mu_{1,\emptyset}+\frac{\lambda}{2}\mu_{1, \lbrace 1,2 \rbrace}\right)(1_A \otimes 1_{E(2)})+\left(\frac{\gamma_1}{2}\mu_{1,\lbrace 1 \rbrace}-\frac{\lambda}{2}\mu_{0,\lbrace 1,2 \rbrace}-\mu_{0,\emptyset}\right)(1_A \otimes gx_1)+\\
&+\left(\frac{\gamma_1}{2}\mu_{1,\lbrace 2 \rbrace}+\beta_1\mu_{0,\lbrace 1,2 \rbrace}\right)(1_A \otimes gx_2)-\frac{\gamma_1}{2}\mu_{1,\lbrace 1,2 \rbrace}(1_A \otimes x_1x_2)+\\
&+\left(\frac{\lambda}{2}\mu_{1,\lbrace 2 \rbrace}+\beta_1\mu_{1,\lbrace 1 \rbrace}\right)(G \otimes g)+\frac{\lambda}{2}\mu_{1,\lbrace 1,2 \rbrace}(G \otimes x_1)-\mu_{1,\lbrace 2 \rbrace}(G \otimes gx_1x_2)+\\
&+\left(\frac{\lambda}{2}\mu_{0,\lbrace 1,2 \rbrace}-\frac{\gamma_1}{2}\mu_{1,\lbrace 1 \rbrace}\right)(X_1 \otimes g)-\frac{\gamma_2}{2}\mu_{1, \lbrace 1,2 \rbrace}(X_1 \otimes x_1)-\mu_{0,\lbrace 1,2 \rbrace}(X_1 \otimes gx_1x_2)+\\
&+\left(\mu_{1,\emptyset}+\frac{\lambda}{2}\mu_{1,\lbrace 1,2 \rbrace}\right)(GX_1 \otimes 1_{E(2)})-\mu_{1,\lbrace 1 \rbrace}(GX_1 \otimes gx_1)+\mu_{1,\lbrace 1,2 \rbrace}(GX_1 \otimes x_1x_2)+\\
&-\left(\frac{\gamma_1}{2}\mu_{1,\lbrace 2 \rbrace}+\beta_1\mu_{0,\lbrace 1,2 \rbrace}\right)(X_2 \otimes g)+\frac{\gamma_1}{2}\mu_{1,\lbrace 1,2 \rbrace}(X_2 \otimes x_1)-\mu_{1,\lbrace 2 \rbrace}(GX_2 \otimes gx_1)+\\
&-\mu_{0,\lbrace 1,2 \rbrace}(X_1X_2 \otimes gx_1)-\mu_{1,\lbrace 1,2 \rbrace}(GX_1X_2 \otimes x_1),
\end{split}
\end{equation}
\begin{equation}\label{Tx2case2}
\begin{split}
\T(x_2)&=-\frac{\gamma_2}{2}\left(\mu_{1,\emptyset}+\frac{\lambda}{2}\mu_{1, \lbrace 1,2 \rbrace}\right)(1_A \otimes 1_{E(2)})+\left(\frac{\gamma_2}{2}\mu_{1,\lbrace 1 \rbrace}-\beta_2\mu_{0,\lbrace 1,2 \rbrace}\right)(1_A \otimes gx_1)+\\
&+\left(\frac{\gamma_2}{2}\mu_{1,\lbrace 2 \rbrace}+\frac{\lambda}{2}\mu_{0,\lbrace 1,2 \rbrace}-\mu_{0,\emptyset}\right)(1_A \otimes gx_2)-\frac{\gamma_2}{2}\mu_{1,\lbrace 1,2 \rbrace}(1_A \otimes x_1x_2)+\\
&+\left(\frac{\lambda}{2}\mu_{1,\lbrace 1 \rbrace}+\beta_2\mu_{1,\lbrace 2 \rbrace}\right)(G \otimes g)+\frac{\lambda}{2}\mu_{1,\lbrace 1,2 \rbrace}(G \otimes x_2)+\mu_{1,\lbrace 1 \rbrace}(G \otimes gx_1x_2)+\\
&+\left(-\frac{\gamma_2}{2}\mu_{1,\lbrace 1 \rbrace}+\beta_2\mu_{0,\lbrace 1,2 \rbrace}\right)(X_1 \otimes g)-\frac{\gamma_2}{2}\mu_{1,\lbrace 1,2 \rbrace}(X_1 \otimes x_2)-\mu_{1,\lbrace 1 \rbrace}(GX_1 \otimes gx_2)+\\
&-\left(\frac{\gamma_2}{2}\mu_{1,\lbrace 2 \rbrace}+\frac{\lambda}{2}\mu_{0,\lbrace 1,2 \rbrace}\right)(X_2 \otimes g)+\frac{\gamma_1}{2}\mu_{1, \lbrace 1,2 \rbrace}(X_2 \otimes x_2)-\mu_{0,\lbrace 1,2 \rbrace}(X_2 \otimes gx_1x_2)+\\
&+\left(\frac{\lambda}{2}\mu_{1,\lbrace 1,2 \rbrace}+\mu_{1,\emptyset}\right)(GX_2 \otimes 1_{E(2)})-\mu_{1,\lbrace 2 \rbrace}(GX_2 \otimes gx_2)+\mu_{1,\lbrace 1,2 \rbrace}(GX_2 \otimes x_1x_2)+\\
&-\mu_{0,\lbrace 1,2 \rbrace}(X_1X_2 \otimes gx_2)-\mu_{1,\lbrace 1,2 \rbrace}(GX_1X_2 \otimes x_2),
\end{split}
\end{equation}
\begin{equation}\label{Tgx1x2case2}
\begin{split}
\T(gx_1x_2)&=\frac{1}{4}\left[\left(2\beta_1\gamma_2-\gamma_1\lambda\right)\mu_{1,\lbrace 1 \rbrace}+\left(\gamma_2\lambda-2\beta_2\gamma_1\right)\mu_{1,\lbrace 2 \rbrace}+\left(\lambda^2-4\beta_1\beta_2\right)\mu_{0,\lbrace 1,2 \rbrace}\right](1_A \otimes g)+\\
&+\frac{\gamma_2\lambda}{4}\mu_{1,\lbrace 1,2 \rbrace}(1_A \otimes x_1)+\left(-\frac{\gamma_1}{2}\mu_{1,\lbrace 1 \rbrace}+\mu_{0,\emptyset}-\frac{\gamma_2}{2}\mu_{1,\lbrace 2 \rbrace}\right)(1_A \otimes gx_1x_2)+\\
&-\frac{\lambda}{2}\left(\mu_{1,\emptyset}+\frac{\lambda}{2}\mu_{1,\lbrace 1,2 \rbrace}\right)(G \otimes 1_{E(2)})+\left(\frac{\lambda}{2}\mu_{1,\lbrace 1 \rbrace}+\beta_2\mu_{1,\lbrace 2 \rbrace}\right)(G \otimes gx_1)+\\
&-\left(\frac{\lambda}{2}\mu_{1,\lbrace 2 \rbrace}+\beta_1\mu_{1,\lbrace 1 \rbrace}\right)(G \otimes gx_2)-\frac{\lambda}{2}\mu_{1,\lbrace 1,2 \rbrace}(G \otimes x_1x_2)+\\
&+\frac{\gamma_2}{2}\left(\mu_{1,\emptyset}+\frac{\lambda}{2}\mu_{1,\lbrace 1,2 \rbrace}\right)(X_1 \otimes 1_{E(2)})+\left(\beta_2\mu_{0,\lbrace 1,2 \rbrace}-\frac{\gamma_2}{2}\mu_{1,\lbrace 1 \rbrace}\right)(X_1 \otimes gx_1)+\\
&-\left(\frac{\lambda}{2}\mu_{0,\lbrace 1,2 \rbrace}-\frac{\gamma_1}{2}\mu_{1,\lbrace 1 \rbrace}\right)(X_1 \otimes gx_2)+\frac{\gamma_2}{2}\mu_{1,\lbrace 1,2 \rbrace}(X_1 \otimes x_1x_2)+\\
&+\mu_{1,\lbrace 1 \rbrace}(GX_1 \otimes gx_1x_2)-\frac{\gamma_1}{2}\left(\mu_{1,\emptyset}+\frac{\lambda}{2}\mu_{1, \lbrace 1,2 \rbrace}\right)(X_2 \otimes 1_{E(2)})+\\
&-\left(\frac{\gamma_2}{2}\mu_{1,\lbrace 2 \rbrace}+\frac{\lambda}{2}\mu_{0,\lbrace 1,2 \rbrace}\right)(X_2 \otimes gx_1)+\left(\frac{\gamma_1}{2}\mu_{1,\lbrace 2 \rbrace}+\beta_1\mu_{0,\lbrace 1,2 \rbrace}\right)(X_2 \otimes gx_2)+\\
&-\frac{\gamma_1}{2}\mu_{1,\lbrace 1,2 \rbrace}(X_2 \otimes x_1x_2)+\frac{\lambda}{2}\mu_{1,\lbrace 1,2 \rbrace}(GX_2 \otimes x_1)+\mu_{1,\lbrace 2 \rbrace}(GX_2 \otimes gx_1x_2)+\\
&+\frac{\gamma_2}{4}\mu_{1, \lbrace 1,2 \rbrace}(X_1X_2 \otimes x_1)+\mu_{0,\lbrace 1,2 \rbrace}(X_1X_2 \otimes gx_1x_2)+\\
&+\left(\mu_{1,\emptyset}+\frac{\lambda}{2}\mu_{1,\lbrace 1,2 \rbrace}\right)(GX_1X_2 \otimes 1_{E(2)})+\mu_{1,\lbrace 1,2 \rbrace}(GX_1X_2 \otimes x_1x_2)
\end{split}
\end{equation}
with $\mu_{j,P} \in \Bbbk$ such that
\begin{equation}\label{sys:mu}
\begin{cases}
&\mu_{1, \lbrace 1 \rbrace}\gamma_1+\mu_{1, \lbrace 2 \rbrace}\gamma_2=0,\\
&2\mu_{1,\lbrace 1 \rbrace}\alpha-\mu_{0, \lbrace 1,2 \rbrace}\gamma_2=0,\\
&2\mu_{1,\lbrace 2 \rbrace}\alpha+\mu_{0, \lbrace 1,2 \rbrace}\gamma_1=0.
\end{cases}
\end{equation}
The values on even monomials are
\begin{equation}\label{T1case2}
\begin{split}
\T(1_{E(2)})&= \eta(\emptyset;0,\emptyset)(1_A \otimes 1_{E(2)})-\frac{\gamma_2}{2}\eta(\emptyset;1,\lbrace 1,2 \rbrace)(1_A \otimes gx_1)+\frac{\gamma_1}{2}\eta(\emptyset;1,\lbrace 1,2 \rbrace) (1_A \otimes gx_2)+\\
&+\eta(\emptyset;0,\lbrace 1,2 \rbrace)(1_A \otimes x_1x_2)-\frac{\lambda}{2}\eta(\emptyset;1,\lbrace 1,2 \rbrace)(G\otimes g)+\eta(\emptyset;1,\lbrace 1 \rbrace)(G \otimes x_1)+\\
&+\eta(\emptyset;1,\lbrace 2 \rbrace)( G \otimes x_2)+\eta(\emptyset;1,\lbrace 1,2 \rbrace)(G \otimes gx_1x_2)+\frac{\gamma_2}{2}\eta(\emptyset;1,\lbrace 1,2 \rbrace) (X_1 \otimes g)+\\
&+\eta(\emptyset;0,\lbrace 1,2 \rbrace)(X_1 \otimes x_2)+\eta(\emptyset;1,\lbrace 1 \rbrace)(GX_1 \otimes 1_{E(2)})-\eta(\emptyset;1,\lbrace 1,2 \rbrace)(G X_1 \otimes gx_2)+\\
&-\frac{\gamma_1}{2}\eta(\emptyset;1,\lbrace 1,2 \rbrace)(X_2 \otimes g)-\eta(\emptyset;0,\lbrace 1,2 \rbrace)(X_2 \otimes x_1)+\eta(\emptyset;1,\lbrace 2 \rbrace) (G X_2 \otimes 1_{E(2)})+\\
&+\eta(\emptyset;1,\lbrace 1,2 \rbrace)(G X_2 \otimes gx_1)+\eta(\emptyset;0,\lbrace 1,2 \rbrace) (X_1X_2\otimes 1_{E(2)})+\\
&+\eta(\emptyset;1,\lbrace 1,2 \rbrace) (G X_1X_2 \otimes g),
\end{split}
\end{equation}
\begin{equation}\label{Tgx1case2}
\begin{split}
\T(gx_1)&=\left[\eta(\emptyset;0,\emptyset)+\eta(\lbrace 1 \rbrace;0,\lbrace 1 \rbrace)\right](1_A \otimes x_1)+\eta(\lbrace 1 \rbrace;0,\lbrace 2 \rbrace)(1_A \otimes x_2)+\\
&-\frac{\gamma_1}{2}\eta(\emptyset;1,\lbrace 1,2 \rbrace) (1_A \otimes gx_1x_2)+\eta(\lbrace 1 \rbrace;1,\emptyset) (G \otimes 1_{E(2)})+\\
&-\frac{\lambda}{2}\eta(\emptyset;1,\lbrace 1,2 \rbrace)(G\otimes gx_1)+\left[\eta(\lbrace 1 \rbrace;1,\lbrace 1,2 \rbrace)-\eta(\emptyset;1,\lbrace 2 \rbrace)\right]( G \otimes x_1x_2)+\\
&+\eta(\lbrace 1 \rbrace;0,\lbrace 1 \rbrace) (X_1 \otimes 1_{E(2)})+\frac{\gamma_2}{2}\eta(\emptyset;1,\lbrace 1,2 \rbrace) (X_1 \otimes gx_1)+\\
&-\eta(\emptyset;0,\lbrace 1,2 \rbrace)(X_1 \otimes x_1x_2)+\eta(\emptyset;1,\lbrace 1 \rbrace)(GX_1 \otimes x_1)+\\
&+\eta(\lbrace 1 \rbrace;1,\lbrace 1,2 \rbrace) (GX_1 \otimes x_2)+\eta(\emptyset;1,\lbrace 1,2 \rbrace)(G X_1 \otimes gx_1x_2)+\\
&+\eta(\lbrace 1 \rbrace;0,\lbrace 2 \rbrace)(X_2 \otimes 1_{E(2)})-\frac{\gamma_1}{2}\eta(\emptyset;1,\lbrace 1,2 \rbrace)(X_2 \otimes gx_1)+\\
&+\left[\eta(\emptyset;1,\lbrace 2 \rbrace)-\eta(\lbrace 1 \rbrace;1,\lbrace 1,2 \rbrace)\right](GX_2 \otimes x_1)+\eta(\emptyset;0,\lbrace 1,2 \rbrace) (X_1X_2\otimes x_1)+\\
&+\eta(\lbrace 1 \rbrace;1,\lbrace 1,2 \rbrace) (GX_1X_2 \otimes 1_{E(2)})+\eta(\emptyset;1,\lbrace 1,2 \rbrace) (G X_1X_2 \otimes gx_1),
\end{split}
\end{equation}
\begin{equation}\label{Tgx2case2}
\begin{split}
\T(gx_2)&=\eta(\lbrace 2 \rbrace;0,\lbrace 1 \rbrace) (1_A \otimes x_1)+\left[\eta(\emptyset;0,\emptyset)+\eta(\lbrace 2 \rbrace;0,\lbrace 2 \rbrace)\right] (1_A \otimes x_2)+\\
&-\frac{\gamma_2}{2}\eta(\emptyset;1,\lbrace 1,2 \rbrace)(1_A \otimes gx_1x_2)+\eta(\lbrace 2 \rbrace;1,\emptyset) (G \otimes 1_{E(2)})+\\
&-\frac{\lambda}{2}\eta(\emptyset;1,\lbrace 1,2 \rbrace)(G\otimes gx_2)+\left[\eta(\emptyset;1,\lbrace 1 \rbrace)
+\eta(\lbrace 2 \rbrace;1,\lbrace 1,2 \rbrace) \right](G \otimes x_1x_2)+\\
&+\eta(\lbrace 2 \rbrace;0,\lbrace 1 \rbrace) (X_1 \otimes 1_{E(2)})+\frac{\gamma_2}{2}\eta(\emptyset;1,\lbrace 1,2 \rbrace) (X_1 \otimes gx_2)+\\
&+\left[\eta(\emptyset;1,\lbrace 1 \rbrace)+\eta(\lbrace 2 \rbrace;1,\lbrace 1,2 \rbrace)\right] (GX_1 \otimes x_2)+\eta(\lbrace 2 \rbrace;0,\lbrace 2 \rbrace) (X_2 \otimes 1_{E(2)})+\\
&-\frac{\gamma_1}{2}\eta(\emptyset;1,\lbrace 1,2 \rbrace)(X_2 \otimes gx_2)-\eta(\emptyset;0,\lbrace 1,2 \rbrace)(X_2 \otimes x_1x_2)+\\
&-\eta(\lbrace 2 \rbrace;1,\lbrace 1,2 \rbrace)(GX_2 \otimes x_1)+\eta(\emptyset;1,\lbrace 2 \rbrace) (G X_2 \otimes x_2)+\\
&+\eta(\emptyset;1,\lbrace 1,2 \rbrace)(G X_2 \otimes gx_1x_2)+\eta(\emptyset;0,\lbrace 1,2 \rbrace) (X_1X_2\otimes x_2)+\\
&+\eta(\lbrace 2 \rbrace;1,\lbrace 1,2 \rbrace) (GX_1X_2 \otimes 1_{E(2)})+\eta(\emptyset;1,\lbrace 1,2 \rbrace) (G X_1X_2 \otimes gx_2),
\end{split}
\end{equation}
\begin{equation}\label{Tx1x2case2}
\begin{split}
\T(x_1x_2)&=\eta(\lbrace 1, 2\rbrace;0,\emptyset) (1_A \otimes 1_{E(2)})-\frac{\gamma_2}{2}\eta(\lbrace 1,2 \rbrace;1,\lbrace 1,2 \rbrace) (1_A \otimes gx_1)+\\
&+\frac{\gamma_1}{2}\eta(\lbrace 1,2 \rbrace;1,\lbrace 1,2 \rbrace)(1_A \otimes gx_2)+\sum_{F \subseteq \lbrace 1, 2 \rbrace}\eta(F;0,F) (1_A \otimes x_1x_2)+\\
&-\frac{\lambda}{2}\eta(\lbrace 1,2 \rbrace;1,\lbrace 1,2 \rbrace) (G \otimes g)+\left[\eta(\lbrace 1, 2\rbrace;1,\lbrace 1 \rbrace)-\eta(\lbrace 2 \rbrace;1,\emptyset)\right] (G \otimes x_1)+\\
&+\left[\eta(\lbrace 1, 2\rbrace;1,\lbrace 2 \rbrace)+\eta(\lbrace 1 \rbrace;1,\emptyset)\right] (G \otimes x_2)+\\
&+\left[\eta(\lbrace 1, 2\rbrace;1,\lbrace 1,2 \rbrace)-\frac{\lambda}{2}\eta(\emptyset;1,\lbrace 1,2 \rbrace)\right](G\otimes gx_1x_2)+\\
&+\frac{\gamma_2}{2}\eta(\lbrace 1,2 \rbrace;1,\lbrace 1,2 \rbrace)(X_1 \otimes g)+\eta(\lbrace 2 \rbrace;0,\lbrace 1 \rbrace) (X_1 \otimes x_1)+\\
&+\left[\eta(\lbrace 1 \rbrace;0,\lbrace 1 \rbrace)+\eta(\lbrace 1, 2\rbrace;0,\lbrace 1,2 \rbrace)\right] (X_1 \otimes x_2)+\frac{\gamma_2}{2}\eta(\emptyset;1,\lbrace 1,2 \rbrace) (X_1 \otimes gx_1x_2)+\\
&+\eta(\lbrace 1, 2\rbrace;1,\lbrace 1 \rbrace) (GX_1 \otimes 1_{E(2)})-\eta(\lbrace 1, 2\rbrace;1,\lbrace 1,2 \rbrace) (GX_1 \otimes gx_2)+\\
&+\left[\eta(\emptyset;1,\lbrace 1 \rbrace)+\eta(\lbrace 2 \rbrace;1,\lbrace 1,2 \rbrace)\right] (GX_1 \otimes x_1x_2)-\frac{\gamma_1}{2}\eta(\lbrace 1,2 \rbrace;1,\lbrace 1,2 \rbrace) (X_2 \otimes g)+\\
&-\left[\eta(\lbrace 2 \rbrace;0,\lbrace 2 \rbrace)+\eta(\lbrace 1, 2\rbrace;0,\lbrace 1,2 \rbrace)\right](X_2 \otimes x_1)+\eta(\lbrace 1 \rbrace;0,\lbrace 2 \rbrace)(X_2 \otimes x_2)+\\
&-\frac{\gamma_1}{2}\eta(\emptyset;1,\lbrace 1,2 \rbrace)(X_2 \otimes gx_1x_2)+\eta(\lbrace 1, 2\rbrace;1,\lbrace 2 \rbrace) (GX_2 \otimes 1_{E(2)})+\\
&+\eta(\lbrace 1, 2\rbrace;1,\lbrace 1,2 \rbrace)(GX_2 \otimes gx_1)+\left[\eta(\emptyset;1,\lbrace 2 \rbrace)-\eta(\lbrace 1 \rbrace;1,\lbrace 1,2 \rbrace)\right](GX_2 \otimes x_1x_2)+\\
&+\eta(\lbrace 1, 2\rbrace;0,\lbrace 1,2 \rbrace) (X_1X_2 \otimes 1_{E(2)})+\eta(\emptyset;0,\lbrace 1,2 \rbrace) (X_1X_2\otimes x_1x_2)+\\
&-\eta(\lbrace 2 \rbrace;1,\lbrace 1,2 \rbrace) (GX_1X_2 \otimes x_1)+\eta(\lbrace 1 \rbrace;1,\lbrace 1,2 \rbrace) (GX_1X_2 \otimes x_2)+\\
&+\eta(\lbrace 1, 2\rbrace;1,\lbrace 1,2 \rbrace) (GX_1X_2 \otimes g)+\eta(\emptyset;1,\lbrace 1,2 \rbrace) (G X_1X_2 \otimes gx_1x_2),
\end{split}
\end{equation}
with $\eta(P;k,Q) \in \Bbbk$ such that
\begin{equation}\label{sys:eta1}
\begin{cases}
\gamma_1\eta(\emptyset;1,\lbrace 1 \rbrace) +\gamma_2\eta(\emptyset;1,\lbrace 2 \rbrace)=0,\\
2\alpha \eta(\emptyset;1,\lbrace 1\rbrace)-\gamma_2\eta(\emptyset;0,\lbrace 1,2 \rbrace)=0,\\
2\alpha\eta(\emptyset;1,\lbrace 2 \rbrace)+\gamma_1\eta(\emptyset;0,\lbrace 1,2 \rbrace)=0,\\
2\beta_1\eta(\emptyset;1,\lbrace 1 \rbrace)+\lambda\eta(\emptyset;1,\lbrace 2 \rbrace)=0,\\
\gamma_1\eta(\emptyset;1,\lbrace 1 \rbrace)-\lambda\eta(\emptyset;0,\lbrace 1,2 \rbrace)=0,\\
2\beta_1\eta(\emptyset;0,\lbrace 1,2 \rbrace)+\gamma_1\eta(\emptyset;1,\lbrace 2 \rbrace)=0,\\
2\beta_2\eta(\emptyset;1,\lbrace 2 \rbrace)+\lambda\eta(\emptyset;1,\lbrace 1 \rbrace)=0,\\
2\beta_2\eta(\emptyset;0,\lbrace 1,2 \rbrace)-\gamma_2\eta(\emptyset;1,\lbrace 1 \rbrace)=0,
\end{cases}
\end{equation}
\begin{equation}\label{sys:eta2}
\begin{cases}
2\alpha \eta(\lbrace i \rbrace;1,\emptyset)+\gamma_1\eta(\lbrace i \rbrace;0,\lbrace 1 \rbrace)+\gamma_2\eta(\lbrace i \rbrace;0,\lbrace 2 \rbrace)=0,\\
\gamma_2\eta(\lbrace i \rbrace;1,\lbrace 1,2 \rbrace)=0,\\
\gamma_1\eta(\lbrace i \rbrace;1,\lbrace 1,2 \rbrace)=0,\\
\alpha\eta(\lbrace i \rbrace;1,\lbrace 1,2 \rbrace)=0,\\
\lambda\eta(\lbrace i \rbrace;1,\lbrace 1,2 \rbrace)=0,\\
2\beta_1\eta(\lbrace i \rbrace;0,\lbrace 1 \rbrace)+\gamma_1\eta(\lbrace i \rbrace;1,\emptyset)+\lambda\eta(\lbrace i \rbrace;0,\lbrace 2 \rbrace)=0,\\
\beta_1\eta(\lbrace i \rbrace;1,\lbrace 1,2 \rbrace)=0,\\
2\beta_2\eta(\lbrace i \rbrace;0,\lbrace 2 \rbrace)+\gamma_2\eta(\lbrace i \rbrace;1,\emptyset)+\lambda\eta(\lbrace i \rbrace;0,\lbrace 1 \rbrace)=0,\\
\beta_2\eta(\lbrace i \rbrace;1,\lbrace 1,2 \rbrace)=0,
\end{cases}
\end{equation}
for $i=1,2$, and
\begin{equation}\label{sys:eta3}
\begin{cases}
\gamma_1\eta(\lbrace 1,2 \rbrace;1,\lbrace 1 \rbrace)+ \gamma_2\eta(\lbrace 1,2 \rbrace;1,\lbrace 2 \rbrace)=0,\\
2\alpha\eta(\lbrace 1,2 \rbrace;1,\lbrace 1\rbrace)-\gamma_2\eta(\lbrace 1,2 \rbrace;0,\lbrace 1,2 \rbrace)=0,\\
2\alpha\eta(\lbrace 1,2 \rbrace;1,\lbrace 2 \rbrace)+\gamma_1\eta(\lbrace 1,2 \rbrace;0,\lbrace 1,2 \rbrace)=0,\\
2\beta_1\eta(\lbrace 1,2 \rbrace;1,\lbrace 1 \rbrace)+\lambda\eta(\lbrace 1,2 \rbrace;1,\lbrace 2 \rbrace)=0,\\
\gamma_1\eta(\lbrace 1,2 \rbrace;1,\lbrace 1 \rbrace)-\lambda\eta(\lbrace 1,2 \rbrace;0,\lbrace 1,2 \rbrace)=0,\\
2\beta_1\eta(\lbrace 1,2 \rbrace;0,\lbrace 1,2 \rbrace)+\gamma_1\eta(\lbrace 1,2 \rbrace;1,\lbrace 2 \rbrace)=0,\\
2\beta_2\eta(\lbrace 1,2 \rbrace;1,\lbrace 2 \rbrace)+\lambda\eta(\lbrace 1,2 \rbrace;1,\lbrace 1 \rbrace)=0,\\
2\beta_2\eta(\lbrace 1,2 \rbrace;0,\lbrace 1,2 \rbrace)-\gamma_2\eta(\lbrace 1,2 \rbrace;1,\lbrace 1 \rbrace)=0.
\end{cases}
\end{equation}

\begin{theorem}
The elements of $\int_{(E(2),\psi)}$ are $\Bbbk$-linear maps $\T:E(2) \rightarrow A \otimes E(2)^{\op}$ defined by \eqref{Tgcase2}--\eqref{Tgx1x2case2} and \eqref{T1case2}--\eqref{Tx1x2case2}
with $\mu_{j,P}, \eta(P;k,Q) \in \Bbbk$ satisfying \eqref{sys:mu}, \eqref{sys:eta1}, \eqref{sys:eta2} for $i=1,2$ and \eqref{sys:eta3}.
An integral of this kind is total if and only if the coefficients $\eta(\emptyset; k,Q)$ further satisfy the equalities described in Corollary \ref{cor:totint}.
\end{theorem} 

\begin{remark}
A partial result when $n=2$ was already obtained in \cite{mt3}, under the further assumption that $\alpha,\beta_1,\beta_2, \gamma_1, \gamma_2, \lambda \neq 0$. As in \cite{mt2}, the authors work with Casimir morphisms instead of integrals.
\end{remark}

\subsubsection{Integrals satisfying the h-property}
Let $\T:E(2) \rightarrow A \otimes E(2)^{\op}$ be a total integral on $(E(2),\psi)$. To determine which integrals satisfy the h-property, we write down the values of $(\mathrm{Id}_A \otimes \varepsilon)\T(h)$ for each element of the canonical basis of $E(2)$, since they will be useful for calculations.
\begin{equation}\label{eq:epsTg}
\begin{split}
(\mathrm{Id}_A \otimes \varepsilon)\T(g)&=\mu_{0,\emptyset}1_A+\mu_{1,\emptyset}G +\frac{\gamma_2}{2}\mu_{1, \lbrace 1,2 \rbrace}X_1+\mu_{1,\lbrace 1 \rbrace}GX_1-\frac{\gamma_1}{2}\mu_{1, \lbrace 1,2 \rbrace}X_2+\\
&+\mu_{1,\lbrace 2 \rbrace}GX_2+\mu_{0,\lbrace 1,2 \rbrace}X_1X_2+\mu_{1,\lbrace 1,2 \rbrace}GX_1X_2,
\end{split}
\end{equation}
\begin{equation}\label{eq:epsTx1}
\begin{split}
(\mathrm{Id}_A \otimes \varepsilon)\T(x_1)&= -\frac{\gamma_1}{2}\left(\mu_{1,\emptyset}+\frac{\lambda}{2}\mu_{1, \lbrace 1,2 \rbrace}\right)1_A+\left(\frac{\lambda}{2}\mu_{1,\lbrace 2 \rbrace}+\beta_1\mu_{1,\lbrace 1 \rbrace}\right)G+\\
&+\left(\frac{\lambda}{2}\mu_{0,\lbrace 1,2 \rbrace}-\frac{\gamma_1}{2}\mu_{1,\lbrace 1 \rbrace}\right)X_1+\left(\mu_{1,\emptyset}+\frac{\lambda}{2}\mu_{1,\lbrace 1,2 \rbrace}\right)GX_1+\\
&-\left(\frac{\gamma_1}{2}\mu_{1,\lbrace 2 \rbrace}+\beta_1\mu_{0,\lbrace 1,2 \rbrace}\right)X_2,
\end{split}
\end{equation}
\begin{equation}\label{eq:epsTx2}
\begin{split}
(\mathrm{Id}_A \otimes \varepsilon)\T(x_2)&=-\frac{\gamma_2}{2}\left(\mu_{1,\emptyset}+\frac{\lambda}{2}\mu_{1, \lbrace 1,2 \rbrace}\right)1_A+\left(\frac{\lambda}{2}\mu_{1,\lbrace 1 \rbrace}+\beta_2\mu_{1,\lbrace 2 \rbrace}\right)G+\\
&+\left(-\frac{\gamma_2}{2}\mu_{1,\lbrace 1 \rbrace}+\beta_2\mu_{0,\lbrace 1,2 \rbrace}\right)X_1-\left(\frac{\gamma_2}{2}\mu_{1,\lbrace 2 \rbrace}+\frac{\lambda}{2}\mu_{0,\lbrace 1,2 \rbrace}\right)X_2+\\
&+\left(\frac{\lambda}{2}\mu_{1,\lbrace 1,2 \rbrace}+\mu_{1,\emptyset}\right)GX_2,
\end{split}
\end{equation}
\begin{equation}\label{eq:epsTgx1x2}
\begin{split}
(\mathrm{Id}_A \otimes \varepsilon)\T(gx_1x_2)&=\frac{1}{4}\left[\left(2\beta_1\gamma_2-\gamma_1\lambda\right)\mu_{1,\lbrace 1 \rbrace}+\left(\gamma_2\lambda-2\beta_2\gamma_1\right)\mu_{1,\lbrace 2 \rbrace}+\left(\lambda^2-4\beta_1\beta_2\right)\mu_{0,\lbrace 1,2 \rbrace}\right]1_A+\\
&-\frac{\lambda}{2}\left(\mu_{1,\emptyset}+\frac{\lambda}{2}\mu_{1,\lbrace 1,2 \rbrace}\right)G+\frac{\gamma_2}{2}\left(\mu_{1,\emptyset}+\frac{\lambda}{2}\mu_{1,\lbrace 1,2 \rbrace}\right)X_1+\\
&-\frac{\gamma_1}{2}\left(\mu_{1,\emptyset}+\frac{\lambda}{2}\mu_{1, \lbrace 1,2 \rbrace}\right)X_2+\left(\mu_{1,\emptyset}+\frac{\lambda}{2}\mu_{1,\lbrace 1,2 \rbrace}\right)GX_1X_2,
\end{split}
\end{equation}
with $\mu_{j,P} \in \Bbbk$ satisfying \eqref{sys:mu}. The values on even monomials are
\begin{equation}\label{eq:epsT1}
(\mathrm{Id}_A \otimes \varepsilon)\T(1_{E(2)})= 1_A,
\end{equation}
\begin{equation}\label{eq.epsTgx1}
(\mathrm{Id}_A \otimes \varepsilon)\T(gx_1)=\eta(\lbrace 1 \rbrace;1,\emptyset)G+\eta(\lbrace 1 \rbrace;0,\lbrace 1 \rbrace) X_1+\eta(\lbrace 1 \rbrace;0,\lbrace 2 \rbrace)X_2+\eta(\lbrace 1 \rbrace;1,\lbrace 1,2 \rbrace)GX_1X_2,
\end{equation}
\begin{equation}\label{eq:epsTgx2}
(\mathrm{Id}_A \otimes \varepsilon)\T(gx_2)=\eta(\lbrace 2 \rbrace;1,\emptyset) G+\eta(\lbrace 2 \rbrace;0,\lbrace 1 \rbrace) X_1+\eta(\lbrace 2 \rbrace;0,\lbrace 2 \rbrace) X_2+\eta(\lbrace 2 \rbrace;1,\lbrace 1,2 \rbrace)GX_1X_2,
\end{equation}
\begin{equation}\label{eq:epsTx1x2}
\begin{split}
(\mathrm{Id}_A \otimes \varepsilon)\T(x_1x_2)&=\eta(\lbrace 1, 2\rbrace;0,\emptyset)1_A-\frac{\lambda}{2}\eta(\lbrace 1,2 \rbrace;1,\lbrace 1,2 \rbrace)G+\frac{\gamma_2}{2}\eta(\lbrace 1,2 \rbrace;1,\lbrace 1,2 \rbrace)X_1+\\
&+\eta(\lbrace 1, 2\rbrace;1,\lbrace 1 \rbrace) GX_1-\frac{\gamma_1}{2}\eta(\lbrace 1,2 \rbrace;1,\lbrace 1,2 \rbrace)X_2+\eta(\lbrace 1, 2\rbrace;1,\lbrace 2 \rbrace) GX_2+\\
&+\eta(\lbrace 1, 2\rbrace;0,\lbrace 1,2 \rbrace)X_1X_2+\eta(\lbrace 1, 2\rbrace;1,\lbrace 1,2 \rbrace)GX_1X_2,
\end{split}
\end{equation}
with $\eta(P;k,Q) \in \Bbbk$ verifying \eqref{sys:eta2} for $i=1,2$ and \eqref{sys:eta3}. Our first goal is to show that if $\T$ sastisfies the h-property, then $\mu_{0,\lbrace 1,2 \rbrace}=0$. 
\begin{lemma}\label{lemma:useful1}
Let $\T:E(2) \rightarrow A \otimes E(2)^{\op}$ be a total integral satisfying the h-property. Then $\mu_{0,\lbrace 1,2 \rbrace}=0$.
\end{lemma}

\begin{proof}
Since $\T$ has the h-property, it verifies \eqref{eq:Tgg} and \eqref{eq:TgxP} for $P=\lbrace 1 \rbrace, \lbrace 2 \rbrace, \lbrace 1,2 \rbrace$. We assume that $\mu_{0,\lbrace 1,2 \rbrace} \neq 0$ and show that this assumption necessarily leads to a contradiction. With a bit of work (or the help of a software) one can check that the coefficients of $(\mathrm{Id}_A \otimes \varepsilon)\left(\T(g)\T(g)\right)$ relative to $X_1X_2$ and $GX_1X_2$ are respectively $\lambda \mu_{0,\lbrace 1,2 \rbrace}^2+2\mu_{0,\emptyset}\mu_{0,\lbrace 1,2 \rbrace}+\alpha\lambda\mu_{1,\lbrace 1,2 \rbrace}^2+2\alpha \mu_{1,\emptyset}\mu_{1,\lbrace 1,2 \rbrace}$ and $2 \left(\mu_{0,\emptyset}\mu_{1, \lbrace 1,2 \rbrace}+\mu_{1,\emptyset}\mu_{0, \lbrace 1,2 \rbrace}+\lambda\mu_{0, \lbrace 1,2 \rbrace}\mu_{1, \lbrace 1,2 \rbrace} \right)$, thanks to \eqref{sys:mu}. By \eqref{eq:Tgg}, they both must vanish, hence
\begin{equation}\label{eq:aux6}
\begin{cases}
\mu_{1,\emptyset}=-\lambda\mu_{1, \lbrace 1,2 \rbrace}-\frac{\mu_{0,\emptyset}\mu_{1, \lbrace 1,2 \rbrace}}{\mu_{0, \lbrace 1,2 \rbrace}},\\
\left(\alpha \mu_{1,\lbrace 1,2 \rbrace}^2-\mu_{0,\lbrace 1,2 \rbrace}^2 \right)\left(2 \mu_{0,\emptyset}+\lambda \mu_{0,\lbrace 1,2 \rbrace} \right)=0.
\end{cases}
\end{equation}
To proceed further, we show that $\alpha$ must be non-zero. If $\alpha=0$ we immediately get $\gamma_1=\gamma_2=0$ from \eqref{sys:mu} and $\mu_{0,\emptyset}=-\frac{\lambda}{2}\mu_{0,\lbrace 1,2 \rbrace}$, $\mu_{1,\emptyset}=-\frac{\lambda}{2}\mu_{1,\lbrace 1,2 \rbrace}$ thanks to \eqref{eq:aux6}. Then
\begin{equation}\label{eq:contradiction1}
(\mathrm{Id}_A \otimes \varepsilon)\left(\T(g)\T(g)\right)=(\lambda^2-4\beta_1\beta_2)\left[\frac{\mu_{0,\lbrace 1,2 \rbrace}^2}{4}1_A+\frac{\mu_{0,\lbrace 1,2 \rbrace}\mu_{1,\lbrace 1,2 \rbrace}}{2}G\right]
\end{equation}
and \eqref{eq:Tgg} implies $\mu_{1,\lbrace 1,2 \rbrace}=0$. Consequently
\[
(\mathrm{Id}_A \otimes \varepsilon)\left(\T(x_1)\T(g)\right)=-\frac{\mu_{1,\lbrace 2 \rbrace}}{\mu_{0,\lbrace 1,2 \rbrace}} G-X_1
\]
and \eqref{eq:TgxP} with $P=\lbrace 1 \rbrace$ implies $\eta(\lbrace 1 \rbrace; 1, \emptyset)=-\frac{\mu_{1,\lbrace 2 \rbrace}}{\mu_{0,\lbrace 1,2 \rbrace}}$, $\eta(\lbrace 1 \rbrace; 0, \lbrace 1 \rbrace)=-1$ and $\eta(\lbrace 1 \rbrace; 0, \lbrace 2 \rbrace)=\eta(\lbrace 1 \rbrace; 1, \lbrace 1,2 \rbrace)=0$, so that \eqref{sys:eta2} with $i=1$ yields $\beta_1=\lambda=0$. In view of \eqref{eq:contradiction1} this contradicts \eqref{eq:Tgg}, therefore $\alpha$ cannot be zero. Accordingly \eqref{sys:mu} reduces to $\mu_{1,\lbrace 1 \rbrace}=\frac{\gamma_2}{2 \alpha} \mu_{0,\lbrace 1,2 \rbrace}$ and $\mu_{1,\lbrace 2 \rbrace}=-\frac{\gamma_1}{2 \alpha} \mu_{0,\lbrace 1,2 \rbrace}$. Now we show that $2\mu_{0,\emptyset}+\lambda \mu_{0,\lbrace 1,2 \rbrace} \neq 0$. If $\mu_{0,\emptyset}=-\frac{\lambda}{2} \mu_{0,\lbrace 1,2 \rbrace}$, \eqref{eq:aux6} yields $\mu_{1,\emptyset}=-\frac{\lambda}{2} \mu_{1,\lbrace 1,2 \rbrace}$ and \eqref{eq:Tgg} boils down to
\[
\begin{cases}
\frac{\mu_{0,\lbrace 1,2 \rbrace}^2+\alpha\mu_{1,\lbrace 1,2 \rbrace}^2}{4\alpha}\left(\beta_2\gamma_1^2-\gamma_1\gamma_2\lambda +\beta_1 \gamma_2^2+\alpha\lambda^2-4\alpha \beta_1 \beta_2 \right)=1,\\
\frac{\mu_{0,\lbrace 1,2 \rbrace}\mu_{1,\lbrace 1,2 \rbrace}}{2 \alpha}\left(\beta_2\gamma_1^2-\gamma_1\gamma_2\lambda +\beta_1 \gamma_2^2+\alpha\lambda^2-4\alpha \beta_1 \beta_2 \right)=0,
\end{cases}
\]
i.e. $\mu_{1,\lbrace 1,2 \rbrace}=0$ and
\begin{equation}\label{eq:contradiction2}
\mu_{0,\lbrace 1,2 \rbrace}^2\left(\beta_2\gamma_1^2-\gamma_1\gamma_2\lambda +\beta_1 \gamma_2^2+\alpha\lambda^2-4\alpha \beta_1 \beta_2 \right)=4\alpha.
\end{equation}
Then $(\mathrm{Id}_A \otimes \varepsilon)\left(\T(x_1)\T(g)\right)=\frac{\gamma_1}{2\alpha}G-X_1 $ and \eqref{eq:TgxP} with $P=\lbrace 1 \rbrace$ implies $\eta(\lbrace 1 \rbrace; 1, \emptyset)=\frac{\gamma_1}{2\alpha}$, $\eta(\lbrace 1 \rbrace; 0, \lbrace 1 \rbrace)=-1$ and $\eta(\lbrace 1 \rbrace; 0, \lbrace 2 \rbrace)=\eta(\lbrace 1 \rbrace; 1, \lbrace 1,2 \rbrace)=0$, so that \eqref{sys:eta2} with $i=1$ yields $\gamma_1^2-4 \alpha \beta_1=0$, $\gamma_1\gamma_2-2 \alpha\lambda=0$. There follows
\[\beta_2\gamma_1^2-\gamma_1\gamma_2\lambda +\beta_1 \gamma_2^2+\alpha\lambda^2-4\alpha \beta_1 \beta_2=\beta_2\gamma_1^2-\frac{\gamma_1^2\gamma_2^2}{2\alpha} +\frac{\gamma_1^2\gamma_2^2}{4\alpha}+\frac{\gamma_1^2\gamma_2^2}{4\alpha}-\beta_2\gamma_1^2=0
\]
and \eqref{eq:contradiction2} becomes $0=4\alpha \neq 0$, which is a contradiction. We deduce that $2\mu_{0,\emptyset}+\lambda \mu_{0,\lbrace 1,2 \rbrace} \neq 0$ and \eqref{eq:aux6} forces $\alpha \mu_{1,\lbrace 1,2 \rbrace}^2=\mu_{0,\lbrace 1,2 \rbrace}^2$. If $\alpha \notin (\Bbbk^{\times})^2$ the proof has reached its end, otherwise we calculate that
\begin{align*}
&(\mathrm{Id}_A \otimes \varepsilon)\left(\T(g)\T(g)\right)\\
&=\left[ \frac{\beta_2\gamma_1^2-\gamma_1\gamma_2\lambda+\beta_1\gamma_2^2+\alpha \lambda^2 -4\alpha \beta_1 \beta_2}{2}\mu_{1,\lbrace 1,2 \rbrace}^2+2 \left(\mu_{0,\emptyset}+\frac{\lambda}{2} \mu_{0,\lbrace 1,2 \rbrace}\right)^2\right]1_A+\\
&+2\frac{\mu_{1,\lbrace 1,2 \rbrace}}{\mu_{0,\lbrace 1,2 \rbrace}}\left[\frac{\beta_2\gamma_1^2-\gamma_1\gamma_2\lambda +\beta_1\gamma_2^2+\alpha \lambda^2-4\alpha\beta_1\beta_2}{4}\mu_{1,\lbrace 1,2 \rbrace}^2-\left(\mu_{0,\emptyset} +\frac{\lambda}{2} \mu_{0,\lbrace 1,2 \rbrace} \right)^2\right]G,
\end{align*}
which, given \eqref{eq:Tgg}, yields
\begin{equation}\label{eq:aux7}
\begin{cases}
 (\beta_2\gamma_1^2-\gamma_1\gamma_2\lambda+\beta_1\gamma_2^2+\alpha \lambda^2 -4\alpha \beta_1 \beta_2)\mu_{1,\lbrace 1,2 \rbrace}^2=1,\\
\left(\mu_{0,\emptyset} +\frac{\lambda}{2} \mu_{0,\lbrace 1,2 \rbrace} \right)^2=\frac{1}{4}.
\end{cases}
\end{equation}
Then we have a solution only if $\beta_2\gamma_1^2-\gamma_1\gamma_2\lambda+\beta_1\gamma_2^2+\alpha \lambda^2 -4\alpha \beta_1 \beta_2 \in (\Bbbk^{\times})^2$ and in this case
\begin{align*}
(\mathrm{Id}_A \otimes \varepsilon)\left(\T(x_1)\T(g)\right)&=\frac{(2\beta_1\gamma_2-\gamma_1\lambda)(2\mu_{0,\emptyset}+\lambda \mu_{0,\lbrace 1,2 \rbrace})}{4}\mu_{1,\lbrace 1,2 \rbrace}1_A +\frac{\gamma_1}{4\alpha}G-\frac{1}{2}X_1+\\
&-\frac{(\gamma_1\gamma_2-2\alpha\lambda)(2\mu_{0,\emptyset}+\lambda \mu_{0,\lbrace 1,2 \rbrace})}{4\alpha}\mu_{1,\lbrace 1,2 \rbrace}GX_1+\\
&+\frac{(\gamma_1^2-4\alpha\beta_1)(2\mu_{0,\emptyset}+\lambda \mu_{0,\lbrace 1,2 \rbrace})}{4\alpha}\mu_{1,\lbrace 1,2 \rbrace}GX_2.
\end{align*}
Equality \eqref{eq:TgxP} with $P=\lbrace 1 \rbrace$ implies
\[
\begin{cases}
(2\beta_1\gamma_2-\gamma_1\lambda)(2\mu_{0,\emptyset}+\lambda \mu_{0,\lbrace 1,2 \rbrace})\mu_{1,\lbrace 1,2 \rbrace}=0,\\
(\gamma_1\gamma_2-2\alpha\lambda)(2\mu_{0,\emptyset}+\lambda \mu_{0,\lbrace 1,2 \rbrace})\mu_{1,\lbrace 1,2 \rbrace}=0,\\
(\gamma_1^2-4\alpha\beta_1)(2\mu_{0,\emptyset}+\lambda \mu_{0,\lbrace 1,2 \rbrace})\mu_{1,\lbrace 1,2 \rbrace}=0.
\end{cases}
\]
and, since $2\mu_{0,\emptyset}+\lambda \mu_{0,\lbrace 1,2 \rbrace} \neq 0$ and $\mu_{1,\lbrace 1,2 \rbrace}^2=\frac{\mu_{0,\lbrace 1,2 \rbrace}^2}{\alpha} \neq 0$, we get $\lambda=\frac{\gamma_1\gamma_2}{2 \alpha}$ and $\beta_1=\frac{\gamma_1^2}{4\alpha}$.
Thus the first equality of \eqref{eq:aux7} reads
\[
1=(\beta_2\gamma_1^2-\gamma_1\gamma_2\lambda+\beta_1\gamma_2^2+\alpha \lambda^2 -4\alpha \beta_1 \beta_2)\mu_{1,\lbrace 1,2 \rbrace}^2= (\beta_2\gamma_1^2-\gamma_1^2\beta_2)\mu_{1,\lbrace 1,2 \rbrace}^2=0,\\
\]
which is clearly a contradiction. We can conclude that if $\T: E(2) \rightarrow A \otimes E(2)^{\op}$ is a total integral satisfying the h-property, the coefficient $\mu_{0,\lbrace 1,2 \rbrace}$ must vanish.
\end{proof}
Now we can explicitly describe the integrals on $(E(2),\psi)$ that verify the h-property.
\begin{theorem}
An integral $\T:E(2) \rightarrow A \otimes E(2)^{\op}$ on $(E(2),\psi)$ satisfies the h-property if and only if either
\begin{enumerate}[1)]
\itemsep1em
\item $\T(h)=1_A \otimes S^{-1}(h)$ for every $h \in E(2)$ (the canonical total integral $\underline{\T}$).\label{hprop_case1}
\item $\T(g^jx_P)=(-1)^{j+P}1_A \otimes S^{-1}(g^jx_P)$ for every $g^jx_P$.\label{hprop_case2}
\item $\begin{aligned}[t]
&\T(1_{E(2)})=1_A \otimes 1_{E(2)}, &&\T(g)=\textbf{A}(G \otimes 1_{E(2)}),\\
&\T(x_1)=\textbf{A}\left[ -\frac{\gamma_1}{2}(1_A \otimes 1_{E(2)})+GX_1 \otimes 1_{E(2)}\right], &&\T(gx_1)=\frac{\gamma_1}{2\alpha} (G \otimes 1_{E(2)})-X_1 \otimes 1_{E(2)},\\
&\T(x_2)=\textbf{A}\left[-\frac{\gamma_2}{2}(1_A \otimes 1_{E(2)})+GX_2 \otimes 1_{E(2)}\right], &&\T(gx_2)=\frac{\gamma_2}{2\alpha} (G \otimes 1_{E(2)})-X_2 \otimes 1_{E(2)},\\
\end{aligned}$\\[1.2ex]
$\begin{aligned}[t]
&\T(x_1x_2)=-\frac{\gamma_1\gamma_2}{4\alpha}(1_A \otimes 1_{E(2)})+\frac{\gamma_2}{2\alpha} (GX_1 \otimes 1_{E(2)})-\frac{\gamma_1}{2\alpha}(GX_2 \otimes 1_{E(2)})+X_1X_2 \otimes 1_{E(2)},\\
&\T(gx_1x_2)=\textbf{A}\left[-\frac{\gamma_1\gamma_2}{4\alpha}(G \otimes 1_{E(2)})+\frac{\gamma_2}{2}(X_1 \otimes 1_{E(2)})-\frac{\gamma_1}{2}(X_2 \otimes 1_{E(2)})+GX_1X_2 \otimes 1_{E(2)}\right],
\end{aligned}$\\[1.2ex]
with $\textbf{A} \in \Bbbk$ such that $\alpha\textbf{A}^2=1$, provided $\alpha \in (\Bbbk^{\times})^2$, $\beta_1=\frac{\gamma_1^2}{4\alpha}$, $\beta_2=\frac{\gamma_2^2}{4\alpha}$, $\lambda=\frac{\gamma_1\gamma_2}{2\alpha}$. \label{hprop_case3}
\item $\T(1_{E(2)})= 1_A \otimes 1_{E(2)}$, $\T(gx_1)=1_A \otimes x_1$, $\T(gx_2)=1_A \otimes x_2$, $\T(x_1x_2)=1_A \otimes x_1x_2$,
\begin{align*}
\T(g)&=\textbf{A}\bigg(\frac{\gamma_2}{2}(1_A \otimes x_1)-\frac{\gamma_1}{2}(1_A \otimes x_2)-\frac{\lambda}{2}(G \otimes 1_{E(2)})+G \otimes x_1x_2+\frac{\gamma_2}{2}(X_1 \otimes 1_{E(2)})+\\
&+GX_1 \otimes x_2-\frac{\gamma_1}{2}(X_2 \otimes 1_{E(2)})-GX_2 \otimes x_1+GX_1X_2 \otimes 1_{E(2)}\bigg),\\[0.5ex]
\T(x_1)&=\textbf{A}\bigg(-\frac{\gamma_1}{2}(1_A \otimes x_1x_2)+\frac{\lambda}{2}(G \otimes x_1)-\frac{\gamma_2}{2}(X_1 \otimes x_1)+GX_1 \otimes x_1x_2+\\
&+\frac{\gamma_1}{2}(X_2 \otimes x_1)-GX_1X_2 \otimes x_1\bigg),\\[1ex]
\T(x_2)&=\textbf{A}\bigg(-\frac{\gamma_2}{2}(1_A \otimes x_1x_2)+\frac{\lambda}{2}(G \otimes x_2)-\frac{\gamma_2}{2}(X_1 \otimes x_2)+\frac{\gamma_1}{2}(X_2 \otimes x_2)+\\
&+GX_2 \otimes x_1x_2-GX_1X_2 \otimes x_2\bigg),\\[0.5ex]
\T(gx_1x_2)&=\textbf{A}\bigg(\frac{\gamma_2\lambda}{4}(1_A \otimes x_1)-\frac{\lambda}{2}(G \otimes x_1x_2)+\frac{\gamma_2}{2}(X_1 \otimes x_1x_2)-\frac{\gamma_1}{2}(X_2 \otimes x_1x_2)+\\
&+\frac{\lambda}{2}(GX_2 \otimes x_1)+\frac{\gamma_2}{4}(X_1X_2 \otimes x_1)+GX_1X_2 \otimes x_1x_2\bigg),
\end{align*}
with $\left(\beta_2\gamma_1^2-\gamma_1\gamma_2\lambda+\beta_1\gamma_2^2+\alpha\lambda^2-4\alpha\beta_1\beta_2\right)\textbf{A}^2=4$, provided $\beta_2\gamma_1^2-\gamma_1\gamma_2\lambda+\beta_1\gamma_2^2+\alpha\lambda^2-4\alpha\beta_1\beta_2\in (\Bbbk^{\times})^2$.
\label{hprop_case4}
\end{enumerate}
\end{theorem}
\begin{proof}
Consider a total integral satisfying the h-property and such that $\mu_{1,\lbrace 1, 2 \rbrace}=0$. Then
\[
(\mathrm{Id}_A \otimes \varepsilon)\left(\T(g)\T(g)\right)=(\mu_{0,\emptyset}^2+\alpha\mu_{1,\emptyset}^2)1_A+2\mu_{0,\emptyset}\mu_{1,\emptyset}G +2\mu_{0,\emptyset}\mu_{1,\lbrace 1 \rbrace}GX_1+2\mu_{0,\emptyset}\mu_{1,\lbrace 2 \rbrace}GX_2.
\]
If $\mu_{0,\emptyset} \neq 0$, \eqref{eq:Tgg} implies $\mu_{1,\emptyset}=\mu_{1,\lbrace 1 \rbrace}=\mu_{1,\lbrace 2 \rbrace}=0$ and $\mu_{0,\emptyset}=\pm 1$. Therefore
\[(\mathrm{Id}_A \otimes \varepsilon)\T(gx_P)\overset{\eqref{eq:TgxP}}{=}\pm(\mathrm{Id}_A \otimes \varepsilon)\T(x_P)=0\]
for $P=\lbrace 1 \rbrace, \lbrace 2 \rbrace, \lbrace 1,2 \rbrace$ and using \eqref{eq:epsTg}--\eqref{eq:epsTx1x2} we find \ref{hprop_case1} and \ref{hprop_case2}.
On the other hand, if $\mu_{0,\emptyset}=0$, \eqref{eq:Tgg} forces $\alpha \mu_{1,\emptyset}^2=1$ which has a solution only if $\alpha \in (\Bbbk^{\times})^2$. In this case \eqref{sys:mu} reduces to $\mu_{1,\lbrace 1 \rbrace}=\mu_{1, \lbrace 2 \rbrace}=0$ and from
\[(\mathrm{Id}_A \otimes \varepsilon)\T(gx_P)\overset{\eqref{eq:TgxP}}{=}\mu_{1,\emptyset}(\mathrm{Id}_A \otimes \varepsilon)\T(x_P)G\]
for $P=\lbrace 1 \rbrace, \lbrace 2 \rbrace, \lbrace 1,2 \rbrace$ and \eqref{sys:eta2} we obtain \ref{hprop_case3}. 

Now consider a total integral satisfying the h-property and such that $\mu_{1,\lbrace 1, 2 \rbrace} \neq 0$. One can check that the coefficient of $(\mathrm{Id}_A \otimes \varepsilon)\left(\T(g)\T(g)\right)$ relative to $GX_1X_2$ is $2\mu_{0,\emptyset}\mu_{1,\lbrace 1, 2 \rbrace}$, so that \eqref{eq:Tgg} yields $\mu_{0,\emptyset}=0$. At this point
\begin{align*}
(\mathrm{Id}_A \otimes \varepsilon)\left(\T(g)\T(g)\right)&=\left(\frac{\beta_2\gamma_1^2-\gamma_1\gamma_2\lambda+\beta_1\gamma_2^2-4\alpha\beta_1\beta_2}{4}\mu_{1,\lbrace 1,2 \rbrace}^2 + \alpha \mu_{1,\emptyset}^2  \right)1_A+\\
&+\left(\beta_1\gamma_2 \mu_{1,\lbrace 1 \rbrace}- \beta_2 \gamma_1 \mu_{1,\lbrace 2 \rbrace}+ \gamma_2 \lambda \mu_{1,\lbrace 2 \rbrace}\right)\mu_{1,\lbrace 1,2 \rbrace} G+\\
&+\frac{\gamma_2(2\mu_{1,\emptyset}+\lambda\mu_{1,\lbrace 1,2 \rbrace})}{2}\mu_{1,\lbrace 1,2 \rbrace}GX_1-\frac{\gamma_1(2\mu_{1,\emptyset}+\lambda\mu_{1,\lbrace 1,2 \rbrace})}{2}\mu_{1,\lbrace 1,2 \rbrace}GX_2+\\
&+\alpha(2\mu_{1,\emptyset}+\lambda\mu_{1,\lbrace 1,2 \rbrace})\mu_{1,\lbrace 1,2 \rbrace}X_1X_2,
\end{align*}
in view of \eqref{sys:mu}. If $2\mu_{1,\emptyset}+\lambda\mu_{1,\lbrace 1,2 \rbrace} \neq 0$, condition \eqref{eq:Tgg} forces $\alpha=\gamma_1=\gamma_2$ and the above quantity vanishes generating a contradiction. Therefore $2\mu_{1,\emptyset}+\lambda\mu_{1,\lbrace 1,2 \rbrace}=0$ and \eqref{eq:Tgg} together with \eqref{sys:mu}
become equivalent to
\begin{equation}\label{eq:aux8}
\begin{cases}
\left(\beta_2\gamma_1^2-\gamma_1\gamma_2\lambda+\beta_1\gamma_2^2+\alpha\lambda^2-4\alpha\beta_1\beta_2\right)\mu_{1,\lbrace 1,2 \rbrace}^2=4,\\
\beta_1\gamma_2 \mu_{1,\lbrace 1 \rbrace}+(- \beta_2 \gamma_1+ \gamma_2 \lambda )\mu_{1,\lbrace 2 \rbrace}=0,\\
\gamma_1 \mu_{1, \lbrace 1 \rbrace}+\gamma_2 \mu_{1,\lbrace 2 \rbrace}=0,\\
\alpha \mu_{1,\lbrace 1 \rbrace}=0,\\
\alpha \mu_{1,\lbrace 2 \rbrace}=0.
\end{cases}
\end{equation}
It is clear that if $\alpha \neq 0$ then $\mu_{1,\lbrace 1 \rbrace}=\mu_{1,\lbrace 2 \rbrace}=0$. On the other hand, if $\alpha=0$ system \eqref{eq:aux8} rewrites as
\begin{equation*}
\begin{cases}
\left(\beta_2\gamma_1^2-\gamma_1\gamma_2\lambda+\beta_1\gamma_2^2\right)\mu_{1,\lbrace 1,2 \rbrace}^2=4,\\[1ex]
\begin{pmatrix}
\beta_1\gamma_2 & -\beta_2 \gamma_1+\gamma_2 \lambda\\
\gamma_1  & \gamma_2
\end{pmatrix}
\begin{pmatrix}
\mu_{1,\lbrace 1 \rbrace}\\
\mu_{1,\lbrace 2 \rbrace}
\end{pmatrix}
=
\begin{pmatrix}
0\\
0
\end{pmatrix},
\end{cases}
\end{equation*}
which clearly yields $\mu_{1,\lbrace 1 \rbrace}=\mu_{1,\lbrace 2 \rbrace}=0$. Therefore the coefficients $\mu_{1,\lbrace 1 \rbrace}$ and $\mu_{1,\lbrace 2 \rbrace}$ must vanish in any case. Then we obtain 
$\left(\beta_2\gamma_1^2-\gamma_1\gamma_2\lambda+\beta_1\gamma_2^2+\alpha\lambda^2-4\alpha\beta_1\beta_2\right)\mu_{1,\lbrace 1,2 \rbrace}^2=4$ from \eqref{eq:aux8} and by employing \eqref{eq:TgxP} for $P=\lbrace 1 \rbrace, \lbrace 2 \rbrace, \lbrace 1,2 \rbrace$ we get \ref{hprop_case4}.
\end{proof}

\subsection{rt-Separability}

\cite[Thm. 6.1]{mt} and its generalization \cite[Thm. 5.7]{FR3} were obtained with the additional hypothesis that the normalized Casimir morphism $B:H \otimes H \rightarrow A \otimes H^{\op}$ realizing separability be of the form $B(h \otimes h')=B^A(h \otimes h') \otimes 1_H$ for all $h,h' \in H$, in order to simplify calculations. A cowreath endowed with a morphism of this type was named rt-separable in \cite[Def. 2.11]{FR3} to distinguish it from classic separable ones. In view of Theorem \ref{thm:maincas}, a cowreath $(A \otimes H^{\op}, H, \psi)$ is rt-separable if and only if there is a total integral $\T:H \rightarrow A \otimes H^{\op}$ on $(H,\psi)$ of the form $\T(h)=\T^A(h) \otimes 1_H$ for every $h \in H$. We show how to retrieve \cite[Thm. 5.7]{FR3} from Theorem \ref{thm:mainint}, after rephrasing the statement of the former using integrals in place of Casimir morphisms. The image of $\T^A$ on odd monomials is rather complicated and makes use of notations introduced in \cite[Sec. 3]{FR3} that, for sake of brevity, we do not recall here.

\begin{theorem}\cite[Thm. 5.7]{FR3}
Let $H=E(n)$ and $A=Cl(\alpha,\beta_i, \gamma_i, \lambda_{ij})$ be a Clifford algebra endowed with the canonical $H$-coaction \eqref{rhoGX}. The cowreath $(A\otimes H^{\op},H,\psi )$ is rt-separable if, and only if,
\begin{enumerate}
\item $A=Cl(0,0,0,0)$ \\
or
\item $A=Cl\left(\alpha,\frac{\gamma_i^2}{4\alpha}, \gamma_i, \frac{\gamma_i\gamma_j}{2 \alpha}\right)$ with $\alpha \neq 0$.  
\end{enumerate}
Let $\overline{Q}=Q \sqcup \lbrace 0 \rbrace$, $X_0=G$, $\lambda_{0i}=\gamma_i$ for $i=1, \ldots, n$. The total integral $\T:H \rightarrow A \otimes H^{\op}$ is given by $\T(h)=\T^A(h) \otimes 1_H$ where
\begin{equation}\label{eq:even_int}
\T^A(g^{|Q|}x_Q)=\sum_{\underset{|R| \equiv_2 |Q|}{R \subseteq Q}}(-1)^{S(Q \setminus R, Q)+|Q|}\eta_{Q \setminus R} X_R+\sum_{\underset{|R| \not\equiv_2 |Q|}{R \subseteq Q}}(-1)^{S(Q \setminus R, Q)+|Q|}\eta_{Q \setminus R} GX_R,
\end{equation}
\begin{equation}\label{eq:odd_int}
\T^A(g^{|Q|+1}x_Q)=\frac{\mu}{2^{\left\lfloor \frac{|\overline{Q}|}{2}\right\rfloor}}\sum_{\underset{|R| \equiv_2 |\overline{Q}|}{R \subseteq \overline{Q}}}(-1)^{S(R,\overline{Q})+\frac{|\overline{Q}|-|R|}{2}}2^{\left\lfloor \frac{|R|}{2}\right\rfloor}\sum_{\mathcal{P} \in \pi_2(\overline{Q}\setminus R)}\sgn(\mathcal{P})\Lambda(\mathcal{P})X_R,
\end{equation}
for every $Q \subseteq \lbrace 1, \ldots, n \rbrace$, where $\mu \in \Bbbk$ and $\eta_S \in \Bbbk$ for every $S \subseteq \lbrace 1, \ldots, n \rbrace$, $\eta_{\emptyset}=1$. If $\alpha \neq 0$, then $\eta_T =\sum_{i \in T}(-1)^{S(\lbrace i \rbrace, T)+1}\eta_{T \setminus  \lbrace i \rbrace}\frac{\gamma_i}{2\alpha}$ for every $T \subseteq \lbrace 1, \ldots, n \rbrace$ with $|T|$ odd.
\end{theorem}

\begin{proof}
Assume $(A \otimes H^{\op},H,\psi)$ is rt-separable via the total integral $\T:H \rightarrow A \otimes H^{\op}$. Then $\T(g)=\T^A(g) \otimes 1_H$ and by Theorem \ref{thm:mainint} we must have
\[\T(g)=\sum_{j,P}\sum_{F \subseteq P} (-1)^{S(P \setminus F,P)+(j+|P|+1)|F|} \mu_{j,P}G^jX_{P \setminus F} \otimes g^{j+|P|+1}x_F.\]
As a consequence $\mu_{j,P}=0$ if $(j,P)\neq (1,\emptyset)$ and $\T^A(g)=\mu_{1,\emptyset}G$, which is \eqref{eq:odd_int} for $Q=\emptyset$ once we put $\mu=\mu_{1,\emptyset}$. Notice that \eqref{eq:mu-0P} and \eqref{eq:mu-1P} are trivially satisfied for every $P \subseteq \lbrace 1, \ldots, n \rbrace$. Now observe that Theorem \ref{thm:mainint} also implies
\[
\T^A(g^{|Q|}x_{Q}x_r)= \frac{1}{2}\left[\T^A(g^{|Q|+1}x_Q)X_r+(-1)^{|Q|+1}X_r \T^A(g^{|Q|+1}x_Q)\right]
\]
for every $Q \subseteq \lbrace 1, \ldots, n \rbrace \ni r$ such that $r >j$ for every $j \in Q$. This is enough to conclude that \eqref{eq:odd_int} must hold for every $Q \subseteq \lbrace 1, \ldots, n \rbrace$, as one can prove by following the same steps performed in the first part of the proof of \cite[Prop. 5.5]{FR3}. Next, Theorem \ref{thm:mainint} also implies that
\[
\T(g^{|P|}x_P)=\sum_{F \subsetneq P}(-1)^{S(F,P)+|P|+|F|+1}\T(g^{|F|}x_{F}) (1_A \otimes x_{P \setminus F})+\mathfrak{U}_P
\]
 for every $P \subseteq \lbrace 1, \ldots, n\rbrace$, where
\begin{equation}\label{eq:aux9}
\mathfrak{U}_P=\sum_{k,Q}\sum_{F \subseteq Q}  \eta(P;k,Q)(-1)^{S(Q \setminus F,Q)+(|P|+|Q|+k)|F|} G^k X_{Q \setminus F} \otimes g^{|P|+|Q|+k}x_F
\end{equation}
and $\eta(P;k,Q) \in \Bbbk$ satisfy \eqref{eq:etaG} and \eqref{eq:eta1.1}-\eqref{eq:eta0.2} for $i=1, \ldots, n$ and every $Q \not \ni i$. Since
\begin{equation*}
\mathfrak{U}_P\overset{\eqref{eq:T-even}}{=}\sum_{F \subseteq P}(-1)^{S(F,P)+|P|+|F|}\T(g^{|F|}x_{F}) (1_A \otimes x_{P \setminus F})
\end{equation*}
and $\T$ is of the form $\T=\T^A \otimes u_H$, it becomes clear that each non-zero term in the RHS of \eqref{eq:aux9} must be of the form $\mathfrak{a} \otimes x_F$ with $\mathfrak{a}\in A$ and $F \subseteq P$. Then $\eta(P;k,Q)=0$ if $|P|+k+|Q|$ is odd or $Q \not\subseteq P$. Now we can calculate
\begin{align*}
&\T(g^{|P|}x_P)\\
\overset{\eqref{eq41}}&{=}\sum_{R \subseteq P} (-1)^{S(P \setminus R,P)+(|P|-|R|)|R|}\mathfrak{U}_R(1_A \otimes x_{P \setminus R})\\
&=\sum_{R \subseteq P} (-1)^{S(R,P)}\sum_{\underset{|Q| \equiv_2 |R|}{Q \subseteq R}}\sum_{F \subseteq Q}  \eta(R;0,Q)(-1)^{S(Q \setminus F,Q)} X_{Q \setminus F} \otimes x_Fx_{P \setminus R}+\\
&+\sum_{R \subseteq P} (-1)^{S(R,P)}\sum_{\underset{|Q| \not\equiv_2 |R|}{Q \subseteq R}}\sum_{F \subseteq Q}  \eta(R;1,Q)(-1)^{S(Q \setminus F,Q)} GX_{Q \setminus F} \otimes x_Fx_{P \setminus R}\\
&=\sum_{T \subseteq P} (-1)^{S(P \setminus T,P)}\sum_{\underset{|Q| \equiv_2 |P|-|T|}{Q \subseteq P \setminus T}}\sum_{F \subseteq Q}  \eta(P \setminus T;0,Q)(-1)^{S(Q \setminus F,Q)+S(F, F \sqcup T)} X_{Q \setminus F} \otimes x_{F \sqcup T}+\\
&+\sum_{T \subseteq P} (-1)^{S(P \setminus T,P)}\sum_{\underset{|Q| \not\equiv_2 |P|-|T|}{Q \subseteq P \setminus T}}\sum_{F \subseteq Q}  \eta(P \setminus T;1,Q)(-1)^{S(Q \setminus F,Q)+S(F, F \sqcup T)} GX_{Q \setminus F} \otimes x_{F \sqcup T}
\end{align*}
where $(-1)^{S(P \setminus R,P)+(|P|-|R|)|R|}=(-1)^{S(R,P)}$ follows from \cite[Lem. 3.6]{FR3}.
If we put $U:=F \sqcup T$, $V:=Q \sqcup T$, then we find
\begin{align*}
&\T(g^{|P|}x_P)\\
&=\sum_{T \subseteq P}\sum_{\underset{|V| \equiv_2 |P|}{T \subseteq V \subseteq P}}\sum_{T \subseteq U \subseteq V}  \eta(P \setminus T;0,V \setminus T)(-1)^{S(P \setminus T,P)+S(V \setminus U,V \setminus T)+S(U \setminus T,U)} X_{V \setminus U} \otimes x_U+\\
&+\sum_{T \subseteq P}\sum_{\underset{|V| \not\equiv_2 |P|}{T \subseteq V  \subseteq P }}\sum_{T \subseteq U \subseteq V}  \eta(P \setminus T;1,V \setminus T)(-1)^{S(P \setminus T,P)+S(V \setminus U,V \setminus T)+S(U \setminus T, U)} GX_{V \setminus U} \otimes x_U.
\end{align*}
and since $\T=\T^A \otimes u_H$ we get
\begin{equation}\label{eq:T-even2}
\T^A(g^{|P|}x_P)=\sum_{\underset{|V| \equiv_2 |P|}{V \subseteq P}}\eta(P;0,V) X_V +\sum_{\underset{|V| \not\equiv_2 |P|}{V  \subseteq P }}\eta(P;1,V)GX_V
\end{equation}
and
\begin{align*}
\sum_{T \subseteq U}\eta(P \setminus T;k,V \setminus T)(-1)^{S(P \setminus T,P)+S(V \setminus U,V \setminus T)+S(U \setminus T,U)}=0
\end{align*}
for each $\emptyset \neq U \subseteq V \subseteq P$, $|V|+k \equiv_2 |P|$ and $k=0,1$.
Notice that we can simply consider the case when $U=V$ to avoid many redundancies. Therefore
\[
\sum_{T \subseteq V}\eta(P \setminus T;k,V \setminus T)(-1)^{S(P \setminus T,P)+S(V \setminus T,V)}=0
\]
for each $\emptyset \neq V \subseteq P$, $|V|+k \equiv_2 |P|$ and $k=0,1$. With the help of the above, one can prove  by induction on $|P|$ that 
\begin{equation}\label{eq:rt-sepk}
\eta(P;k,V)=(-1)^{S(P \setminus V,P)+|V|}\eta(P \setminus V; k, \emptyset) \quad \textrm{for every } V \subseteq P, \ |V|+k \equiv_2 |P|, \ k=0,1
\end{equation}
hence we can rewrite \eqref{eq:T-even2} as
\begin{equation*}
\T^A(g^{|P|}x_P)=\sum_{\underset{|V| \equiv_2 |P|}{V \subseteq P}}(-1)^{S(P \setminus V,P)+|P|}\eta_{P \setminus V} X_V +\sum_{\underset{|V| \not\equiv_2 |P|}{V  \subseteq P }}(-1)^{S(P \setminus V,P)+|P|}\eta_{P \setminus V}GX_V,
\end{equation*}
where $\eta_{P\setminus V}$ is equal to $\eta(P \setminus V; 0, \emptyset)$ when $|P| \equiv_2 |V|$ and to $-\eta(P \setminus V; 1, \emptyset)$ otherwise. Thus we have shown that \eqref{eq:even_int} holds true. Moreover $\eta_{\emptyset}=\eta(\emptyset;0,\emptyset)=1$ since $\T$ is total (see Corollary \ref{cor:totint}).

To conclude the proof we need to show that $A$ can only be $Cl(0,0,0,0)$ or $Cl\left(\alpha,\frac{\gamma_i^2}{4\alpha}, \gamma_i, \frac{\gamma_i\gamma_j}{2 \alpha}\right)$ when $\alpha \neq 0$, and that in the second case $\eta_T =\sum_{i \in T}(-1)^{S(\lbrace i \rbrace, T)+1}\frac{\gamma_i}{2\alpha}\eta_{T \setminus  \lbrace i \rbrace}$ for every $T \subseteq \lbrace 1, \ldots, n \rbrace$ with $|T|$ odd. To this aim we consider \eqref{eq:etaG} with $k=1$, $P=\textbf{n}:=\lbrace 1, \ldots, n \rbrace$ and $Q=\textbf{n} \setminus \lbrace i \rbrace$ for $i \in \textbf{n}$
\[
(-1)^{n}2\alpha \eta(\textbf{n};1,\textbf{n} \setminus \lbrace i \rbrace)=  (-1)^{S(\textbf{n} \setminus \lbrace i \rbrace,\textbf{n})} \gamma_i\eta(\textbf{n};0,\textbf{n}).
\]
Since $\eta(\textbf{n};0,\textbf{n})=(-1)^n$ by \eqref{eq:rt-sepk}, we immediately see that $\alpha=0$ implies $\gamma_i=0$ for every $i \in \textbf{n}$. Then \eqref{eq:eta0.2} with the same choice of $P$ and $Q$ gives $(-1)^{n+1}(-1)^{S(\textbf{n} \setminus \lbrace i \rbrace, \textbf{n})}2\beta_i\eta(\textbf{n};0,\textbf{n})=0$ and we find $\beta_i=0$ for every $i \in \textbf{n}$. Finally fix $i \in \textbf{n}$, $l \in \textbf{n} \setminus \lbrace i \rbrace$ and consider \eqref{eq:eta0.2} for
$P=\textbf{n}$ and $Q=\textbf{n} \setminus \lbrace l \rbrace$
\[
(-1)^{S(\lbrace l \rbrace,\textbf{n})+1}\lambda_{li}\eta(\textbf{n};0,\textbf{n})=0.
\]
It follows that $\lambda_{il}=0$ for every $i,l \in \textbf{n}$, $i<l$, so that \eqref{eq:etaG} and \eqref{eq:eta1.1}-\eqref{eq:eta0.2} are trivially satisfied and the cowreath $(A \otimes H^{\op},H,\psi)$ is rt-separable without further restrictions on the $\eta_T$'s.

On the other hand, assume $\alpha \neq 0$ and consider \eqref{eq:eta0.2} with $P=\textbf{n}$ and $Q=\textbf{n} \setminus \lbrace i \rbrace$ for some $i \in \textbf{n}$. We obtain
\[
(-1)^{n+1}(-1)^{S(\textbf{n} \setminus \lbrace i \rbrace,\textbf{n})}2\beta_i\eta(\textbf{n};0,\textbf{n})+\gamma_i\eta(\textbf{n};1,\textbf{n} \setminus \lbrace i \rbrace)=0,
\]
i.e. $2\beta_i-\gamma_i\eta(\lbrace i \rbrace;1,\emptyset)=0$, thanks to \eqref{eq:rt-sepk}. Similarly, equality \eqref{eq:etaG} with $k=1$ and the same $P$ and $Q$ yields $\eta(\lbrace i \rbrace;1,\emptyset)=\frac{\gamma_i}{2 \alpha}$, so that we must have $\beta_i=\frac{\gamma_i^2}{4\alpha}$. 
Finally consider $i,j \in \textbf{n}$, $i \neq j$ and set $P=\textbf{n}$ and $Q=\textbf{n} \setminus \lbrace i,j \rbrace$ in \eqref{eq:eta0.1} which gives $\lambda_{ji}=\frac{\gamma_i\gamma_j}{2 \alpha}$. 
Then we have shown that if the cowreath $(A \otimes H^{\op},H,\psi)$ is rt-separable and $\alpha \neq 0$, $A$ must be of type $Cl\left(\alpha,\frac{\gamma_i^2}{4\alpha}, \gamma_i, \frac{\gamma_i\gamma_j}{2 \alpha}\right)$. To conclude we observe that \eqref{eq:etaG} for $k=1$, $|P|$ odd and $Q=\emptyset$ reads 
\[
\eta(P;1,\emptyset)=\sum_{l \in P} (-1)^{S(\lbrace l \rbrace,P)}\frac{ \gamma_l}{2\alpha}\eta(P \setminus \lbrace l \rbrace; 0, \emptyset),
\]
i.e
\[
\eta_P=\sum_{l \in P} (-1)^{S(\lbrace l \rbrace,P)+1}\frac{ \gamma_l}{2\alpha}\eta_{P \setminus \lbrace l \rbrace}.
\]
One can check that \eqref{eq:etaG} and \eqref{eq:eta1.1}-\eqref{eq:eta0.2} do not produce any further equality. 
\end{proof}

\section{Frobenius cowreaths \texorpdfstring{$\left(A \otimes E(n)^{\op},E(n),\psi\right)$}{}}

Let $H=E(n)$ and $(A,\rho_A)$ be a finite-dimensional $H$-comodule algebra. We want to determine under what conditions the cowreath $(A \otimes H^{\op},H,\psi)$ is Frobenius, using Theorem \ref{thm:Frob-cow}. Once we consider the left integral $\Lambda=(1+g)x_1x_2 \cdots x_n \in H$, it is not hard to see that the distinguished grouplike element of $H^*$ is determined by $\mu(g)=(-1)^n$ and $\mu(g^jx_P)=0$ for $g^jx_P \in \mathbb{B}_H$ and $P \neq \emptyset$.
If $n$ is even, we have $\mu=\varepsilon_H$ (that is, $H$ is unimodular) and  $(\mathrm{Id}_A \otimes \mu)\rho_A=\mathrm{Id}_A$, which is clearly an inner automorphism of $A$. Therefore, when $n$ is even the cowreath $(A \otimes H^{\op}, H,\psi)$ is Frobenius, independently from the $H$-coaction involved.

Now suppose $n$ is odd and consider the following description of $H$-coactions on $A$, obtained in \cite{FR2}.

\begin{theorem}\cite[Thm. 4.4]{FR2}
Let $A$ be a finite dimensional algebra over a field $\Bbbk$ of characteristic $\mathrm{char}(\Bbbk) \neq 2$. Then an $E(n)$-comodule algebra structure on $A$ is given by:
\begin{equation*}
\rho(a)= \underset{P \subseteq \lbrace 1, \ldots, n \rbrace}{\sum}(-1)^{\frac{|P|(|P|+1)}{2}} \left[ d_P(a) \otimes \frac{x_P +(-1)^{|P|}gx_P}{2}+\varphi(d_P(a)) \otimes \frac{x_P +(-1)^{|P|+1}gx_P}{2} \right],
\end{equation*}
where
\begin{enumerate}
\item $\varphi$ is an automorphism  of $A$ of order $o(\varphi) \leq 2$ (i.e. an involution of $A$),
\item $d_{\emptyset}=\mathrm{Id}_A$ and $d_P=d_{i_1}d_{i_2}\cdots d_{i_{|P|}}$ is a composition of $\varphi$-derivations such that $d_i^2 \equiv 0$, $\varphi d_i = -d_i \varphi$ and $d_id_j=-d_jd_i$.
\end{enumerate}
\end{theorem}
It follows that
\[
(\mathrm{Id}_A \otimes \mu)\rho_A(a)=\frac{1+(-1)^n}{2}a+ \frac{1 +(-1)^{n+1}}{2}\varphi(a)=\varphi(a),
\]
thus, when $n$ is odd, the cowreath $(A \otimes H^{\op},H,\psi)$ is Frobenius if and only if the $H$-coaction $\rho_A$ is defined via an inner involution $\varphi$.

\begin{theorem}
Let $H=E(n)$ and $(A,\rho_A)$ be a finite-dimensional $H$-comodule algebra. The cowreath $(A \otimes H^{\op},H,\psi)$ is Frobenius if and only if either $n$ is even or the coaction $\rho_A$ corresponds to a tuple $(\varphi, d_1, \ldots, d_n)$ where $\varphi$ is an inner involution of $A$.
\end{theorem}

Clearly, if $A$ is a central simple algebra over $\Bbbk$, the cowreath $(A \otimes H^{\op},H,\psi)$ must be Frobenius (cf. Corollary \ref{cor:simpleA}).

\bigskip

When we consider $A=Cl(\alpha,\beta_i, \gamma_i, \lambda_{ij})$, a $2^{n+1}$-dimensional Clifford algebra, and we assume that $n$ is odd and $A$ is semisimple, \cite[Cor. 3.5]{FR2} and \cite[Thm. 3.7]{FR2} assure us that $A$ is a central simple algebra over $\Bbbk$, and thus the cowreath $(A \otimes H^{\op},H,\psi)$ is Frobenius. If the comodule algebra structure $\rho_A$ of $A$ is the canonical $H$-coaction \eqref{rhoGX}  we can prove that the semisimplicity of $A$ is also a necessary condition for the same cowreath to be Frobenius. It was first observed for the case $n=1$ in \cite[Example 6.8]{BT}.

\begin{proposition}
Let $A=Cl(\alpha,\beta_i, \gamma_i, \lambda_{ij})$ be a Clifford algebra endowed with the canonical $E(n)$-comodule algebra structure given by \eqref{rhoGX}. The cowreath $(A\otimes E(n)^{\op},E(n),\psi )$ is Frobenius
if and only if either $n$ is even or $n$ is odd and $A$ is simple.
\end{proposition}

\begin{proof}
It is enough to show that when $n$ is odd and $A$ is not semisimple, then any element $\A \in a$ satisfying \eqref{eq:propA} is not invertible. To this aim, we work with a basis of $A$ which can be, in principle, different from $\mathbb{B}_A$. Given that $\mathrm{char} (\Bbbk) \neq 2$, $A$ admits a set of orthogonal generators, i.e. elements $Y_0, Y_1, \ldots, Y_n \in A$ such that they generate $A$ as an algebra and
\[Y_iY_j+Y_jY_i=0, \quad \textrm{for every }i \neq j.\]
This can be seen as a consequence of \cite[Cor. I.2.4]{La} and \cite[Thm. V.1.8]{La}. Then, a basis of $A$ is given by the set of elements $\mathbb{B}'_A=\lbrace Y_P \ | \ P \subseteq \lbrace 0, 1 \ldots, n \rbrace \rbrace$. Here the meaning of the notation is the same used for $X_P$, that is, $Y_P:=Y_{i_1}Y_{i_2} \cdots Y_{i_r}$ where $P=\lbrace i_1 <i_2 < \ldots < i_r \rbrace$ and $Y_{\emptyset}=1_A$. Write $Y_i=\theta_{i0}G+\sum_{j=1}^n \theta_{ij}X_j$ and define $y_i:=\sum_{j=1}^n \theta_{ij}x_j$ for every $i=0,1, \ldots, n$. It follows that the coaction on each $Y_i$ is given by
\begin{equation}\label{eq:rhoY}
\rho(Y_i)=\sum_{j=0}^n \theta_{ij}\rho(X_j)=\theta_{i0}G \otimes g+\sum_{j=1}^n \theta_{ij}(X_j \otimes g +1_A \otimes x_j)=Y_i \otimes g +1_A \otimes y_i.
\end{equation}

Suppose that there exists an element $\A \in A$ such that \eqref{eq:propA} holds. Since the $Y_i$'s are generators for $A$, which is an $H$-comodule algebra, and $\mu$ is an algebra map, \eqref{eq:propA} is equivalent to
\[Y_i \A=\mu((Y_i)_1)\A (Y_i)_0\overset{\eqref{eq:rhoY}}{=}\mu(g)\A Y_i +\sum_{j=1}^n \theta_{ij}\mu(x_j)\A=-\A Y_i,\]
for every $i=0, 1, \ldots, n$, hence
\begin{equation}\label{eq:eiA}
Y_i \A=-\A Y_i \qquad \textrm{for every } i=0, 1, \ldots, n.
\end{equation}
Let us write $\A=\sum_P \eta_P Y_P$ using the basis $\mathbb{B}'_A$ of $A$. For every $P \subsetneq \lbrace 0,1,\ldots, n \rbrace$ with even cardinality we can always find an $Y_i$ such that $i \notin P$ and for such $i$ we have $Y_iY_P=Y_PY_i \neq 0$. Then \eqref{eq:eiA} forces $\eta_P=0$ for every $P \subsetneq \lbrace 0,1,\ldots, n \rbrace$ with even cardinality. Hence
\[\A=\zeta Y_0Y_1 \cdots Y_n+\sum_{|P| \ \textrm{odd}} \eta_P Y_P,\]
where $\zeta=\eta_{\lbrace 0,1, \ldots, n \rbrace}$.
Since $A$ is not semisimple, by \cite[Thm. 3.4]{FR2} some of the $Y_i$'s square to zero. We can assume, without losing of generality, that they are $Y_0, \ldots, Y_k$ for some $k$ with $0 \leq k \leq n$. For every $P \not\subseteq \lbrace 0, \ldots, k \rbrace$ with $|P|$ odd we can always find an $Y_i$ such that $i \in P$ and $i \geq k$, so that $Y_iY_P=Y_PY_i \neq 0$. As a consequence \eqref{eq:eiA} forces $\eta_P=0$ for every $P \not \subseteq \lbrace 0, \ldots, k \rbrace$ with $|P|$ odd. Then we can write
\[\A=\zeta Y_0Y_1 \cdots Y_n+\sum_{\underset{|P| \ \textrm{odd}}{P \subseteq \lbrace 0, \ldots, k \rbrace}} \eta_P Y_P.\]
If we multiply $\A$ by the product $Y_0\cdots Y_k$ we see that
\[\A Y_0\cdots Y_k=\zeta Y_0Y_1 \cdots Y_n(Y_0\cdots Y_k)+\sum_{\underset{|P| \ \textrm{odd}}{P \subseteq \lbrace 0, \ldots, k \rbrace}} \eta_P Y_P(Y_0\cdots Y_k)=0,\]
since each term contains the square of some $Y_i$ with $0 \leq i \leq k$. This shows that $\A$ is a zero divisor in $A$ and thus it cannot be invertible.
\end{proof}

\begin{corollary}
Let $A=Cl(\alpha,\beta_i, \gamma_i, \lambda_{ij})$ be a Clifford algebra endowed with the canonical $E(n)$-comodule algebra structure given by \eqref{rhoGX}. If $n$ is odd and $A$ is not semisimple, the cowreath $(A\otimes E(n)^{\op},E(n),\psi )$ is (h-)separable, but not Frobenius.
\end{corollary}

\begin{remark}
These examples, together with Theorem \ref{thm:maincas}, suggest that there is a class of non-Frobenius Galois cowreaths for which the existence of a total integral is equivalent to separability, and thus that the hypothesis of \cite[Thm. 5.4]{BT} can be relaxed. 
\end{remark}

\end{document}